\documentclass{article}
\usepackage{authblk}
\usepackage{a4wide}
\usepackage{amsthm}
\usepackage{graphicx}
\usepackage{amssymb}
\usepackage{amsmath}
\usepackage{ascmac}
\usepackage{setspace}
\usepackage{float}
\usepackage[dvips,usenames]{color}
\usepackage{colortbl}
\usepackage{algorithm}
\usepackage{algorithmic}
\usepackage{setspace}

\newtheorem{Theorem}[equation]{Theorem}
\newtheorem{Corollary}[equation]{Corollary}
\newtheorem{Lemma}[equation]{Lemma}

\theoremstyle{definition}
\newtheorem{Definition}[equation]{Definition}

\theoremstyle{remark}
\newtheorem{Remark}[equation]{Remark}
\numberwithin{equation}{section}

\DeclareMathOperator{\ev}{ev}
\DeclareMathOperator{\id}{id}

\DeclareMathOperator{\ad}{ad}

\DeclareMathOperator{\row}{row}
\DeclareMathOperator{\col}{col}

\newcommand{\plim}[1][]{\mathop{\varprojlim}\limits_{#1}}
\newcommand{\ve}{\varepsilon}

\allowdisplaybreaks

\begin{document}
\title{An Example of Homomorphisms from Guay's affine Yangians to non-rectangular $W$-algebras}
\author{Mamoru Ueda\thanks{mueda@ualberta.ca}}
\affil{Department of Mathematical and Statistical Sciences, University of Alberta, 11324 89 Ave NW, Edmonton, AB T6G 2J5, Canada}
\date{}
\maketitle

\begin{abstract}
We construct a non-trivial homomorphism from the Guay's affine Yangian associated with $\widehat{\mathfrak{sl}}(n)$ to the universal enveloping algebra of the $W$-algebra  associated with a Lie algebra $\mathfrak{gl}(m+n)$ and a nilpotent element of type $(2^{n},1^{m-n})$ for $m>n$. 
\end{abstract}

\section{Introduction}

A $W$-algebra $\mathcal{W}^k(\mathfrak{g},f)$ is a vertex algebra associated with a finite dimensional reductive Lie algebra $\mathfrak{g}$ and a nilpotent element $f$. It appeared in the study of two dimensional conformal field theories (\cite{Z}) and has been studied by both physicists and mathematicians since 1980's. We call a $W$-algebra associated with $\mathfrak{gl}(n)$ (resp. $\mathfrak{gl}(ln)$) and a principal nilpotent element (resp. a nilpotent element of type of $(l^n)$) a principal (resp. rectangular) $W$-algebra. The AGT (Alday-Gaiotto-Tachikawa) conjecture suggests that there exists a representation of the principal $W$-algebra of type $A$ on the equivariant homology space of the moduli space of $U(r)$-instantons. Schiffmann and Vasserot \cite{SV} gave this representation by using an action of the Yangian associated with $\widehat{\mathfrak{gl}}(1)$ on this equivariant homology space. 

In the rectangular case, the author \cite{U4} constructed a surjective homomorphism from the Guay's affine Yangian (\cite{Gu2} and \cite{Gu1}) to the universal enveloping algebra of a rectangular $W$-algebra of type $A$. The Guay's affine Yangian $Y_{\hbar,\ve}(\widehat{\mathfrak{sl}}(n))$ is a 2-parameter Yangian and is the deformation of the universal enveloping algebra of the central extension of $\mathfrak{sl}(n)[u^{\pm1},v]$. In \cite{KU}, Kodera and the author showed that this homomorphism can be written down by using the coproduct (\cite{Gu1} and \cite{GNW}) and evaluation map (\cite{Gu1} and \cite{K1}) for the Guay's affine Yangian. It is known that the Guay's affine Yangian has a representation on the equivariant homology space of affine Laumon spaces (\cite{FFNR} and \cite{FT}). Similarly to principal $W$-algebras, we expect that we can construct geometric representations of rectangular $W$-algebras by using the results of \cite{U4} and \cite{KU}. In non-rectangular cases, it is conjectured that there exists a similar relationship between the affine shifted Yangian and an iterated $W$-algebra. The iterated $W$-algebra $\mathcal{W}^k_l(\mathfrak{g},(\mathfrak{a}_1,\cdots,\mathfrak{a}_l),(f_1,\cdots,f_l))$ is a vertex algebra associated with a finite dimensional simple Lie algebra $\mathfrak{g}$, an ordered sequence of subalgebras $(\mathfrak{a}_1,\cdots,\mathfrak{a}_L)$, and a sequence $(f_1,\cdots,f_l)$ of nilpotent elements of $(\mathfrak{a}_1,\cdots,\mathfrak{a}_l)$ defined by an iterated quantum Hamiltonian reduction. 
In \cite{CE}, it is conjectured that 
\begin{align*}
&\mathcal{W}^k_l(\mathfrak{gl}(N),(\mathfrak{gl}(m_1s_1),\cdots,\mathfrak{gl}(m_ls_l)),(f_1,\cdots,f_l)\\
&\qquad\qquad\cong\mathcal{W}^k(\mathfrak{gl}(N),f_1+\cdots+f_L)\otimes\bigotimes_{1\leq i<j\leq l}\beta\gamma^{m_im_j(s_j-s_i)},
\end{align*}
where $N=m_1s_1+\cdots+m_Ls_L$, $0<s_1<\cdots<s_L$, $f_i\in\mathfrak{gl}(m_is_i)$ is a nilpotent element of type $(m_i)^{s_i}$ and $\beta\gamma$ is the vertex algebra of a pair of symplectic bosons.
Moreover, it is also conjectured in \cite{CE} that an action of an iterated $W$-algebra of type $A$ on the equivariant homology space of the affine Laumon space will be given through an action of an affine shifted Yangian constructed in \cite {FT}.
Based on this conjecture, we can expect that there exists a non-trivial homomorphism from the affine shifted Yangian to an iterated $W$-algebras associated with a general nilpotent element. However,  tackling this issue is very difficult and remains unresolved.

In finite setting, Brundan-Kleshchev \cite{BK} gave a surjective homomorphism from a shifted Yangian, which is a subalgebra of the finite Yangian associated with $\mathfrak{gl}(n)$, to a finite $W$-algebra (\cite{Pr}) of type $A$ for its general nilpotent element. A finite $W$-algebra $\mathcal{W}^{\text{fin}}(\mathfrak{g},f)$ is an associative algebra associated with a reductive Lie algebra $\mathfrak{g}$ and a nilpotent element $f$ and is a finite analogue of a $W$-algebra $\mathcal{W}^k(\mathfrak{g},f)$ (\cite{DSK1} and \cite{A1}). In \cite{DKV}, De Sole, Kac and Valeri constructed a homomorphism from the finite Yangian of type $A$ to the finite $W$-algebras of type $A$ by using the Lax operator, which is a restriction of the homomorphism given by Brundan-Kleshchev in \cite{BK}.

In this article, motivated by the work of De Sole, Kac and Valeri \cite{DKV}, we construct a homomorphism from the Guay's affine Yangian $Y_{\hbar,\ve}(\widehat{\mathfrak{sl}}(n))$ to the universal enveloping algebra of the $W$-algebra $\mathcal{W}^k(\mathfrak{gl}(m+n),f)$ associated with a Lie algebra $\mathfrak{gl}(m+n)$ and a nilpotent element $f$ of type $(2^{n},1^{m-n})$ for $m>n$. 

In Section 2, we recall the definition of the Guay's affine Yangian and its evaluation map. We also explain the minimalistic presentation of the Guay's affine Yangian, which are used for the proof of the main theorem.
In Section 3, we recall the definition of a $W$-algebra. Moreover, we construct the elements 
\begin{gather*}
\{W^{(1)}_{i,j}\mid 1\leq i\leq m-n, 1\leq j\leq m\text{ or }m-n<i,j\leq m\},\\
\{W^{(2)}_{i,j}\mid m-n<i\leq m, 1\leq j\leq m\},
\end{gather*}
which are strong generators of $\mathcal{W}^{k}(\mathfrak{gl}(m+n),f)$. 
In Section 4, We compute the $\lambda$-brackets among these strong generators. 
The appendix is devoted to the proof of these computations.

In Section 5, we construct an algebra homomorphism from the Guay's affine Yangian to the universal enveloping algebra of $\mathcal{W}^{k}(\mathfrak{gl}(m+n),f)$.
\begin{Theorem}\label{T1}
Let $n\geq3$ and
\begin{gather*}
\hbar=-1,\qquad\ve=k+m+n.
\end{gather*}
Then, there exists an algebra homomorphism 
\begin{equation*}
\Phi\colon Y_{\hbar,\ve}(\widehat{\mathfrak{sl}}(n))\to \mathcal{U}(\mathcal{W}^{k}(\mathfrak{gl}(m+n),f)),
\end{equation*} 
where $\mathcal{U}(\mathcal{W}^{k}(\mathfrak{gl}(m+n),f))$ is the universal enveloping algebra of $\mathcal{W}^{k}(\mathfrak{gl}(m+n),f)$.
\end{Theorem}
The proof of Theorem~\ref{T1} is due to the OPE calculation results given in Section 4. We check the compatibility of $\Phi$ with the defining relations of the minimalistic presentation of the Guay's affine Yangian by a direct computation.
In Section 6, we explain the relationship between Theorem~\ref{T1} and the homomorphism given in \cite{DKV}.  

Actually, in \cite{U6}, we have extended Theorem~\ref{T1} to a general nilpotent element case. We gave a homomorphism from the Guay's affine Yangian to the universal enveloping algebra of a general $W$-algebra of type $A$ in a similar way to \cite{KU}. For this generalization, we need to extend the definition of the Guay's affine Yangian and construct the homomorphism corresponding to the coproduct for the Guay's affine Yangian. For the proof of the well-definedness of the coproduct, we use Theorem~\ref{T1}. 
We hope that the main result in \cite{U6} will help to resolve the conjecture in \cite{CE}. One of the difficulty for the solving the conjecture is that it is very complicated to make a finite presentation for the shifted Yangian. Thus, we need to consider the easiest case, that is, $\mathcal{W}^k(\mathfrak{gl}(m+n),f)$. Since all of $\lambda$-brackets of  $\mathcal{W}^k(\mathfrak{gl}(m+n),f)$ have already been computed in this article, we suppose that the computation results in this paper will be helpful for solving the conjecture.

\section{Guay's affine Yangian}

Let us recall the definition of the Guay's affine Yangian. The Guay's affine Yangian $Y_{\ve_1,\ve_2}(\widehat{\mathfrak{sl}}(n))$ was first introduced by Guay (\cite{Gu2} and \cite{Gu1}) and is the deformation of the universal enveloping algebra of the central extension of $\mathfrak{sl}(n)[u^{\pm1},v]$.
\begin{Definition}\label{Prop32}
Let $n\geq3$ and a $n\times n$ matrix $(a_{i,j})_{0\leq i,j\leq n-1}$ be
\begin{gather*}
a_{ij} =
	\begin{cases}
	2 &\text{if } i=j, \\
	         -1&\text{if }j=i\pm1,\\
	        -1 &\text{if }(i,j)=(0,n-1),(n-1,0),\\
		0  &\text{otherwise.}
	\end{cases}
\end{gather*}
The Guay's affine Yangian $Y_{\hbar,\ve}(\widehat{\mathfrak{sl}}(n))$ is the associative algebra generated by $X_{i,r}^{+}, X_{i,r}^{-}, H_{i,r}$ $(i \in \{0,1,\cdots, n-1\}, r = 0,1)$ subject to the following defining relations:
\begin{gather}
[H_{i,r}, H_{j,s}] = 0,\label{Eq2.1}\\
[X_{i,0}^{+}, X_{j,0}^{-}] = \delta_{ij} H_{i, 0},\label{Eq2.2}\\
[X_{i,1}^{+}, X_{j,0}^{-}] = \delta_{ij} H_{i, 1} = [X_{i,0}^{+}, X_{j,1}^{-}],\label{Eq2.3}\\
[H_{i,0}, X_{j,r}^{\pm}] = \pm a_{ij} X_{j,r}^{\pm},\label{Eq2.4}\\
[\tilde{H}_{i,1}, X_{j,0}^{\pm}] = \pm a_{ij}\left(X_{j,1}^{\pm}\right),\text{ if }(i,j)\neq(0,n-1),(n-1,0),\label{Eq2.5}\\
[\tilde{H}_{0,1}, X_{n-1,0}^{\pm}] = \mp \left(X_{n-1,1}^{\pm}-(\ve+\dfrac{n}{2}\hbar) X_{n-1, 0}^{\pm}\right),\label{Eq2.6}\\
[\tilde{H}_{n-1,1}, X_{0,0}^{\pm}] = \mp \left(X_{0,1}^{\pm}+(\ve+\dfrac{n}{2}\hbar) X_{0, 0}^{\pm}\right),\label{Eq2.7}\\
[X_{i, 1}^{\pm}, X_{j, 0}^{\pm}] - [X_{i, 0}^{\pm}, X_{j, 1}^{\pm}] = \pm a_{ij}\dfrac{\hbar}{2} \{X_{i, 0}^{\pm}, X_{j, 0}^{\pm}\}\text{ if }(i,j)\neq(0,n-1),(n-1,0),\label{Eq2.8}\\
[X_{0, 1}^{\pm}, X_{n-1, 0}^{\pm}] - [X_{0, 0}^{\pm}, X_{n-1, 1}^{\pm}]= \pm\dfrac{\hbar}{2} \{X_{0, 0}^{\pm}, X_{n-1, 0}^{\pm}\} - (\ve+\dfrac{n}{2}\hbar) [X_{0, 0}^{\pm}, X_{n-1, 0}^{\pm}],\label{Eq2.9}\\
(\ad X_{i,0}^{\pm})^{1-a_{ij}} (X_{j,0}^{\pm})= 0 \text{ if }i \neq j, \label{Eq2.10}
\end{gather}
where $\tilde{H}_{i,1}=H_{i,1}-\dfrac{\hbar}{2}H_{i,0}^2$ and we set $\{a,b\}$ as $ab+ba$.
\end{Definition}
\begin{Remark}
The defining relations of $Y_{\hbar,\ve}(\widehat{\mathfrak{sl}}(n))$ are different from those of $Y_{\ve_1,\ve_2}(\widehat{\mathfrak{sl}}(n))$ which is called the Guay's affine Yangian in \cite{K1}. In \cite{K1}, generators of $Y_{\ve_1,\ve_2}(\widehat{\mathfrak{sl}}(n))$ are denoted by 
\begin{equation*}
\{x^\pm_{i,r},h_{i,r}\mid0\leq i\leq n-1,r\in\mathbb{Z}_{\geq0}\}
\end{equation*}
with 2-parameters $\ve_1$ and $\ve_2$.
Actually, the algebra $Y_{\hbar,\ve}(\widehat{\mathfrak{sl}}(n))$ is isomorphic to  $Y_{\ve_1,\ve_2}(\widehat{\mathfrak{sl}}(n))$.
The isomorphism $\Psi$ from $Y_{\hbar,\ve}(\widehat{\mathfrak{sl}}(n))$ to $Y_{\ve_1,\ve_2}(\widehat{\mathfrak{sl}}(n))$ is given by
\begin{gather*}
\Psi(H_{i,0})=h_{i,0},\quad\Psi(X^\pm_{i,0})=x^\pm_{i,0},\\
\Psi(H_{i,1})=\begin{cases}
h_{0,1}&\text{ if $i=0$},\\
h_{i,1}+\dfrac{i}{2}(\ve_1-\ve_2)h_{i,0}&\text{ if $i\neq0$},
\end{cases}\\
\hbar=\ve_1+\ve_2,\qquad
\ve=-n\ve_2.
\end{gather*}
In this paper, we do not use $Y_{a,b}(\widehat{\mathfrak{sl}}(n))$ in the meaning of \cite{K1}.
\end{Remark}

Let us recall the evaluation map for the Guay's affine Yangian (see \cite{Gu1} and \cite {K1}). First, we consider the Lie algebra 
\begin{equation*}
\widehat{\mathfrak{gl}}(n)=\Big(\bigoplus_{s\in\mathbb{Z}}\limits\bigoplus_{1\leq i,j\leq n}\limits E_{i,j}t^s\Big)\oplus\mathbb{C}c\oplus\mathbb{C}z
\end{equation*}
whose commutator relations are determined by
\begin{gather*}
[E_{p,q}t^s,E_{i,j}t^u]=\delta_{i,q}E_{p,j}t^{s+u}-\delta_{p,j}E_{i,q}t^{s+u}+s\delta_{i,q}\delta_{p,j}\delta_{s+u,0}c+s\delta_{p,q}\delta_{i,j}\delta_{s+u,0}z,\\
\text{$c$ and $z$ are central elements}.
\end{gather*}
Next, we introduce a completion of $U(\widehat{\mathfrak{gl}}(n))/U(\widehat{\mathfrak{gl}}(n))(z-1)$ following \cite{MNT}. We set the grading of $U(\widehat{\mathfrak{gl}}(n))/U(\widehat{\mathfrak{gl}}(n))(z-1)$ as $\text{deg}(E_{i,j}t^s)=s$ and $\text{deg}(c)=0$.
Then, $U(\widehat{\mathfrak{gl}}(n))/U(\widehat{\mathfrak{gl}}(n))(z-1)$ becomes a graded algebra and we denote the set of the degree $d$ elements of $U(\widehat{\mathfrak{gl}}(n))/U(\widehat{\mathfrak{gl}}(n))(z-1)$ by $U(\widehat{\mathfrak{gl}}(n))_d$. 
We obtain the completion
\begin{equation*}
U(\widehat{\mathfrak{gl}}(n))_{{\rm comp}}=\bigoplus_{d\in\mathbb{Z}}U(\widehat{\mathfrak{gl}}(n))_{{\rm comp},d},
\end{equation*}
where
\begin{equation*}
U(\widehat{\mathfrak{gl}}(n))_{{\rm comp},d}=\plim[N]U(\widehat{\mathfrak{gl}}(n))_{d}/\sum_{r>N}\limits U(\widehat{\mathfrak{gl}}(n))_{d-r}U(\widehat{\mathfrak{gl}}(n))_{r}.
\end{equation*}
The evaluation map for the Guay's affine Yangian is a non-trivial homomorphism from the Guay's affine Yangian to $U(\widehat{\mathfrak{gl}}(n))_{{\rm comp}}$. 
Here after, we set
\begin{gather*}
h_i=\begin{cases}
E_{n,n}-E_{1,1}+c&(i=0),\\
E_{ii}-E_{i+1,i+1}&(1\leq i\leq n-1),
\end{cases}\\
x^+_i=\begin{cases}
E_{n,1}t&(i=0),\\
E_{i,i+1}&(1\leq i\leq n-1),
\end{cases}
\quad x^-_i=\begin{cases}
E_{1,n}t^{-1}&(i=0),\\
E_{i+1,i}&(1\leq i\leq n-1).
\end{cases}
\end{gather*}
\begin{Theorem}[Section 6 in \cite{Gu1} and Theorem~3.8 in \cite{K1}]\label{thm:main}
Suppose that $\hbar\neq0$ and set $c=\dfrac{-n\hbar-\ve}{\hbar}$.
Then, there exists an algebra homomorphism 
\begin{equation*}
\ev_{\hbar,\ve} \colon Y_{\hbar,\ve}(\widehat{\mathfrak{sl}}(n)) \to U(\widehat{\mathfrak{gl}}(n))_{{\rm comp}}
\end{equation*}
uniquely determined by 
\begin{gather*}
	\ev_{\hbar,\ve}(X_{i,0}^{+}) = x_{i}^{+}, \quad \ev_{\hbar,\ve}(X_{i,0}^{-}) = x_{i}^{-},\quad \ev_{\hbar,\ve}(H_{i,0}) = h_{i},
\end{gather*}
\begin{gather*}
	\ev_{\hbar,\ve}(H_{i,1}) = \begin{cases}
		\hbar ch_{0} -\hbar E_{n,n} (E_{1,1}-c) \\
		\ +\hbar \displaystyle\sum_{s \geq 0} \limits\displaystyle\sum_{k=1}^{n}\limits E_{n,k}t^{-s} E_{k,n}t^s-\hbar \displaystyle\sum_{s \geq 0} \displaystyle\sum_{k=1}^{n}\limits E_{1,k}t^{-s-1} E_{k,1}t^{s+1}\\ \qquad\qquad\qquad\qquad\qquad\qquad\qquad\qquad\qquad\qquad\qquad\qquad\qquad\qquad\text{ if $i = 0$},\\
		-\dfrac{i}{2}\hbar h_{i} -\hbar E_{i,i}E_{i+1,i+1} \\
		\quad+ \hbar\displaystyle\sum_{s \geq 0}  \limits\displaystyle\sum_{k=1}^{i}\limits E_{i,k}t^{-s} E_{k,i}t^s+\hbar\displaystyle\sum_{s \geq 0} \limits\displaystyle\sum_{k=i+1}^{n}\limits E_{i,k}t^{-s-1} E_{k,i}t^{s+1}\\
\quad-\hbar\displaystyle\sum_{s \geq 0}\limits\displaystyle\sum_{k=1}^{i}\limits E_{i+1,k}t^{-s}E_{k,i+1}t^s-\hbar\displaystyle\sum_{s \geq 0}\limits\displaystyle\sum_{k=i+1}^{n} \limits E_{i+1,k}t^{-s-1}E_{k,i+1}t^{s+1}\\
\qquad\qquad\qquad\qquad\qquad\qquad\qquad\qquad\qquad\qquad\qquad\qquad\qquad\qquad \text{ if $i \neq 0$},
	\end{cases}
\end{gather*}
\begin{align*}
\ev_{\hbar,\ve}(X^+_{i,1})&=\begin{cases}
\hbar cx_{0}^{+} + \hbar \displaystyle\sum_{s \geq 0} \limits\displaystyle\sum_{k=1}^{n}\limits E_{n,k}t^{-s} E_{k,1}t^{s+1}\\
\qquad\qquad\qquad\qquad\qquad\qquad\qquad\qquad\qquad\qquad\qquad\qquad\qquad\qquad \text{ if $i = 0$},\\
-\dfrac{i}{2}\hbar x_{i}^{+}+ \hbar \displaystyle\sum_{s \geq 0}\limits\displaystyle\sum_{k=1}^i\limits E_{i,k}t^{-s} E_{k,i+1}t^s+\hbar \displaystyle\sum_{s \geq 0}\limits\displaystyle\sum_{k=i+1}^{n}\limits E_{i,k}t^{-s-1} E_{k,i+1}t^{s+1}\\
\qquad\qquad\qquad\qquad\qquad\qquad\qquad\qquad\qquad\qquad\qquad\qquad\qquad\qquad \text{ if $i \neq 0$},
\end{cases}
\end{align*}
\begin{align*}
\ev_{\hbar,\ve}(X^-_{i,1})&=\begin{cases}
\hbar cx_{0}^{-} +\hbar \displaystyle\sum_{s \geq 0} \limits\displaystyle\sum_{k=1}^{n}\limits E_{1,k}t^{-s-1} E_{k,n}t^{s},\\
\qquad\qquad\qquad\qquad\qquad\qquad\qquad\qquad\qquad\qquad\qquad\qquad\qquad\qquad \text{ if $i = 0$},\\
-\dfrac{i}{2}\hbar x_{i}^{-}+\hbar \displaystyle\sum_{s \geq 0}\limits\displaystyle\sum_{k=1}^i\limits E_{i+1,k}t^{-s} E_{k,i}t^{s}+ \hbar \displaystyle\sum_{s \geq 0}\limits\displaystyle\sum_{k=i+1}^{n}\limits E_{i+1,k}t^{-s-1} E_{k,i}t^{s+1}\\
\qquad\qquad\qquad\qquad\qquad\qquad\qquad\qquad\qquad\qquad\qquad\qquad\qquad\qquad \text{ if $i \neq 0$}.
\end{cases}
\end{align*}
\end{Theorem}
\section{Explicit generators of general $W$-algebras}
\subsection{$W$-algebras}
In this subsection, we recall the definition of a $W$-algebra. We fix some notations for vertex algebras. For a vertex algebra $V$, we denote the generating field associated with $v\in V$ by $v(z)=\displaystyle\sum_{n\in\mathbb{Z}}\limits v_{(n)}z^{-n-1}$. We also denote the OPE of $V$ by
\begin{equation*}
u(z)v(w)\sim\displaystyle\sum_{s\geq0}\limits \dfrac{(u_{(s)}v)(w)}{(z-w)^{s+1}}
\end{equation*}
for all $u, v\in V$. We denote the vacuum vector (resp.\ the translation operator) by $|0\rangle$ (resp.\ $\partial$).

Let us take a finite dimensional reductive Lie algebra $\mathfrak{g}$, a nilpotent element $f\in\mathfrak{g}$, a non-degenerate symmetric invariant bilinear form $(\ ,\ )$ and a complex number $k$. We fix an $\mathfrak{sl}(2)$-triple $(e,h,f)$ satisfying that 
\begin{gather*}
\mathfrak{g}=\bigoplus_{j\in\frac{1}{2}\mathbb{Z}}\mathfrak{g}_j=\bigoplus_{j\in\frac{1}{2}\mathbb{Z}}\{x\in\mathfrak{g}\mid [h,x]=2jx\},\\ f\in\mathfrak{g}_{-1},
\end{gather*}
and $\ad(f)\colon\mathfrak{g}_{\frac{1}{2}}\to\mathfrak{g}_{-\frac{1}{2}}$ is an isomorphism.
We choose a basis $\{u_i\}_{i\in S'}$ of $\mathfrak{g}_{\frac{1}{2}}$ and extend this basis to  $\{u_i\}_{i\in \tilde{S}}$ of $\mathfrak{g}$. We take $S\subset\tilde{S}$ to satisfy that $\{u_i\}_{i\in S}$ is a basis of $\mathfrak{g}_+=\bigoplus_{j>0}\mathfrak{g}_j$. We also denote by $\mathfrak{b}$ a Lie subalgebra $\bigoplus_{j\leq0}\mathfrak{g}_j$. We introduce the inner product on $\mathfrak{b}$ by
\begin{equation*}
\kappa(x,y)=k(x|y)+\dfrac{1}{2}(\kappa_{\mathfrak{g}}(x,y)-\kappa_{\mathfrak{g}_0}(x,y)-\kappa_{\frac{1}{2}}(x,y)),
\end{equation*}
where $\kappa_{\mathfrak{g}}$ (resp. $\kappa_{\mathfrak{g}_0}$) is the killing form on $\mathfrak{g}$ (resp. $\mathfrak{g}_0$) and $\kappa_{\frac{1}{2}}$ is the trace form of $\mathfrak{g}_0$ on $\mathfrak{g}_{\frac{1}{2}}$. We denote the universal affine vertex algebra associated with $\mathfrak{b}$ and $\kappa$ by $V^\kappa(\mathfrak{b})$. By the PBW theorem, we can identify $V^\kappa(\mathfrak{b})$ with $U(t^{-1}\mathfrak{b}[t^{-1}])$. In order to simplify the notation, here after, we denote the generating field $(ut^{-1})(z)$ as $u(z)$. By the definition of $V^\kappa(\mathfrak{b})$, generating fields $u(z)$ and $v(z)$ satisfy the OPE
\begin{gather}
u(z)v(w)\sim\dfrac{[u,v](w)}{z-w}+\dfrac{\kappa(u,v)}{(z-w)^2}\label{OPE1}
\end{gather}
for all $u,v\in\mathfrak{b}$. 

We construct a vertex algebra $A$ generated by
\begin{equation*}
\{v(z),\psi_i(z),\Phi_j(z)\mid v\in\mathfrak{b},i\in S,j\in S'\}
\end{equation*}
whose relations are given by
\begin{gather*}
v_1(z)v_2(w)\sim\dfrac{[v_1,v_2]}{(z-w)}+\dfrac{\kappa(v_1,v_2)}{(z-w)^2}\text{ for }v_1,v_2\in\mathfrak{b},\\
v(z)\psi_i(w)\sim-\dfrac{\sum_{k\in S}\limits c_{i,k}(v)\psi_k}{(z-w)}\text{ for }v\in\mathfrak{b},\\
v(z)\Phi_i(w)\sim\begin{cases}
\dfrac{\sum_{j\in S'}\limits c_{j,i}(v)\Phi_j(z)}{(z-w)}&\text{ if }v\in\mathfrak{g}_0,\\
0&\text{ if }v\in\mathfrak{g}_j\text{ where }j\neq0,
\end{cases}\\
\psi_i(z)\psi_j(w)\sim\Phi_i(z)\Phi_j(w)\sim\psi_i(z)\Phi_j(w)\sim0,
\end{gather*}
where we set $c_{k,i}(v)$ as $[v,u_i]=\sum_{k\in \widetilde{S}}\limits c_{k,i}(v)u_k$ and $v(z)$ is an even and other generators are odd. 
By Theorem~2.4 in \cite{KRW}, we can give a definition of a $W$-algebra.
\begin{Definition}\label{Def1}
We can define an odd differential
\begin{equation*}
d_0\colon V^\kappa(\mathfrak{b})\to A
\end{equation*}
determined by
\begin{gather*}
d_0(|0\rangle)=0,\qquad[d_0,\partial]=0,\\
\begin{align*}
d_0(v(z))&=\sum_{\beta\in S}\limits([f,v]|u_\beta)\psi_\beta(z)+\sum_{\beta\in S}\limits:\psi_\beta(z)\Phi_{[v,u_\beta]}(z):\\
&\quad-\sum_{\substack{\beta\in S\\ [v,u_\beta]\in\mathfrak{b}}}:\psi^\beta(z)[v,u_\beta](z):+\sum_{\beta\in S}\limits(k(v|u_\beta)+\text{str}_{\mathfrak{g}_+}(p_+(\ad v)(\ad u_\beta)))\partial\psi_\beta(z)\\
&\quad+\delta(v\in\mathfrak{g}_0)\sum_{\alpha\in S_{\frac{1}{2}}}\limits\Phi_{[u_\alpha,v]}(z),
\end{align*}
\end{gather*}
where $p_+$ is a natural projection of $\mathfrak{g}$ to $\mathfrak{g}_+$.
Then, we define a $W$-algebra $\mathcal{W}^k(\mathfrak{g},f,k)$ by
\begin{equation*}
\mathcal{W}^k(\mathfrak{g},f,k)=\{x\in V^\kappa(\mathfrak{b})|d_0(x)=0\}.
\end{equation*}
\end{Definition}
\subsection{Explicit generators}
In this article, we deal with the $W$-algebra associated with a Lie algebra $\mathfrak{g}=\mathfrak{gl}(m+n)$ and a nilpotent element of type $(2^{n},1^{m-n})$ for $m\geq n$.
We take $\{e_{i,j}\}_{1\leq i,j\leq m+n}$ as a unit basis of $\mathfrak{gl}(m+n)$. We consider the nilpotent element  
\begin{equation*}
f=\sum_{m+1\leq j\leq m+n}e_{j,j-n}.
\end{equation*}
We also fix an inner product of $\mathfrak{g}$ determined by
\begin{equation*}
(e_{i,j},e_{p,q})=k\delta_{i,q}\delta_{p,j}+\delta_{i,j}\delta_{p,q}.
\end{equation*}
In order to simplify the notation, we set
\begin{gather*}
\col(i)=\begin{cases}
1&\text{ if }i\leq m,\\
2&\text{ if }i>m,
\end{cases}\qquad\row(i)=\begin{cases}
i&\text{ if }i\leq m,\\
i-n&\text{ if }i>m.
\end{cases}
\end{gather*}
For all $1\leq i,j\leq m+n$, For $m-n<i\leq m$ we set $\hat{i}$ to be $i+n$, and for $m<j\leq m+n$ we set $\tilde{j}$ to be $j-n$.
By using these notations, we can rewrite $f$ as $\sum_{m-n+1\leq j\leq m}\limits e_{\hat{j},j}$. For example, in the case when $m=4,n=3$, we have
\begin{align*}
f&=e_{5,2}+e_{6,3}+e_{7,4}.
\end{align*}
We take an $\mathfrak{sl}(2)$-triple satisfying that $e_{i,j}\in\mathfrak{g}_s$ if and only if $\col(j)-\col(i)=s$.
Then, by the definition of $\mathfrak{b}$ and $\kappa$, we have
\begin{align*}
\mathfrak{b}&=\bigoplus_{\substack{1\leq i,j\leq m+n\\\col(i)\geq\col(j)}}\limits \mathbb{C}e_{i,j}=\bigoplus_{1\leq i,j\leq m}\mathbb{C}e_{i,j}\oplus\bigoplus_{m+1\leq i,j\leq m+n}\mathbb{C}e_{i,j}\oplus\bigoplus_{m-n+1\leq i,j\leq m}\mathbb{C}e_{\hat{i},j}.
\end{align*}
and
\begin{equation*}
\kappa(e_{i,j},e_{p,q})=\begin{cases}
\alpha_1\delta_{i,q}\delta_{p,j}+\delta_{i,j}\delta_{p,q}&\text{ if }1\leq i,j\leq m,\\
\alpha_2\delta_{i,q}\delta_{p,j}+\delta_{i,j}\delta_{p,q}&\text{ if }m+1\leq i,j\leq m+n,\\
0&\text{otherwise},
\end{cases}
\end{equation*}
where $\alpha_1=k+n$ and $\alpha_2=k+m$. By the definition of $S$, we have
\begin{align*}
S&=\{(i,j)\mid1\leq i,j\leq m+n,\col(j)>\col(i)\},\\
\widetilde{S}&=\{(i,j)\mid1\leq i,j\leq m+n\}.
\end{align*}
We consider the following Lie algebra
\begin{equation*}
\mathfrak{a}=\mathfrak{b}\oplus\displaystyle\bigoplus_{\substack{1\leq i,j\leq m+n\\\col(i)>\col(j)}}\limits\mathbb{C}\psi_{i,j}=\mathfrak{b}\oplus\bigoplus_{m-n+1\leq i,j\leq m}\mathbb{C}\psi_{\hat{i},j}
\end{equation*}
whose commutator relations;
\begin{align*}
[e_{i,j},\psi_{p,q}]&=\delta_{j,p}\psi_{i,q}-\delta_{i,q}\psi_{p,j},\\
[\psi_{i,j},\psi_{p,q}]&=0,
\end{align*}
where $e_{i,j}$ is an even element and $\psi_{i,j}$ is an odd element. We set the inner product on $\mathfrak{a}$ such that
\begin{gather*}
\widetilde{\kappa}(e_{i,j},e_{p,q})=\kappa(e_{i,j},e_{p,q}),\qquad\widetilde{\kappa}(e_{i,j},\psi_{p,q})=\widetilde{\kappa}(\psi_{i,j},\psi_{p,q})=0.
\end{gather*}
Then, a vertex algebra $A$ defined in Section 3.1 is the universal affine vertex algebra associated with 
$\mathfrak{a}$ and $\widetilde{\kappa}$ by corresponding $\phi_{(i,j)}\in A$ to $\psi_{j,i}\in V^{\widetilde{\kappa}}(\mathfrak{a})$.

Similarly to $V^\kappa(\mathfrak{b})$, we identify $V^{\widetilde{\kappa}}(\mathfrak{a})$ with $U(t^{-1}\mathfrak{a}[t^{-1}])$. For all $u\in \mathfrak{a}$, let $u[-s]$ be $ut^{-s}$. In this section, we regard $V^{\widetilde{\kappa}}(\mathfrak{a})$ (resp.\ $V^\kappa(\mathfrak{b})$) as a non-associative superalgebra structure by $(-1)$-product. In particular, we obtain
\begin{equation*}
u[-w]\cdot v[-s]=(u[-w])_{(-1)}v[-s].
\end{equation*}
We sometimes omit $\cdot$ in order to simplify the notation. By \cite{KW1} and \cite{KW2}, a $W$-algebra $\mathcal{W}^k(\mathfrak{gl}(m+n),f)$ can be realized as the vertex subalgebra of $V^\kappa(\mathfrak{b})$ as follows.

Let us define an odd differential $d_0 \colon V^{\kappa}(\mathfrak{b})\to V^{\widetilde{\kappa}}(\mathfrak{a})$ determined by
\begin{gather}
d_0(|0\rangle)=0,\\
[d_0,\partial]=0,\label{ee5800}
\end{gather}
\begin{align}
d_0(e_{i,j}[-1])
&=\sum_{\substack{\col(i)>\col(r)\geq\col(j)}}\limits e_{r,j}[-1]\psi_{i,r}[-1]-\sum_{\substack{\col(j)<\col(r)\leq\col(i)}}\limits \psi_{r,j}[-1]e_{i,r}[-1]\nonumber\\
&\quad+\delta(\col(i)>\col(j))\alpha_2\psi_{i,j}[-2]+\psi_{\hat{i},j}[-1]-\psi_{i,\tilde{j}}[-1].\label{ee1}
\end{align}
Here, we assumed that $e_{p,q}=0$ and $\psi_{p,q}= 0$ for out-of-range subscripts. We also supposed that $X_{\hat{i},j}=0$ and $X_{i,\tilde{j}}=0$ ($X=e,\psi$) if $\hat{i}$ and $\tilde{j}$ do not exist.

By Definition~\ref{Def1} we obtain the following theorem.
\begin{Theorem}\label{T125}
The $W$-algebra $\mathcal{W}^k(\mathfrak{gl}(m+n),f)$ is the vertex subalgebra of $V^\kappa(\mathfrak{b})$ defined by
\begin{equation*}
\mathcal{W}^k(\mathfrak{gl}(m+n),f)=\{y\in V^\kappa(\mathfrak{b})\mid d_0(y)=0\}.
\end{equation*}
\end{Theorem}

We construct two kinds of elements $W^{(1)}_{i,j}$ and $W^{(2)}_{i,j}$, which are actually strong generators of $\mathcal{W}^k(\mathfrak{gl}(m+n),f)$. 

\begin{Theorem}
Let us set
\begin{align*}
W^{(1)}_{i,j}&=\sum_{\substack{1\leq h,l\leq m+n,\\\row(h)=i,\row(l)=j,\\\col(h)=\col(l)}}e_{h,l}[-1]\ \text{for }1\leq i\leq m-n,1\leq j\leq m\text{ or }m-n<i,j\leq m,\\
W^{(2)}_{i,j}&=e_{\hat{i},j}[-1]-\alpha_2e_{i,j}[-2]+\sum_{m-n<u\leq m}\limits e_{u,j}[-1]e_{\hat{i},\hat{u}}[-1]-\sum_{1\leq u\leq m-n}\limits e_{u,j}[-1]e_{i,u}[-1]\\
&\qquad\qquad\qquad\qquad\qquad\qquad\qquad\qquad\qquad\qquad \text{for }m-n<i\leq m,1\leq j\leq m.
\end{align*}
Then, the $W$-algebra $\mathcal{W}^k(\mathfrak{gl}(m+n),f)$ has the following strong generators;
\begin{gather*}
\{W^{(1)}_{i,j}\mid 1\leq i\leq m-n, 1\leq j\leq m\text{ or }m-n<i,j\leq m\},\\
\{W^{(2)}_{i,j}\mid m-n<i\leq m, 1\leq j\leq m\}.
\end{gather*}
\end{Theorem}
\begin{proof}
First, we show that $W^{(1)}_{i,j}$ and $W^{(2)}_{i,j}$ are contained in $\mathcal{W}^k(\mathfrak{gl}(m+n),f)$. It is enough to show that $d_0(W^{(r)}_{i,j})=0$.
We only show the case when $r=2$. The case when $r=1$ is proven in a similar way.
By \eqref{ee1}, if $\col(p)=\col(q)$, we obtain
\begin{align}
d_0(e_{p,q}[-1])
&=\delta(m-n<p\leq m)\psi_{\widehat{p},q}[-1]-\delta(\col(q)=2)\psi_{p,\widetilde{q}}[-1].\label{ee307}
\end{align}
If $\col(p)=\col(q)+1=2$, by \eqref{ee1}, we also have
\begin{align}
d_0(e_{p,q}[-1])
&=\sum_{\substack{1\leq r\leq m}}e_{r,q}[-1]\psi_{p,r}[-1]-\sum_{\substack{m+1\leq r\leq m+n}}\psi_{r,q}[-1]e_{p, r}[-1]+\alpha_2\psi_{p,q}[-2].\label{ee2}
\end{align}
By the definition of $W^{(2)}_{i,j}$, we find that
\begin{align}
d_0(W^{(2)}_{i,j})
&=A_1+A_2+A_3+A_4+A_5+A_6,\label{ee3}
\end{align}
where
\begin{gather*}
A_1=d_0(e_{\hat{i},j}[-1]),\ A_2=-\alpha_2d_0(e_{i,j}[-2]),\\
A_3=\sum_{m-n<u\leq m}\limits d_0(e_{u,j}[-1])e_{\hat{i},\hat{u}}[-1],\\
A_4=\sum_{m-n<u\leq m}\limits e_{u,j}[-1]d_0(e_{\hat{i},\hat{u}}[-1]),\\
A_5=-\sum_{1\leq u\leq m-n}\limits d_0(e_{u,j}[-1])e_{i,u}[-1],\\
A_6=-\sum_{1\leq u\leq m-n}\limits e_{u,j}[-1]d_0(e_{i,u}[-1]).
\end{gather*}
We rewrite $A_i$ respectively.
By \eqref{ee2}, we obtain
\begin{align}
A_1
&=\sum_{1\leq r\leq m}e_{r,j}[-1]\psi_{\hat{i},r}[-1]-\sum_{m-n\leq r\leq m}\psi_{\hat{r},j}[-1]e_{\hat{i}, \hat{r}}[-1]+\alpha_2\psi_{\hat{i},j}[-2].\label{ee5}
\end{align}
By \eqref{ee307} and \eqref{ee5800}, we obtain
\begin{align}
A_2
&=-\alpha_2\psi_{\hat{i},j}[-2]+\alpha_2\cdot0\nonumber\\
&=-\alpha_2\psi_{\hat{i},j}[-2].\label{ee6}
\end{align}
since $\tilde{j}$ does not exist.
By \eqref{ee307}, we obtain
\begin{align}
A_3
&=\sum_{m-n<u\leq m}\limits (\psi_{\hat{u},j}[-1]-0)e_{\hat{i},\hat{u}}[-1]\nonumber\\
&=\sum_{m-n<u\leq m}\limits \psi_{\hat{u},j}[-1]e_{\hat{i},\hat{u}}[-1]\label{ee7},\\
A_4
&=\sum_{m-n<u\leq m}\limits e_{u,j}[-1](0-\psi_{\hat{i},u}[-1])\nonumber\\
&=-\sum_{m-n<u\leq m}\limits e_{u,j}[-1]\psi_{\hat{i},u}[-1],\label{ee8}\\
A_5
&=-\sum_{1\leq u\leq m-n}\limits (0-0)e_{i,u}[-1]\nonumber\\
&=0,\label{ee9}\\
A_6
&=-\sum_{1\leq u\leq m-n}\limits e_{u,j}[-1](\psi_{\hat{i},u}[-1]-0)\nonumber\\
&=-\sum_{1\leq u\leq m-n}\limits e_{u,j}[-1]\psi_{\hat{i},u}[-1].\label{ee10}
\end{align}
Here after, in order to simplify the notation, let us denote the $i$-th term of the right hand side of (the number of the equation) by 
$\text{(the number of the equation)}_i$. By a direct computation, we obtain
\begin{gather*}
\eqref{ee5}_1+\eqref{ee8}+\eqref{ee10}=0,\\
\eqref{ee5}_2+\eqref{ee7}=0,\\
\eqref{ee5}_3+\eqref{ee6}=0.
\end{gather*}
Then, adding \eqref{ee5}-\eqref{ee10}, we obtain $d_0(W^{(2)}_{i,j})=0$.

Finally, we show that $W^{(1)}_{i,j}$ and $W^{(2)}_{i,j}$ are strong generators of $\mathcal{W}^k(\mathfrak{gl}(m+n),f)$. By the definition of $f$, a basis of $\mathfrak{gl}(m+n)^f=\{x\in\mathfrak{gl}(m+n)\mid[f,x]=0\}$ is 
\begin{gather*}
\{\sum_{\substack{1\leq h,l\leq m+n,\\\row(h)=p,\col(l)=q,\\\col(h)=\col(l)}}e_{h,l}\mid1\leq p\leq m-n, 1\leq q\leq m\text{ or }m-n<p,q\leq m\}\qquad\qquad\\
\qquad\qquad\cup\{e_{\hat{i},j}\mid m-n<i\leq m,1\leq j\leq m-n\}.
\end{gather*}
Thus, by Theorem 4.1 in \cite{KW1}, $W^{(1)}_{i,j}$ and $W^{(2)}_{i,j}$ become strong generators of $\mathcal{W}^k(\mathfrak{gl}(m+n),f)$.
\end{proof}

\section{OPEs of the $W$-algebra $\mathcal{W}^k(\mathfrak{gl}(m+n),f)$}
Let us recall the definition of a universal enveloping algebra of a vertex algebra in the sense of \cite{FZ} and \cite{MNT}.
For any vertex algebra $V$, let $L(V)$ be the Borchards Lie algebra, that is,
\begin{align}
 L(V)=V{\otimes}\mathbb{C}[t,t^{-1}]/\text{Im}(\partial\otimes\id +\id\otimes\frac{d}{d t})\label{844},
\end{align}
where the commutation relation is given by
\begin{align*}
 [ut^a,vt^b]=\sum_{r\geq 0}\begin{pmatrix} a\\r\end{pmatrix}(u_{(r)}v)t^{a+b-r}
\end{align*}
for all $u,v\in V$ and $a,b\in \mathbb{Z}$. Now, we define the universal enveloping algebra of $V$.
\begin{Definition}[Section~6 in \cite{MNT}]\label{Defi}
We set $\mathcal{U}(V)$ as the quotient algebra of the standard degreewise completion of the universal enveloping algebra of $L(V)$ by the completion of the two-sided ideal generated by
\begin{gather}
(u_{(a)}v)t^b-\sum_{i\geq 0}
\begin{pmatrix}
 a\\i
\end{pmatrix}
(-1)^i(ut^{a-i}vt^{b+i}-(-1)^avt^{a+b-i}ut^{i}),\label{241}\\
|0\rangle t^{-1}-1.\label{242}
\end{gather}
We call $\mathcal{U}(V)$ the universal enveloping algebra of $V$.
\end{Definition}
In Section 5, we will construct a homomorphism from the Guay's affine Yangian to the universal enveloping algebra of $\mathcal{W}^k(\mathfrak{gl}(m+n),f)$. In order to construct the homomorphism, we need to compute the following terms; 
\begin{gather*}
(W^{(1)}_{i,j})_{(u)}W^{(2)}_{s,t}\ (u\geq0),\quad(W^{(2)}_{i,i})_{(0)}W^{(2)}_{j,j},\quad(W^{(2)}_{i,i})_{(1)}W^{(2)}_{j,j}.
\end{gather*}
By Theorem~2.4 in Kac-Roan-Wakimoto~\cite{KRW}, there exists a homomorphism from the universal enveloping algebra of an affinization of $\mathfrak{gl}(m+n)^f$ to $\mathcal{U}(\mathcal{W}^k(\mathfrak{gl}(m+n),f))$.
\begin{Theorem}\label{Tho1}
The following equations hold;
\begin{gather*}
(W^{(1)}_{p,q})_{(0)}W^{(1)}_{i,j}=\delta_{q,i}W^{(1)}_{p,j}-\delta_{p,j}W^{(1)}_{i,q},\\
(W^{(1)}_{p,q})_{(1)}W^{(1)}_{i,j}=\delta_{q,i}\delta_{p,j}(\alpha_1+\delta(p,i>m-n)\alpha_2)|0\rangle+\delta_{p,q}\delta_{i,j}(1+\delta(p,i>m-n))|0\rangle,\\
(W^{(1)}_{p,q})_{(s)}W^{(1)}_{i,j}=0\text{ for all }s>1.
\end{gather*}
In particular, for all $i,j,p,q>m-n$, we obtain
\begin{align*}
&\quad[W^{(1)}_{p,q}t^s,W^{(1)}_{i,j}t^u]\\
&=\delta_{q,i}W^{(1)}_{p,j}t^{s+u}-\delta_{p,j}W^{(1)}_{i,q}t^{s+u}+\delta_{q,i}\delta_{p,j}s(\alpha_1+\alpha_2)|0\rangle t^{s+u-1}+2\delta_{p,q}\delta_{i,j}s|0\rangle t^{s+u-1}\\
&=\delta_{q,i}W^{(1)}_{p,j}t^{s+u}-\delta_{p,j}W^{(1)}_{i,q}t^{s+u}+\delta_{s+u,0}\delta_{q,i}\delta_{p,j}s(\alpha_1+\alpha_2)+2\delta_{s+u,0}\delta_{p,q}\delta_{i,j}s,
\end{align*}
where the last equality is due to \eqref{242}.
\end{Theorem}
By a direct computation, we obtain the following theorem.
\begin{Theorem}
The following four equations hold;
\begin{align*}
&\quad(W^{(1)}_{p,q})_{(0)}W^{(2)}_{i,j}\\
&=-\delta_{p,j}W^{(2)}_{i,q}+\delta_{i,q}\delta(p>m-n)W^{(2)}_{p,j}-\delta_{i,q}\delta(p\leq m-n)\sum_{w\leq m-n}\limits (W^{(1)}_{w,j})_{(-1)}W^{(1)}_{p,w}\\
&\quad-\delta_{i,q}\delta(p\leq m-n)\alpha_2\partial W^{(1)}_{p,j}+\delta(p\leq m-n,q>m-n)(W^{(1)}_{p,j})_{(-1)}W^{(1)}_{i,q},
\end{align*}
\begin{align*}
&\quad(W^{(1)}_{p,q})_{(1)}W^{(2)}_{i,j}\\
&=\delta_{p,j}\alpha_1\delta(q>m-n)W^{(1)}_{i,q}-\delta_{i,q}\delta(p\leq m-n)(\alpha_2+\alpha_1)W^{(1)}_{p,j}\\
&\quad+\delta_{p,q}\delta(j>m-n)W^{(1)}_{i,j}-\delta_{i,j}\delta(q\leq m-n)W^{(1)}_{p,q}+\delta_{p,j}\delta_{q,i}\sum_{w\leq m-n}\limits W^{(1)}_{w,w},
\end{align*}
\begin{align*}
&\quad(W^{(1)}_{p,q})_{(2)}W^{(2)}_{i,j}\\
&=-\delta_{q,i}\delta_{p,j}(\alpha_1+\alpha_2)\alpha_1|0\rangle-\delta_{p,q}\delta_{i,j}(\delta(p,q\leq m-n)\alpha_1+2\alpha_2)|0\rangle,
\end{align*}
\begin{gather*}
(W^{(1)}_{p,q})_{(s)}W^{(2)}_{i,j}=0\text{ for all }s>2.
\end{gather*}
\end{Theorem}
Let us consider the case when $i,j,p,q>m-n$.
\begin{Corollary}\label{COR}
For all $i,j,p,q>m-n$, we obtain
\begin{align}
&\quad[W^{(1)}_{p,q}t^s,W^{(2)}_{i,j}t^u]\nonumber\\
&=-\delta_{p,j}W^{(2)}_{i,q}t^{s+u}+\delta_{i,q}W^{(2)}_{p,j}t^{s+u}\nonumber\\
&\quad+\delta_{p,j}s\alpha_1W^{(1)}_{i,q}t^{s+u-1}+\delta_{p,q}sW^{(1)}_{i,j}t^{s+u-1}+\delta_{p,j}\delta_{q,i}s\sum_{w\leq m-n}\limits W^{(1)}_{w,w}t^{s+u-1}\nonumber\\
&\quad-\delta_{q,i}\delta_{p,j}\delta_{s+u,1}\dfrac{s(s-1)}{2}(\alpha_1+\alpha_2)\alpha_1-\delta_{p,q}\delta_{i,j}\delta_{s+u,1}s(s-1)\alpha_2.\label{Lem1}
\end{align}
\end{Corollary}
By a direct computation, we obtain the following theorem. 
\begin{Theorem}\label{OPE3}
The following equations hold;
\begin{align}
&\quad (W^{(2)}_{p,q})_{(0)}W^{(2)}_{i,j}\nonumber\\
&=\delta(q>m-n)(W^{(2)}_{p,j})_{(-1)}W^{(1)}_{i,q}-\delta(j>m-n)(W^{(1)}_{p,j})_{(-1)}W^{(2)}_{i,q}\nonumber\\
&\quad+\alpha_1\delta(q,j>m-n)(\partial W^{(1)}_{p,j})_{(-1)}W^{(1)}_{i,q}+\delta(q,j>m-n)(\partial W^{(1)}_{p,q})_{(-1)}W^{(1)}_{i,j}\nonumber\\
&\quad-\delta_{q,i}\sum_{w\leq m-n}\limits (W^{(1)}_{w,j})_{(-1)}W^{(2)}_{p,w}-\delta_{q,i}\delta_{p,j}\sum_{x,w\leq m-n}\limits (\partial W^{(1)}_{w,x})_{(-1)}W^{(1)}_{x,w}-\delta_{q,i}\alpha_2 \partial W^{(2)}_{p,j}\nonumber\\
&\quad+\delta_{q,i}\delta(j>m-n)\sum_{w\leq m-n}\limits(\partial W^{(1)}_{p,j})_{(-1)}W^{(1)}_{w,w}\nonumber\\
&\quad-\delta_{p,j}\delta_{q,i}\dfrac{1}{2}(\alpha_1+2\alpha_2)\sum_{x\leq m-n}\limits \partial^2 W^{(1)}_{x,x}-\delta_{i,q}\delta(j>m-n)\dfrac{\alpha_1(\alpha_1+\alpha_2)+1}{2}\partial^2W^{(1)}_{p,j}\nonumber\\
&\quad+\delta_{p,j}\sum_{x\leq m-n}\limits(W^{(1)}_{x,q})_{(-1)}W^{(2)}_{i,x}+\delta_{p,j}\delta(q>m-n)\sum_{x\leq m-n}\limits (\partial W^{(1)}_{x,x})_{(-1)}W^{(1)}_{i,q}\nonumber\\
&\quad-\delta_{i,j}\partial W^{(2)}_{p,q}-\dfrac{1}{2}\delta_{i,j}\delta(q>m-n)(\alpha_1+2\alpha_2)\partial^2W^{(1)}_{p,q}-\dfrac{1}{2}\delta_{i,j}\delta_{p,q}\partial^2\sum_{w\leq m-n}W^{(1)}_{w,w},\label{OPE3-1}
\end{align}
\begin{align}
&\quad (W^{(2)}_{p,q})_{(1)}W^{(2)}_{i,j}\nonumber\\
&=\delta(q,j>m-n)\alpha_1(W^{(1)}_{p,j})_{(-1)}W^{(1)}_{i,q}+\delta(q,j>m-n)(W^{(1)}_{p,q})_{(-1)}W^{(1)}_{i,j}\nonumber\\
&\quad-\delta_{q,i}\alpha_2W^{(2)}_{p,j}+\delta_{q,i}\delta(j>m-n)\sum_{w\leq m-n}\limits (W^{(1)}_{p,j})_{(-1)}W^{(1)}_{w,w}-\delta_{q,i}\alpha_1(\alpha_1+\alpha_2)\partial W^{(1)}_{p,j}\nonumber\\
&\quad-\delta_{p,j}\alpha_2W^{(2)}_{i,q}+\delta_{p,j}\delta(q>m-n)\sum_{x\leq m-n}(W^{(1)}_{x,x})_{(-1)}W^{(1)}_{i,q}\nonumber\\
&\quad-\delta_{i,j}(1+\delta(q\leq m-n))W^{(2)}_{p,q}-2\delta_{i,j}\delta(q>m-n)\alpha_2\partial W^{(1)}_{p,q}\nonumber\\
&\quad-\delta_{p,q}(1+\delta(j\leq m-n))W^{(2)}_{i,j}\nonumber\\
&\quad-\delta_{p,j}\delta_{i,q}\sum_{x,w\leq m-n}(W^{(1)}_{w,x})_{(-1)}W^{(1)}_{x,w}-\delta_{q,i}\delta_{p,j}(\alpha_1+2\alpha_2)\sum_{x\leq m-n}\limits \partial W^{(1)}_{x,x},\label{OPE3-2}
\end{align}
\begin{align}
&\quad(W^{(2)}_{p,q})_{(2)}W^{(2)}_{i,j}\nonumber\\
&=\delta_{p,j}\delta(q>m-n)\alpha_1(\alpha_1+\alpha_2-1)W^{(1)}_{i,q}+\delta_{i,j}\delta(q>m-n)(\alpha_1-2\alpha_2)W^{(1)}_{p,q}\nonumber\\
&\quad-\delta_{i,q}\delta(j>m-n)\alpha_1(\alpha_1+\alpha_2-1)W^{(1)}_{i,q}-\delta_{p,q}\delta(q>m-n)(\alpha_1-2\alpha_2)W^{(1)}_{i,j},\label{OPE3-3}
\end{align}
\begin{align}
&\quad(W^{(2)}_{p,q})_{(3)}W^{(2)}_{i,j}\nonumber\\
&=(1+\alpha_1^2-6\alpha_2^2)\delta_{p,q}\delta_{i,j}|0\rangle\nonumber\\
&\quad+(6\alpha_1\alpha_2(m-n)+\alpha_1+\alpha_2+\alpha_1(m-n)-\alpha_1(m-n)^2-6\alpha_2^2\alpha_1)\delta_{p,j}\delta_{i,q}|0\rangle.\label{OPE3-4}
\end{align}
\end{Theorem}
The proof of \eqref{OPE3-1} and \eqref{OPE3-2} are given in the appendix.

In order to prove the main theorem, we apply Theorem~\ref{OPE3} to the case when $i=j,p=q>m-n$.
By Theorem~\ref{OPE3}, we obtain
\begin{align}
&\quad[W^{(2)}_{p,p}t,W^{(2)}_{i,i}t]\nonumber\\
&=(W^{(2)}_{p,i})_{(-1)}W^{(1)}_{i,p}t^2-(W^{(1)}_{p,i})_{(-1)}W^{(2)}_{i,p}t^2\nonumber\\
&\quad+\alpha_1(\partial W^{(1)}_{p,i})_{(-1)}W^{(1)}_{i,p}t^2+(\partial W^{(1)}_{p,p})_{(-1)}W^{(1)}_{i,i}t^2\nonumber\\
&\quad-\delta_{p,i}\sum_{w\leq m-n}\limits (W^{(1)}_{w,i})_{(-1)}W^{(2)}_{p,w}t^2-\delta_{p,i}\sum_{x,w\leq m-n}\limits (\partial W^{(1)}_{w,x})_{(-1)}W^{(1)}_{x,w}t^2-\delta_{p,i}\alpha_2\partial W^{(2)}_{p,i}t^2\nonumber\\
&\quad+\delta_{p,i}\sum_{w\leq m-n}\limits (\partial W^{(1)}_{p,i})_{(-1)}W^{(1)}_{w,w}t^2-\dfrac{1}{2}\delta_{p,i}(\alpha_1+2\alpha_2)\sum_{x\leq m-n}\limits \partial^2 W^{(1)}_{x,x}t^2\nonumber\\
&\quad-\delta_{i,p}\dfrac{\alpha_1(\alpha_1+\alpha_2)+1}{2}\partial^2W^{(1)}_{p,i}t^2\nonumber\\
&\quad+\delta_{p,i}\sum_{x\leq m-n}\limits (W^{(1)}_{x,p})_{(-1)}W^{(2)}_{i,x}t^2+\delta_{p,i}\sum_{x\leq m-n}\limits (\partial W^{(1)}_{x,x})_{(-1)}W^{(1)}_{i,p}t^2\nonumber\\
&\quad-\partial W^{(2)}_{p,p}t^2-\dfrac{1}{2}(\alpha_1+2\alpha_2)\partial^2W^{(1)}_{p,p}t^2-\dfrac{1}{2}\partial^2\sum_{w\leq m-n}W^{(1)}_{w,w}t^2\nonumber\\
&\quad+\alpha_1(W^{(1)}_{p,i})_{(-1)}W^{(1)}_{i,p}t+(W^{(1)}_{p,p})_{(-1)}W^{(1)}_{i,i}t\nonumber\\
&\quad-\delta_{p,i}\alpha_2W^{(2)}_{p,i}t+\delta_{p,i}\sum_{w\leq m-n}\limits (W^{(1)}_{p,i})_{(-1)}W^{(1)}_{w,w}t-\delta_{p,i}\alpha_1(\alpha_1+\alpha_2)\partial W^{(1)}_{p,i}t\nonumber\\
&\quad-\delta_{p,i}\alpha_2W^{(2)}_{i,p}t+\delta_{p,i}\sum_{x\leq m-n}(W^{(1)}_{x,x})_{(-1)}W^{(1)}_{i,p}t\nonumber\\
&\quad-W^{(2)}_{p,p}t-2\alpha_2\partial W^{(1)}_{p,p}t-W^{(2)}_{i,i}t\nonumber\\
&\quad-\delta_{p,i}\sum_{x,w\leq m-n}(W^{(1)}_{w,x})_{(-1)}W^{(1)}_{x,w}t-\delta_{p,i}(\alpha_1+2\alpha_2)\sum_{x\leq m-n}\limits \partial W^{(1)}_{x,x}t.\label{OPE4}
\end{align}
We rearrange each terms of the right hand side of \eqref{OPE4}. We divide the right hand side of \eqref{OPE4} into 9 pieces;
\begin{align*}
a_{p,i}&=\eqref{OPE4}_3+\eqref{OPE4}_{16},\\
b_{p,i}&=\eqref{OPE4}_4+\eqref{OPE4}_{17},\\
c_{p,i}&=\eqref{OPE4}_5+\eqref{OPE4}_{11},\\
d_{p,i}&=\eqref{OPE4}_6+\eqref{OPE4}_{26},\\
\tilde{e}_{p,i}&=\eqref{OPE4}_7+\eqref{OPE4}_{13}+\eqref{OPE4}_{18}+\eqref{OPE4}_{21}+\eqref{OPE4}_{23}+\eqref{OPE4}_{25},\\
f_{p,i}&=\eqref{OPE4}_8+\eqref{OPE4}_{12}+\eqref{OPE4}_{19}+\eqref{OPE4}_{22},\\
g_{p,i}&=\eqref{OPE4}_9+\eqref{OPE4}_{15}+\eqref{OPE4}_{27},\\
h_{p,i}&=\eqref{OPE4}_{10}+\eqref{OPE4}_{14}+\eqref{OPE4}_{20}+\eqref{OPE4}_{24},\\
k_{p,i}&=\eqref{OPE4}_1+\eqref{OPE4}_2.
\end{align*}
Let us rewrite each terms. By the definition of $c_{p,i}$, $c_{p,i}=0$ is clear. Since
\begin{equation*}
x_{(-1)}y t^2=\sum_{s\geq0}\limits(xt^{-1-s}yt^{2+s}+yt^{1-s}xt^{s})
\end{equation*}
holds by the definition of the universal enveloping algebra of the vertex algebra, we obtain
\begin{align}
k_{p,i}&=(W^{(2)}_{p,i})_{(-1)}W^{(1)}_{i,p}t^2-(W^{(1)}_{p,i})_{(-1)}W^{(2)}_{i,p}t^2\nonumber\\
&=\sum_{s\geq0}\limits(W^{(2)}_{p,i}t^{-1-s}W^{(1)}_{i,p}t^{2+s}+W^{(1)}_{i,p}t^{1-s}W^{(2)}_{p,i}t^{s})\nonumber\\
&\quad-\sum_{s\geq0}\limits(W^{(1)}_{p,i}t^{-1-s}W^{(2)}_{i,p}t^{2+s}+W^{(2)}_{i,p}t^{1-s}W^{(1)}_{p,i}t^{s}).\label{k1}
\end{align}
Since 
\begin{equation*}
(\partial^2 x)t^2=2x,\qquad(\partial x)t=-x
\end{equation*}
hold by the definition of the universal enveloping algebra of the vertex algebra, we obtain
\begin{align}
\tilde{e}_{p,i}&=-\delta_{p,i}\alpha_2\partial W^{(2)}_{p,i}t^2-\partial W^{(2)}_{p,p}t^2-\delta_{p,i}\alpha_2W^{(2)}_{p,i}t-\delta_{p,i}\alpha_2W^{(2)}_{i,p}t-W^{(2)}_{p,p}t-W^{(2)}_{i,i}t\nonumber\\
&=W^{(2)}_{p,p}t-W^{(2)}_{i,i}t,\label{e1}\\
g_{p,i}&=-\delta_{p,i}\dfrac{1}{2}(\alpha_1+2\alpha_2)\sum_{x\leq m-n}\limits \partial^2 W^{(1)}_{x,x}t^2-\dfrac{1}{2}\partial^2\sum_{w\leq m-n}W^{(1)}_{w,w}t^2\nonumber\\
&\quad-\delta_{p,i}(\alpha_1+2\alpha_2)\sum_{x\leq m-n}\limits \partial W^{(1)}_{x,x}t\nonumber\\
&=-\sum_{w\leq m-n}W^{(1)}_{w,w},\label{g1}\\
h_{p,i}&=-\dfrac{\alpha_1(\alpha_1+\alpha_2)+1}{2}\delta_{i,p}\partial^2W^{(1)}_{p,i}t^2-\dfrac{1}{2}(\alpha_1+2\alpha_2)\partial^2W^{(1)}_{p,p}t^2\nonumber\\
&\qquad-\delta_{p,i}\alpha_1(\alpha_1+\alpha_2)\partial W^{(1)}_{p,i}t-2\alpha_2\partial W^{(1)}_{p,p}t\nonumber\\
&=-\alpha_1W^{(1)}_{p,p}-\delta_{i,p}W^{(1)}_{i,i}.\label{h1}
\end{align}
Since
\begin{equation*}
(\partial W^{(1)}_{p,i})_{(-1)}W^{(1)}_{w,w}+(\partial W^{(1)}_{w,w})_{(-1)}W^{(1)}_{i,p}=\partial((W^{(1)}_{w,w})_{(-1)}W^{(1)}_{i,p})
\end{equation*}
holds for $w\leq m-n$ by Theorem~\ref{Tho1}, we obtain
\begin{align}
f_{p,i}&=\delta_{p,i}\sum_{w\leq m-n}\limits (\partial W^{(1)}_{p,i})_{(-1)}W^{(1)}_{w,w}t^2+\delta_{p,i}\sum_{x\leq m-n}\limits (\partial W^{(1)}_{x,x})_{(-1)}W^{(1)}_{i,p}t^2\nonumber\\
&\quad+\delta_{p,i}\sum_{w\leq m-n}\limits (W^{(1)}_{p,i})_{(-1)}W^{(1)}_{w,w}t+\delta_{p,i}\sum_{x\leq m-n}(W^{(1)}_{x,x})_{(-1)}W^{(1)}_{i,p}t\nonumber\\
&=\delta_{p,i}\sum_{w\leq m-n}\limits \partial((W^{(1)}_{i,i})_{(-1)}W^{(1)}_{w,w})t^2+2\delta_{p,i}\sum_{w\leq m-n}\limits (W^{(1)}_{i,i})_{(-1)}W^{(1)}_{w,w}t\nonumber\\
&=0,\label{f1}
\end{align}
where the second equation is due to the definition of the vertex algebras.
We also obtain
\begin{align}
a_{p,i}&=\alpha_1(\partial W^{(1)}_{p,i})_{(-1)}W^{(1)}_{i,p}t^2+\alpha_1(W^{(1)}_{p,i})_{(-1)}W^{(1)}_{i,p}t\nonumber\\
&=\alpha_1\sum_{s\geq0}\limits ((s+1)W^{(1)}_{p,i}t^{-s-1}W^{(1)}_{i,p}t^{s+1}-sW^{(1)}_{i,p}t^{-s}W^{(1)}_{p,i}t^s)\label{a1}
\end{align}
since we have
\begin{align}
&\quad(\partial x)_{(-1)}yt^2+x_{(-1)}yt\nonumber\\
&=\sum_{s\geq0}\limits (\partial xt^{-s-1}yt^{s+2}+yt^{1-s}\partial xt^{s})+\sum_{s\geq0}\limits (xt^{-s-1}yt^{s+1}+yt^{-s}xt^s)\nonumber\\
&=\sum_{s\geq0}\limits ((s+1)xt^{-s-2}yt^{s+2}-syt^{1-s}xt^{s-1})+\sum_{s\geq0}\limits (xt^{-s-1}yt^{s+1}+yt^{-s}xt^s)\nonumber\\
&=\sum_{s\geq0}\limits ((s+1)xt^{-s-1}yt^{s+1}-syt^{-s}xt^s),\label{benri}
\end{align}
where the first and second equations are due to the definition of the universal enveloping algebra of the vertex algebra and the third equation is due to a direct computation.
Similarly to $a_{p,i}$, by \eqref{benri}, we obtain
\begin{align}
b_{p,i}&=(\partial W^{(1)}_{p,p})_{(-1)}W^{(1)}_{i,i}t^2+(W^{(1)}_{p,p})_{(-1)}W^{(1)}_{i,i}t\nonumber\\
&=\sum_{s\geq0}\limits ((s+1)W^{(1)}_{p,p}t^{-s-1}W^{(1)}_{i,i}t^{s+1}-sW^{(1)}_{i,i}t^{-s}W^{(1)}_{p,p}t^s).\label{b1}
\end{align}
By using \eqref{benri}, we can rewrite $d_{p,i}$ as
\begin{align*}
d_{p,i}&=-\delta_{p,i}\sum_{x,w\leq m-n}\limits (\partial W^{(1)}_{w,x})_{(-1)}W^{(1)}_{x,w}t^2-\delta_{p,i}\sum_{x,w\leq m-n}(W^{(1)}_{w,x})_{(-1)}W^{(1)}_{x,w}t\\
&=-\delta_{p,i}\sum_{x,w\leq m-n}\limits\sum_{s\geq0}\limits ((s+1)W^{(1)}_{w,x}t^{-s-1}W^{(1)}_{x,w}t^{s+1}-sW^{(1)}_{x,w}t^{-s}W^{(1)}_{w,x}t^s).
\end{align*}
Since we have
\begin{align*}
\sum_{x,w\leq m-n}\limits\sum_{s\geq0}\limits (s+1)W^{(1)}_{w,x}t^{-s-1}W^{(1)}_{x,w}t^{s+1}=\sum_{x,w\leq m-n}\limits\sum_{s\geq0}\limits sW^{(1)}_{x,w}t^{-s}W^{(1)}_{w,x}t^s
\end{align*}
by changing $x$ and $w$, we obtain
\begin{equation}
d_{p,i}=0.\label{d1}
\end{equation}
By \eqref{k1}-\eqref{d1}, we have
\begin{align}
[W^{(2)}_{p,p}t,W^{(2)}_{i,i}t]&=a_{p,i}+b_{p,i}+\tilde{e}_{p,i}+g_{p,i}+h_{p,i}+k_{p,i}.\label{OPE5}
\end{align}
\section{A homomorphism from the Guay's affine Yangian to the universal enveloping algebra of $\mathcal{W}^{k}(\mathfrak{gl}(m+n),f)$. }
In order to simplify the notation, we will set
\begin{equation*}
w^{(r)}_{i,j}=W^{(r)}_{m-n+i,m-n+j},
\end{equation*}
\begin{Theorem}\label{Main}
Assume that
\begin{equation*}
\hbar=-1,\qquad\ve=n+\alpha_2.
\end{equation*}
There exists an algebra homomorphism 
\begin{equation*}
\Phi\colon Y_{\hbar,\ve}(\widehat{\mathfrak{sl}}(n))\to \mathcal{U}(\mathcal{W}^{k}(\mathfrak{gl}(m+n),f))
\end{equation*} 
determined by
\begin{align*}
\Phi(H_{i,0})=\begin{cases}
w^{(1)}_{n,n}-w^{(1)}_{1,1}+\alpha_1+\alpha_2&\text{ if }i=0,\\
w^{(1)}_{i,i}-w^{(1)}_{i+1,i+1}&\text{ if }i\neq 0,
\end{cases}\\
\Phi(X^+_{i,0})=\begin{cases}
w^{(1)}_{n,1}t&\text{ if }i=0,\\
w^{(1)}_{i,i+1}&\text{ if }i\neq0,
\end{cases}
\quad \Phi(X^-_{i,0})=\begin{cases}
w^{(1)}_{1,n}t^{-1}&\text{ if }i=0,\\
w^{(1)}_{i+1,i}&\text{ if }i\neq0,
\end{cases}
\end{align*}
\begin{align*}
\Phi(H_{i,1})&=\begin{cases}w^{(2)}_{n,n}t-w^{(2)}_{1,1}t+\alpha_1 w^{(1)}_{n,n}+\alpha_1(\alpha_1+\alpha_2)- (\alpha_1+\alpha_2)\Phi(H_{0,0}) +w^{(1)}_{n,n} (w^{(1)}_{1,1}-\alpha_1-\alpha_2)\\
\quad+\displaystyle\sum_{x\leq m-n}\limits W^{(1)}_{x,x}-\displaystyle\sum_{s \geq 0} \limits\displaystyle\sum_{u=1}^{n}\limits w^{(1)}_{n,u}t^{-s} w^{(1)}_{u,n}t^s+\displaystyle\sum_{s \geq 0}\displaystyle\sum_{u=1}^{n}\limits w^{(1)}_{1,u}t^{-s-1} w^{(1)}_{u,1}t^{s+1}\\
\qquad\qquad\qquad\qquad\qquad\qquad\qquad\qquad\qquad\qquad\qquad\qquad\qquad\qquad \text{ if }i=0,\\
w^{(2)}_{i,i}t-w^{(2)}_{i+1,i+1}t+\dfrac{i}{2}\Phi(H_{i,0})+w^{(1)}_{i,i}w^{(1)}_{i+1,i+1}\\
\quad-\displaystyle\sum_{s \geq 0}  \limits\displaystyle\sum_{u=1}^{i}\limits w^{(1)}_{i,u}t^{-s}w^{(1)}_{u,i}t^s-\displaystyle\sum_{s \geq 0} \limits\displaystyle\sum_{u=i+1}^{n}\limits w^{(1)}_{i,u}t^{-s-1} w^{(1)}_{u,i}t^{s+1}\\
\quad+\displaystyle\sum_{s \geq 0}\limits\displaystyle\sum_{u=1}^{i}\limits w^{(1)}_{i+1,u}t^{-s} w^{(1)}_{u,i+1}t^s+\displaystyle\sum_{s \geq 0}\limits\displaystyle\sum_{u=i+1}^{n} \limits w^{(1)}_{i+1,u}t^{-s-1} w^{(1)}_{u,i+1}t^{s+1}\\
\qquad\qquad\qquad\qquad\qquad\qquad\qquad\qquad\qquad\qquad\qquad\qquad\qquad\qquad \text{ if }i\neq0,
\end{cases}
\end{align*}
\begin{align*}
\Phi(X^+_{i,1})&=\begin{cases}
w^{(2)}_{n,1}t^2-
(\alpha_1+\alpha_2)\Phi(X_{0,0}^{+})-\displaystyle\sum_{s \geq 0} \limits\displaystyle\sum_{u=1}^{n}\limits w^{(1)}_{n,u}t^{-s} w^{(1)}_{u,1}t^{s+1}\\
\qquad\qquad\qquad\qquad\qquad\qquad\qquad\qquad\qquad\qquad\qquad\qquad\qquad\qquad \text{ if $i = 0$},\\
w^{(2)}_{i,i+1}t+\dfrac{i}{2}\Phi(X_{i,0}^{+})\\
\quad-\displaystyle\sum_{s \geq 0}\limits\displaystyle\sum_{u=1}^i\limits w^{(1)}_{i,u}t^{-s} w^{(1)}_{u,i+1}t^s-\displaystyle\sum_{s \geq 0}\limits\displaystyle\sum_{u=i+1}^{n}\limits w^{(1)}_{i,u}t^{-s-1} w^{(1)}_{u,i+1}t^{s+1}\\
\qquad\qquad\qquad\qquad\qquad\qquad\qquad\qquad\qquad\qquad\qquad\qquad\qquad\qquad \text{ if $i \neq 0$},
\end{cases}
\end{align*}
\begin{align*}
\Phi(X^-_{i,1})&=\begin{cases}
w^{(2)}_{1,n}+\alpha_1W^{(1)}_{1,n}t^{-1}-(\alpha_1+\alpha_2)\Phi(X_{0,0}^{-})-\displaystyle\sum_{s \geq 0} \limits\displaystyle\sum_{u=1}^{n}\limits w^{(1)}_{1,u}t^{-s-1} w^{(1)}_{u,n}t^s\\
\qquad\qquad\qquad\qquad\qquad\qquad\qquad\qquad\qquad\qquad\qquad\qquad\qquad\qquad \text{ if $i = 0$},\\
w^{(2)}_{i+1,i}t+\dfrac{i}{2}\Phi(X_{i,0}^{-})\\
\quad-\displaystyle\sum_{s \geq 0}\limits\displaystyle\sum_{u=1}^i\limits w^{(1)}_{i+1,u}t^{-s} w^{(1)}_{u,i}t^s-\displaystyle\sum_{s \geq 0}\limits\displaystyle\sum_{u=i+1}^{n}\limits w^{(1)}_{i+1,u}t^{-s-1} w^{(1)}_{u,i}t^{s+1} \\
\qquad\qquad\qquad\qquad\qquad\qquad\qquad\qquad\qquad\qquad\qquad\qquad\qquad\qquad\text{ if $i \neq 0$}.
\end{cases}
\end{align*}
\end{Theorem}

It is enough to show that $\Phi$ is compatible with the defining relations \eqref{Eq2.1}-\eqref{Eq2.10}. By Theorem~\ref{Tho1}, we find that $\Phi$ is compatible with \eqref{Eq2.2} and \eqref{Eq2.10}. 
Thus, it is enough to show that $\Phi$ is compatible with \eqref{Eq2.1} and \eqref{Eq2.3}-\eqref{Eq2.9}. We divide the proof into two pieces, that is, Lemma~\ref{Claim1234} and Lemma~\ref{Claim1236} below. In Lemma~\ref{Claim1234}, we show that $\Phi$ is compatible with \eqref{Eq2.3}-\eqref{Eq2.9}. In Lemma~\ref{Claim1236}, we prove that $\Phi$ is compatible with \eqref{Eq2.1}.

In order to prove Lemmas~\ref{Claim1234} and \ref{Claim1236}, we relate $\Phi$ with the evaluation map for the Guay's affine Yangian. We set $\widetilde{\ev}(H_{i,s})$ and $\widetilde{\ev}(X^\pm_{i,s})$ $(s=0,1)$ as
\begin{align*}
\widetilde{\ev}(H_{i,0})=\Phi(H_{i,0}),\quad\widetilde{\ev}(X^\pm_{i,0})=\Phi(X^\pm_{i,0}),
\end{align*}
\begin{align*}
\widetilde{\ev}(H_{i,1})&=\begin{cases}
\Phi(H_{0,1})-\Big(w^{(2)}_{n,n}t-w^{(2)}_{1,1}t+\alpha_1 w^{(1)}_{n,n}+\sum_{x\leq m-n}\limits W^{(1)}_{x,x}\Big)\text{ if }i=0,\\
\Phi(H_{i,1})-(w^{(2)}_{i,i}t-w^{(2)}_{i+1,i+1}t)\text{ if }i\neq0,
\end{cases}
\end{align*}
\begin{align*}
\widetilde{\ev}(X^+_{i,1})&=\begin{cases}
\Phi(X^+_{i,1})-(w^{(2)}_{n,1}t^2+\alpha_1 w^{(1)}_{n,1}t)&\text{ if }i=0,\\
\Phi(X^+_{i,1})-w^{(2)}_{i,i+1}t&\text{ if }i\neq0,
\end{cases}
\end{align*}
\begin{align*}
\widetilde{\ev}(X^-_{i,1})&=\begin{cases}
\Phi(X^-_{i,1})-w^{(2)}_{1,n}&\text{ if }i=0,\\
\Phi(X^-_{i,1})-w^{(2)}_{i+1,i}t&\text{ if }i\neq 0.
\end{cases}
\end{align*}

We note that, in the case when we set $c$ as $\alpha_1+\alpha_2$, $\widehat{\mathfrak{gl}}(n)$ in Theorem~\ref{thm:main} is the same as the affinization of $\mathfrak{gl}(m+n)^f$ except of the inner product on the diagonal part. By Theorem~\ref{Tho1}, we can prove that $\widetilde{\ev}$ is compatible with \eqref{Eq2.2}-\eqref{Eq2.9} which are parts of the defining relations of the Guay's affine Yangian $Y_{-1,n+\alpha_1+\alpha_2}(\widehat{\mathfrak{sl}}(n))$ in a way similar to the proof of the existence of the evaluation map (see \cite{K1}). This is summarized as the following lemma.
\begin{Lemma}\label{Claim1211}
Let us set $\tilde{\ve}$ as $n+\alpha_1+\alpha_2$.
Then, $\widetilde{\ev}$ is compatible with \eqref{Eq2.2}-\eqref{Eq2.9} which are parts of the defining relations of the Guay's affine Yangian $Y_{\hbar,\tilde{\ve}}(\widehat{\mathfrak{sl}}(n))$.
\end{Lemma}
We remark that $\widetilde{\ev}$ is not an algebra homomorphism since $[\widetilde{\ev}(H_{i,1}),\widetilde{\ev}(H_{j,1})]$ is not equal to zero. See \eqref{equat2} below for the details.
\begin{Lemma}\label{Claim1234}
For all $i,j\in \{0,1,\cdots,n-1\}$, $\Phi$ is compatible with \eqref{Eq2.3}-\eqref{Eq2.9}.
\end{Lemma}
\begin{proof}
We only prove the compatibility with \eqref{Eq2.6}. Compatibilities with other defining relations are proven in a similar way.
It is enough to show that
\begin{align}
[\Phi(\tilde{H}_{0,1}),\Phi(X^+_{n-1,0})]
&=-\{\Phi(X^+_{n-1,1})-\left(n+\alpha_2-\dfrac{n}{2}\right)w^{(1)}_{n-1,n}\},\label{eqA8}\\
[\Phi(\tilde{H}_{0,1}),\Phi(X^-_{n-1,0})]
&=\{\Phi(X^-_{n-1,1})-\left(n+\alpha_2-\dfrac{n}{2}\right)w^{(1)}_{n,n-1}\}.\label{eqA13}
\end{align}
By the definition of $\Phi$, we can rewrite the left hand side of \eqref{eqA8} as
\begin{align}
&\quad[\Phi(\tilde{H}_{0,1}),\Phi(X^+_{n-1,0})]\nonumber\\
&=-[w^{(1)}_{n-1,n},w^{(2)}_{n,n}t]+[w^{(1)}_{n-1,n},w^{(2)}_{1,1}t]\nonumber\\
&\quad-[w^{(1)}_{n-1,n},\alpha_1 w^{(1)}_{n,n}]-\sum_{x\leq m-n}\limits[w^{(1)}_{n-1,n},\alpha_1 w^{(1)}_{x,x}]+[\widetilde{\ev}(\tilde{H}_{0,1}),\widetilde{\ev}(X^+_{n-1,0})].\label{eqA9}
\end{align}
By Corollary~\ref{COR}, we obtain
\begin{align}
-[w^{(1)}_{n-1,n},w^{(2)}_{n,n}t]&=-w^{(2)}_{n-1,n}t,\label{eqA10}\\
[w^{(1)}_{n-1,n},w^{(2)}_{1,1}t]&=0.
\end{align}
By Theorem~\ref{Tho1}, we have
\begin{gather}
-[w^{(1)}_{n-1,n},\alpha_1 w^{(1)}_{n,n}]=-\alpha_1 w^{(1)}_{n-1,n},\label{eqA11}\\
-\sum_{x\leq m-n}\limits[w^{(1)}_{n-1,n},\alpha_1 W^{(1)}_{x,x}]=0.
\end{gather}
By Lemma~\ref{Claim1211}, we also obtain
\begin{align}
&\quad[\widetilde{\ev}(\tilde{H}_{0,1}),\widetilde{\ev}(X^+_{n-1,0})]\nonumber\\
&=-\widetilde{\ev}(X^+_{n-1,1})+\left(n+\alpha_1+\alpha_2-\dfrac{n}{2}\right)w^{(1)}_{n-1,n}.\label{eqA12}
\end{align}
The identity \eqref{eqA8} follows by applying \eqref{eqA10}-\eqref{eqA12} to \eqref{eqA9}.

Similarly, by the definition of $\Phi$, we obtain
\begin{align}
&\quad[\Phi(\tilde{H}_{0,1}),w^{(1)}_{n,n-1}]\nonumber\\
&=-[w^{(1)}_{n,n-1},w^{(2)}_{n,n}t]+[w^{(1)}_{n,n-1},w^{(2)}_{1,1}t]\nonumber\\
&\quad-[w^{(1)}_{n,n-1},\alpha_1 w^{(1)}_{n,n}]-\sum_{x\leq m-n}\limits[w^{(1)}_{n,n-1},\alpha_1 W^{(1)}_{x,x}]+[\widetilde{\ev}(\tilde{H}_{0,1}),\widetilde{\ev}(X^-_{n-1,0})].\label{eqA14}
\end{align}
By Corollary~\ref{COR}, we obtain
\begin{align}
-[w^{(1)}_{n,n-1},w^{(2)}_{n,n}t]&=w^{(2)}_{n,n-1}t,\label{eqA15}\\
[w^{(1)}_{n,n-1},w^{(2)}_{1,1}t]&=0.\label{eqA16}
\end{align}
By Lemma~\ref{Lem1}, we have
\begin{gather}
-[w^{(1)}_{n,n-1},\alpha_1 w^{(1)}_{n,n}]=\alpha_1 w^{(1)}_{n,n-1},\label{eqA17}\\
-\sum_{x\leq m-n}\limits[w^{(1)}_{n,n-1},\alpha_1 W^{(1)}_{x,x}]=0.
\end{gather}
By Lemma~\ref{Claim1211}, we obtain
\begin{align}
&\quad[\widetilde{\ev}(\tilde{H}_{0,1}),\widetilde{\ev}(X^-_{n-1,0})]\nonumber\\
&=\widetilde{\ev}(X^-_{n-1,1})-\left(n+\alpha_1+\alpha_2-\dfrac{n}{2}\right)w^{(1)}_{n,n-1}.\label{eqA18}
\end{align}
The identity \eqref{eqA13} follows by applying \eqref{eqA15}-\eqref{eqA18} to \eqref{eqA14}.
Thus, we have shown that $\Phi$ is compatible with \eqref{Eq2.6}. 
\end{proof}
Next, we show the compatibility with \eqref{Eq2.1}.
\begin{Lemma}\label{Claim1236}
The following equation holds for all $p,q\in \{0,1,\cdots,n-1\},\ r,s\in\{0,1\}$;
\begin{equation*}
[\Phi(H_{p,r}),\Phi(H_{q,s})]=0.
\end{equation*}
\end{Lemma}
\begin{proof}
By Theorem~\ref{Tho1}, we obtain $[\Phi(H_{p,0}),\Phi(H_{q,0})]=0$. By Theorem~\ref{Tho1}, Corollary~\ref{COR} and the definition of $\Phi$, we obtain $[\Phi(H_{p,0}),\Phi(H_{q,1})]=0$. Thus, it is enough to show that the relation $[\Phi(H_{p,1}),\Phi(H_{q,1})]=0$ holds. 
We only show the case when $p,q\neq0$ and $p>q$. The other case is proven in a similar way. 
In order to simplify the notation, we set
\begin{align*}
X_p&=- \displaystyle\sum_{s \geq 0}  \limits\displaystyle\sum_{u=1}^{p}\limits w^{(1)}_{p,u}t^{-s} w^{(1)}_{u,p}t^s- \displaystyle\sum_{s \geq 0} \limits\displaystyle\sum_{u=p+1}^{n}\limits w^{(1)}_{p,u}t^{-s-1}w^{(1)}_{u,p}t^{s+1}.
\end{align*}
By the definition of $\widetilde{\ev}$, we obtain
\begin{align}
\widetilde{\ev}(H_{p,1})&= \dfrac{p}{2}(W_{p,p}^{(1)}-W_{p+1,p+1}^{(1)}) +W_{p,p}^{(1)}W_{p+1,p+1}^{(1)} + X_{p} - X_{p+1} -(W_{p+1,p+1}^{(1)})^{2}\nonumber\\
&=X_p-X_{p+1}+(\text{the term generated by }\{w^{(1)}_{u,u}t^0|1\leq u\leq n\}).\label{gath0}
\end{align}
By Theorem~\ref{Tho1}, Corollary~\ref{COR} and \eqref{gath0}, we obtain
\begin{gather}
[\widetilde{\ev}(H_{p,1}),\widetilde{\ev}(H_{q,1})]=[X_p-X_{p+1},X_q-X_{q+1}],\label{gath1}\\
[\widetilde{\ev}(H_{p,1}),(w^{(2)}_{q,q}-w^{(2)}_{q+1,q+1})t]=[X_p-X_{p+1},(w^{(2)}_{q,q}-w^{(2)}_{q+1,q+1})t].\label{gath2}
\end{gather}
We remark that $[\widetilde{\ev}(H_{p,1}),\widetilde{\ev}(H_{q,1})]$ is not equal to zero since the inner products on the diagonal parts of $\widehat{\mathfrak{gl}}(n)$ and $\widehat{\mathfrak{gl}(m+n)^f}$ are different.

By \eqref{gath1}, \eqref{gath2}, and the definition of $\Phi$, we obtain
\begin{align*}
&\quad[\Phi(H_{p,1}),\Phi(H_{q,1})]\\
&=[(w^{(2)}_{p,p}-w^{(2)}_{q+1,q+1})t,(w^{(2)}_{q,q}-w^{(2)}_{q+1,q+1})t]+[X_p-X_{p+1},(w^{(2)}_{q,q}-w^{(2)}_{q+1,q+1})t]\\
&\quad+[(w^{(2)}_{p,p}-w^{(2)}_{p+1,p+1})t,X_q-X_{q+1}]+[X_p-X_{p+1},X_q-X_{q+1}].
\end{align*}
Setting $Y(i,j)$ as
\begin{align*}
[w^{(2)}_{i,i}t,w^{(2)}_{j,j}t]+[X_i,w^{(2)}_{j,j}t]+[w^{(2)}_{i,i}t,X_j]+[X_i,X_j],
\end{align*}
we can rewrite $[\Phi(H_{p,1}),\Phi(H_{q,1})]$ as
\begin{equation*}
Y(p,q)-Y(p+1,q)-Y(p,q+1)+Y(p+1,q+1).
\end{equation*}
Thus, it is enough to show that the relation
\begin{align}
[w^{(2)}_{i,i}t,w^{(2)}_{j,j}t]+[X_i,w^{(2)}_{j,j}t]+[w^{(2)}_{i,i}t,X_j]+[X_i,X_j]
&=-\alpha_1(w^{(1)}_{i,i}-w^{(1)}_{j,j})\label{equat}
\end{align}
holds for all $i,j\in \{1,\cdots,n\}$ satisfying $i\geq j$. Let us compute each terms of the left hand side of \eqref{equat}. 

We have already computed the first term of the left hand side of \eqref{equat} in Section 4.
Let us compute the last term of \eqref{equat}. By a computation similar to the proof of the existence of the evaluation map (see Theorem~3.8 in \cite{K1} and Section 3 in \cite{K3}), it is equal to
\begin{align}
[X_i,X_j]&=\sum_{s\geq0}s(w^{(1)}_{i,i}t^{-s}w^{(1)}_{j,j}t^{s}-w^{(1)}_{j,j}t^{-s}w^{(1)}_{i,i}t^{s}).\label{equat2}
\end{align}
Next, let us compute the second term and the third term of the right hand side of \eqref{equat}. By the definition of $X_i$, we obtain
\begin{align}
&\quad-[X_i,w^{(2)}_{j,j}t]\nonumber\\
&=\displaystyle\sum_{s \geq 0}  \limits\displaystyle\sum_{u=1}^{i}\limits w^{(1)}_{i,u}t^{-s} [w^{(1)}_{u,i}t^s,w^{(2)}_{j,j}t]\nonumber\\
&\quad+\displaystyle\sum_{s \geq 0}  \limits\displaystyle\sum_{u=1}^{i}\limits[w^{(1)}_{i,u}t^{-s},w^{(2)}_{j,j}t] w^{(1)}_{u,i}t^s\nonumber\\
&\quad+\displaystyle\sum_{s \geq 0} \limits\displaystyle\sum_{u=i+1}^{n}\limits w^{(1)}_{i,u}t^{-s-1} [w^{(1)}_{u,i}t^{s+1},w^{(2)}_{j,j}t]\nonumber\\
&\quad+\displaystyle\sum_{s \geq 0} \limits\displaystyle\sum_{u=i+1}^{n}\limits [w^{(1)}_{i,u}t^{-s-1},w^{(2)}_{j,j}t] w^{(1)}_{u,i}t^{s+1}.\label{bru}
\end{align}
The first term of the right hand side of \eqref{bru} is equal to
\begin{align}
&\quad\displaystyle\sum_{s \geq 0}  \limits\displaystyle\sum_{u=1}^{i}\limits w^{(1)}_{i,u}t^{-s} [w^{(1)}_{u,i}t^s,w^{(2)}_{j,j}t]\nonumber\\
&=-\delta(j\leq i)\displaystyle\sum_{s \geq 0}  \limits w^{(1)}_{i,j}t^{-s}w^{(2)}_{j,i}t^{s+1}+\delta_{i,j}\displaystyle\sum_{s \geq 0}  \limits\sum_{u=1}^iw^{(1)}_{i,u}t^{-s}w^{(2)}_{u,i}t^{s+1}\nonumber\\
&\quad+\delta(j\leq i)\alpha_1\displaystyle\sum_{s \geq 0}\limits sw^{(1)}_{i,j}t^{-s}w^{(1)}_{j,i}t^{s}+\displaystyle\sum_{s \geq 0}\limits sw^{(1)}_{i,i}t^{-s}w^{(1)}_{j,j}t^{s}\nonumber\\
&\quad+\delta_{i,j}\displaystyle\sum_{s \geq 0}\limits\sum_{x\leq m-n}\limits sw^{(1)}_{i,i}t^{-s}w^{(1)}_{x,x}t^{s}\label{aa1}
\end{align}
since we have
\begin{align*}
&\quad[w^{(1)}_{u,i}t^s,w^{(2)}_{j,j}t]\\
&=-\delta_{u,j}w^{(2)}_{j,i}t^{s+1}+\delta_{i,j}w^{(2)}_{u,i}t^{s+1}+s\delta_{u,j}\alpha_1w^{(1)}_{j,i}t^{s}+s\delta_{u,i}w^{(1)}_{j,j}t^{s}+s\delta_{i,j}\delta_{u,j}\sum_{x\leq m-n}\limits w^{(1)}_{x,x}t^{s}
\end{align*}
by Corollary~\ref{COR}. Similarly to \eqref{aa1}, we rewrite the second, third, and 4-th terms of the right hand side of \eqref{bru}. By \eqref{OPE4}, the second term of the right hand side of \eqref{bru} is equal to
\begin{align}
&\quad\displaystyle\sum_{s \geq 0}  \limits\displaystyle\sum_{u=1}^{i}\limits [w^{(1)}_{i,u}t^{-s},w^{(2)}_{j,j}t] w^{(1)}_{u,i}t^s\nonumber\\
&=-\delta_{i,j}\displaystyle\sum_{s \geq 0}  \limits\sum_{u=1}^iw^{(2)}_{j,u}t^{-s+1}w^{(1)}_{u,i}t^{s}+\delta(j\leq i)\displaystyle\sum_{s \geq 0}\limits w^{(2)}_{i,j}t^{-s+1}w^{(1)}_{j,i}t^{s}\nonumber\\
&\quad-\delta_{i,j}\alpha_1\displaystyle\sum_{s \geq 0}  \limits\sum_{u=1}^isw^{(1)}_{j,u}t^{-s}w^{(1)}_{u,i}t^{s}-\displaystyle\sum_{s \geq 0}\limits sw^{(1)}_{j,j}t^{-s}w^{(1)}_{i,i}t^{s}\nonumber\\
&\quad-\delta_{i,j}\displaystyle\sum_{s \geq 0}\limits\sum_{x\leq m-n}\limits sw^{(1)}_{x,x}t^{-s}w^{(1)}_{i,i}t^{s}.\label{aa2}
\end{align}
By Corollary~\ref{COR}, the third term of the right hand side of \eqref{bru} is equal to
\begin{align}
&\quad\displaystyle\sum_{s \geq 0} \limits\displaystyle\sum_{u=i+1}^{n}\limits w^{(1)}_{i,u}t^{-s-1} [w^{(1)}_{u,i}t^{s+1},w^{(2)}_{j,j}t]\nonumber\\
&=-\delta(j>i)\displaystyle\sum_{s \geq 0}\limits w^{(1)}_{i,j}t^{-s-1}w^{(2)}_{j,i}t^{s+2}+\delta_{i,j}\displaystyle\sum_{s \geq 0}\limits\sum_{u=i+1}^nw^{(1)}_{i,u}t^{-s-1}w^{(2)}_{u,i}t^{s+2}\nonumber\\
&\quad+\delta(j>i)\alpha_1\displaystyle\sum_{s \geq 0}\limits(s+1)w^{(1)}_{i,j}t^{-s-1}w^{(1)}_{j,i}t^{s+1}.\label{aa3}
\end{align}
By Corollary~\ref{COR}, the 4-th term of the right hand side of \eqref{bru} is equal to
\begin{align}
&\quad\displaystyle\sum_{s \geq 0} \limits\displaystyle\sum_{u=i+1}^{n}\limits [w^{(1)}_{i,u}t^{-s-1},w^{(2)}_{j,j}t] w^{(1)}_{u,i}t^{s+1} \nonumber\\
&=\delta(j>i)\displaystyle\sum_{s \geq 0}\limits w^{(2)}_{i,j}t^{-s}w^{(1)}_{j,i}t^{s+1}-\delta_{i,j}\displaystyle\sum_{s \geq 0}\limits\sum_{u=i+1}^nw^{(2)}_{i,u}t^{-s}w^{(1)}_{u,i}t^{s+1}\nonumber\\
&\quad+\delta_{i,j}\alpha_1\displaystyle\sum_{s \geq 0}\limits\sum_{u=i+1}^n\limits(s+1)w^{(1)}_{j,u}t^{-s-1}w^{(1)}_{u,i}t^{s+1}\label{aa4}
\end{align}
We prepare some notations. We denote the $i$-th term of the right hand side of \eqref{aa1} (resp.\ \eqref{aa2}, \eqref{aa3}, \eqref{aa4}) by $\eqref{aa1}_i$ (resp.\ $\eqref{aa2}_i$, $\eqref{aa3}_i$, $\eqref{aa4}_i$). We divide $-[X_i,w^{(2)}_{j,j}t]$ into 6 pieces.
\begin{align*}
A_{i,j}&=\eqref{aa1}_2+\eqref{aa2}_1+\eqref{aa3}_2+\eqref{aa4}_2\\
&=\delta_{i,j}\displaystyle\sum_{s \geq 0}  \limits\sum_{u=1}^iw^{(1)}_{i,u}t^{-s}w^{(2)}_{u,i}t^{s+1}-\delta_{i,j}\displaystyle\sum_{s \geq 0}  \limits\sum_{u=1}^iw^{(2)}_{j,u}t^{-s+1}w^{(1)}_{u,i}t^{s},\\
&\quad+\delta_{i,j}\displaystyle\sum_{s \geq 0}\limits\sum_{u=i+1}^nw^{(1)}_{i,u}t^{-s-1}w^{(2)}_{u,i}t^{s+2}-\delta_{i,j}\displaystyle\sum_{s \geq 0}\limits\sum_{u=i+1}^nw^{(2)}_{i,u}t^{-s-1}w^{(1)}_{u,i}t^{s+1},\\
B_{i,j}&=\eqref{aa1}_1+\eqref{aa2}_2+\eqref{aa3}_1+\eqref{aa4}_1\\
&=-\delta(j\leq i)\displaystyle\sum_{s \geq 0}  \limits w^{(1)}_{i,j}t^{-s}w^{(2)}_{j,i}t^{s+1}+\delta(j\leq i)\displaystyle\sum_{s \geq 0}\limits w^{(2)}_{i,j}t^{-s+1}w^{(1)}_{j,i}t^{s},\\
&\quad-\delta(j>i)\displaystyle\sum_{s \geq 0}\limits w^{(1)}_{i,j}t^{-s-1}w^{(2)}_{j,i}t^{s+2}+\delta(j>i)\displaystyle\sum_{s \geq 0}\limits w^{(2)}_{i,j}t^{-s}w^{(1)}_{j,i}t^{s+1},\\
C_{i,j}&=\eqref{aa1}_3+\eqref{aa3}_3\\
&=\delta(j\leq i)\alpha_1\displaystyle\sum_{s \geq 0}  \limits sw^{(1)}_{i,j}t^{-s}w^{(1)}_{j,i}t^{s}+\delta(j>i)\alpha_1\displaystyle\sum_{s \geq 0}\limits(s+1)w^{(1)}_{i,j}t^{-s-1}w^{(1)}_{j,i}t^{s+1},\\
D_{i,j}&=\eqref{aa1}_4+\eqref{aa2}_4\\
&=\displaystyle\sum_{s \geq 0}\limits sw^{(1)}_{i,i}t^{-s}w^{(1)}_{j,j}t^{s}-\displaystyle\sum_{s \geq 0}\limits sw^{(1)}_{j,j}t^{-s}w^{(1)}_{i,i}t^{s},\nonumber\\
\tilde{E}_{i,j}&=\eqref{aa2}_3+\eqref{aa4}_3\\
&=-\delta_{i,j}\alpha_1\displaystyle\sum_{s \geq 0}  \limits\sum_{u=1}^isw^{(1)}_{j,u}t^{-s}w^{(1)}_{u,i}t^{s}+\delta_{i,j}\alpha_1\displaystyle\sum_{s \geq 0}\limits(s+1)w^{(1)}_{j,u}t^{-s-1}w^{(1)}_{u,i}t^{s+1},\\
F_{i,j}&=\eqref{aa1}_5+\eqref{aa2}_5\\
&=\delta_{i,j}\displaystyle\sum_{s \geq 0}\limits\sum_{x\leq m-n}\limits sw^{(1)}_{i,i}t^{-s}w^{(1)}_{x,x}t^{s}-\delta_{i,j}\displaystyle\sum_{s \geq 0}\limits\sum_{x\leq m-n}\limits sw^{(1)}_{x,x}t^{-s}w^{(1)}_{i,i}t^{s}.
\end{align*}
In order to show \eqref{equat}, it is enough to show the following relations;
\begin{gather}
A_{i,j}-A_{j,i}=0,\label{AA1}\\
k_{j,i}+\tilde{e}_{j,i}+g_{j,i}+h_{j,i}+B_{i,j}-B_{j,i}=\alpha_1(w^{(1)}_{i,i}-w^{(1)}_{j,j}),\label{BB1}\\
a_{j,i}+C_{i,j}-C_{j,i}-=0,\label{CC1}\\
b_{j,i}+D_{i,j}-D_{j,i}-\eqref{equat2}=0,\label{DD1}\\
\tilde{E}_{i,j}-\tilde{E}_{j,i}=0,\label{EE1}\\
F_{i,j}-F_{j,i}=0\label{FF1}
\end{gather}
since -\eqref{equat} is equal to the sum of \eqref{AA1}-\eqref{FF1}.
By the definition of $A_{i,j}$, $\tilde{E}_{i,j}$, and $F_{i,j}$, we obtain
\begin{equation*}
A_{i,j}-A_{j,i}=0,\ \tilde{E}_{i,j}-\tilde{E}_{j,i}=0,\ F_{i,j}-F_{j,i}=0.
\end{equation*}
Moreover, by a direct computation, we also obtain \eqref{DD1}. Thus, it is enough to show that \eqref{BB1} and \eqref{CC1} hold.

First, let us show the relation \eqref{BB1}. In the case when $i=j$, it is enough to show that $k_{j,i}+g_{i,i}+h_{i,i}=0$. Since we obtain
\begin{align*}
&\quad(w^{(2)}_{i,i})_{(-1)}w^{(1)}_{i,i}t^2\\
&=(w^{(1)}_{i,i})_{(-1)}w^{(2)}_{i,i}t^2-\partial(w^{(1)}_{i,i})_{(0)}w^{(2)}_{i,i}t^2+\dfrac{1}{2}\partial^2(w^{(1)}_{i,i})_{(1)}w^{(2)}_{i,i}t^2-\dfrac{1}{6}\partial^3(w^{(1)}_{i,i})_{(2)}w^{(2)}_{i,i}t^2\\
&=(w^{(1)}_{i,i})_{(-1)}w^{(2)}_{i,i}t^2-0+\dfrac{1}{2}\partial^2(\alpha_1w^{(1)}_{i,i}+w^{(1)}_{i,i}+\sum_{x\leq m-n}\limits w^{(1)}_{x,x})t^2-0\\
&=(w^{(1)}_{i,i})_{(-1)}w^{(2)}_{i,i}t^2+\alpha_1w^{(1)}_{i,i}+w^{(1)}_{i,i}+\sum_{x\leq m-n}\limits w^{(1)}_{x,x}
\end{align*}
by the definition of the vertex algebra, we have
\begin{align*}
k_{i,i}&=\alpha_1w^{(1)}_{i,i}+w^{(1)}_{i,i}+\sum_{x\leq m-n}\limits w^{(1)}_{x,x}
\end{align*}
by the definition of $k_{i,j}$. Thus, we obtain $k_{j,i}+g_{i,i}+h_{i,i}=0$.

Let us consider the case when $i>j$. By the assumption that $i>j$, we obtain
\begin{align*}
B_{i,j}&=-\displaystyle\sum_{s \geq 0}  \limits w^{(1)}_{i,j}t^{-s}w^{(2)}_{j,i}t^{s+1}+\displaystyle\sum_{s \geq 0}\limits w^{(2)}_{i,j}t^{-s+1}w^{(1)}_{j,i}t^{s},\\
-B_{j,i}&=\displaystyle\sum_{s \geq 0}\limits w^{(1)}_{j,i}t^{-s-1}w^{(2)}_{i,j}t^{s+2}-\displaystyle\sum_{s \geq 0}\limits w^{(1)}_{j,i}t^{-s}w^{(2)}_{i,j}t^{s+1}.
\end{align*}
By \eqref{k1}, we obtain
\begin{align}
&\quad k_{j,i}+B_{i,j}-B_{j,i}\nonumber\\
&=\sum_{s\geq0}\limits w^{(2)}_{j,i}t^{-1-s}w^{(1)}_{i,j}t^{2+s}+\sum_{s\geq0}\limits w^{(1)}_{i,j}t^{1-s}w^{(2)}_{j,i}t^{s}\nonumber\\
&\quad-\sum_{s\geq0}\limits w^{(1)}_{j,i}t^{-1-s}w^{(2)}_{i,j}t^{2+s}-\sum_{s\geq0}\limits w^{(2)}_{i,j}t^{1-s}w^{(1)}_{j,i}t^{s}\nonumber\\
&\quad-\displaystyle\sum_{s \geq 0}\limits w^{(1)}_{i,j}t^{-s}w^{(2)}_{j,i}t^{s+1}+\displaystyle\sum_{s \geq 0}\limits w^{(2)}_{i,j}t^{-s+1}w^{(1)}_{j,i}t^{s}\nonumber\\
&\quad+\displaystyle\sum_{s \geq 0}\limits w^{(1)}_{j,i}t^{-s-1}w^{(2)}_{i,j}t^{s+2}-\displaystyle\sum_{s \geq 0}\limits w^{(1)}_{j,i}t^{-s}w^{(2)}_{i,j}t^{s+1}.\label{BB2}
\end{align}
Since the sum of the third (resp. 4-th) and 7-th (6-th) terms of the right hand side of \eqref{BB2} vanishes, we have
\begin{align*}
&\quad k_{j,i}+B_{i,j}-B_{j,i}\nonumber\\
&=\sum_{s\geq0}\limits w^{(2)}_{j,i}t^{-1-s}w^{(1)}_{i,j}t^{2+s}+\sum_{s\geq0}\limits w^{(1)}_{i,j}t^{1-s}w^{(2)}_{j,i}t^{s}\nonumber\\
&\quad-\displaystyle\sum_{s \geq 0}  \limits w^{(1)}_{i,j}t^{-s}w^{(2)}_{j,i}t^{s+1}-\displaystyle\sum_{s \geq 0}\limits w^{(1)}_{j,i}t^{-s}w^{(2)}_{i,j}t^{s+1}.
\end{align*}
Since
\begin{align*}
&\quad\sum_{s\geq0}\limits w^{(2)}_{j,i}t^{-1-s}w^{(1)}_{i,j}t^{2+s}+\sum_{s\geq0}\limits w^{(1)}_{i,j}t^{1-s}w^{(2)}_{j,i}t^{s}\\
&=\displaystyle\sum_{s \geq 0}\limits w^{(1)}_{j,i}t^{-s}w^{(2)}_{i,j}t^{s+1}+\displaystyle\sum_{s \geq 0}  \limits w^{(1)}_{i,j}t^{-s}w^{(2)}_{j,i}t^{s+1}\\
&\quad+[w^{(1)}_{i,j}t,w^{(2)}_{j,i}]\\
&=\displaystyle\sum_{s \geq 0}\limits w^{(1)}_{j,i}t^{-s}w^{(2)}_{i,j}t^{s+1}+\displaystyle\sum_{s \geq 0}  \limits w^{(1)}_{i,j}t^{-s}w^{(2)}_{j,i}t^{s+1}\\
&\quad+w^{(2)}_{i,i}t-w^{(2)}_{j,j}t+\alpha_1w^{(1)}_{i,i}+\sum_{x\leq m-n}\limits w^{(1)}_{x,x}
\end{align*}
holds for $i\neq j$ by\eqref{OPE4}, we obtain
\begin{align*}
&\quad k_{j,i}+B_{i,j}-B_{j,i}\\
&=w^{(2)}_{i,i}t-w^{(2)}_{j,j}t+\alpha_1w^{(1)}_{i,i}+\sum_{x\leq m-n}\limits w^{(1)}_{x,x}.
\end{align*}
Thus, we have $k_{j,i}+\tilde{e}_{j,i}+g_{j,i}+h_{j,i}+B_{i,j}-B_{j,i}=\alpha_1(w^{(1)}_{i,i}-w^{(1)}_{j,j})$.

Next, let us show the relation \eqref{CC1}. In the case when $i=j$, it is clear. We consider the case that hat $i>j$.
In this case, we obtain
\begin{align*}
C_{i,j}&=\displaystyle\sum_{s \geq 0}\limits(s+1)\alpha_1w^{(1)}_{i,j}t^{-s-1}w^{(1)}_{j,i}t^{s+1},\\
-C_{j,i}&=-\displaystyle\sum_{s \geq 0}\limits(s+1)\alpha_1w^{(1)}_{j,i}t^{-s-1}w^{(1)}_{i,j}t^{s+1}.
\end{align*}
Then, by a direct computation, we obtain \eqref{CC1}.
This completes the proof of Lemma~\ref{Claim1236}.
\end{proof}
\section{A relationship with the homomorphism given by De Sole-Kac-Valeri}
In this section, we note the relationship between the homomorphism in Theorem~\ref{Main} and the homomorphism defined in \cite{DKV}. First, we recall the definition of a finite $W$-algebra. A finite $W$-algebra $\mathcal{W}^{\text{fin}}(\mathfrak{gl}(m+n),f)$ is the Zhu-algebra (\cite{Zhu}) of $\mathcal{W}^k(\mathfrak{gl}(m+n),f)$. Let us set the degree of $\mathcal{U}(\mathcal{W}^k(\mathfrak{g},f))$ by $\text{deg}(W^{(r)}_{i,j}t^s)=s-r+1$. By Theorem~A.2.11 in \cite{NT}, we find that
\begin{equation*}
\mathcal{W}^{\text{fin}}(\mathfrak{gl}(m+n),f)=\mathcal{U}_0/\sum_{r>0}\limits\mathcal{U}_{-r}\mathcal{U}_r,
\end{equation*}
where $\mathcal{W}^k(\mathfrak{g},f)_r$ is the set of degree $r$-terms of $\mathcal{U}(\mathcal{W}^k(\mathfrak{g},f))$. By the definition of $\Phi$, the image of $\Phi\circ\omega$ is contained in $\mathcal{U}_0$. Then, we can define 
\begin{equation*}
p\circ\Phi\circ\omega\colon Y_\hbar(\mathfrak{sl}(n))\to\mathcal{W}^{\text{fin}}(\mathfrak{gl}(m+n),f)
\end{equation*}
where $p$ is the natural projection from $\mathcal{U}_0$ to $\mathcal{W}^{\text{fin}}(\mathfrak{gl}(m+n),f)$.

Next, let us recall the relationship between the finite Yangian and the Guay's affine Yangian.
\begin{Definition}
Let $\mathfrak{g}$ be a finite dimensional simple Lie algebra and $(a_{i,j})_{1\leq i,j\leq N}$ be the Cartan matrix of $\mathfrak{g}$.
The finite Yangian $Y_{\hbar}(\mathfrak{g})$ is the associative algebra over $\mathbb{C}$ generated by $x_{i,r}^{+}, x_{i,r}^{-}, h_{i,r}$ $(1\leq i\leq N, r \in \mathbb{Z}_{\geq 0})$ with two parameters $\ve_1, \ve_2 \in \mathbb{C}$ subject to the following defining relations:
\begin{gather}
	[h_{i,r}, h_{j,s}] = 0, \\
	[x_{i,r}^{+}, x_{j,s}^{-}] = \delta_{ij} h_{i, r+s}, \\
	[h_{i,0}, x_{j,r}^{\pm}] = \pm a_{ij} x_{j,r}^{\pm},\\
	[h_{i, r+1}, x_{j, s}^{\pm}] - [h_{i, r}, x_{j, s+1}^{\pm}] 
	= \pm a_{ij} \dfrac{\hbar}{2} \{h_{i, r}, x_{j, s}^{\pm}\},\\
	[x_{i, r+1}^{\pm}, x_{j, s}^{\pm}] - [x_{i, r}^{\pm}, x_{j, s+1}^{\pm}] 
	= \pm a_{ij}\dfrac{\hbar}{2} \{x_{i, r}^{\pm}, x_{j, s}^{\pm}\},\\
	\sum_{w \in \mathfrak{S}_{1 + |a_{ij}|}}[x_{i,r_{w(1)}}^{\pm}, [x_{i,r_{w(2)}}^{\pm}, \dots, [x_{i,r_{w(1 + |a_{ij}|)}}^{\pm}, x_{j,s}^{\pm}]\dots]] = 0\  (i \neq j).
\end{gather}
\end{Definition}
Setting $\ve+\dfrac{n}{2}\hbar=0$, there exists a natural embedding $\omega$ from the finite Yangian associated with $\mathfrak{sl}(n)$ to the Guay's affine Yangian determined by
\begin{gather*}
\omega(h_{i,0})=H_{i,0},\qquad\omega(x^\pm_{i,0})=X^\pm_{i,0}.
\end{gather*}
Thus, we obtain the homomorphism
\begin{equation*}
p\circ\Phi\circ\omega\colon Y_\hbar(\mathfrak{sl}(n))\to\mathcal{W}^{\text{fin}}(\mathfrak{gl}(m+n),f)
\end{equation*}
since, by the definition of $\Phi$, the image of $\Phi\circ\omega$ is contained in $\mathcal{U}_0$.

In \cite{DKV}, De Sole-Kac-Valeri constructed a homomorphism from the Yangian associated with $\mathfrak{gl}(n)$ to a finite $W$-algebra $\mathcal{W}^{\text{fin}}(\mathfrak{gl}(m+n),f)$ associated with $\mathfrak{gl}(m+n)$ and a nilpotent element $f$. Let us recall the definition of the Yangian associated with $\mathfrak{gl}(n)$ (\cite{D1} and \cite{D2}).
\begin{Definition}
The Yangian $Y(\mathfrak{gl}(n))$ is an associative superalgebra whose generators are $\{t^{(r)}_{i,j}\mid r\geq0, 1\leq i,j\leq n\}$ and defining relations are
\begin{gather*}
t^{(0)}_{i,j}=\delta_{i,j},\\
[t^{(r+1)}_{i,j},t^{(s)}_{u,v}]-[t^{(r)}_{i,j},t^{(s+1)}_{u,v}]=t^{(r)}_{u,j}t^{(s)}_{i,v}-t^{(s)}_{u,j}t^{(r)}_{i,v}.
\end{gather*}
\end{Definition}
Actually, in the case when $\hbar\neq0$, the finite Yangian $Y_\hbar(\mathfrak{sl}(n))$ is a subalgebra of the Yangian associated with $Y(\mathfrak{gl}(n))$.
For $\hbar\neq0$, the embedding $\iota_\hbar\colon Y_{\hbar}(\mathfrak{sl}(m|n))\to Y(\mathfrak{gl}(m|n))$ is given by 
\begin{gather*}
\iota_\hbar(h_{i,0})=t^{(1)}_{i,i}-t^{(1)}_{i+1,i+1},\quad \iota_\hbar(x^+_{i,0})=t^{(1)}_{i,i+1},\quad \iota_\hbar(x^-_{i,0})=t^{(1)}_{i+1,i},
\end{gather*}
\begin{align*}
\iota_\hbar(h_{i,1})&=
-\hbar t^{(2)}_{i,i}+\hbar t^{(2)}_{i+1,i+1}-\dfrac{i}{2}\hbar(t^{(1)}_{i,i}-t^{(1)}_{i+1,i+1}) -\hbar t^{(1)}_{i,i}t^{(1)}_{i+1,i+1} \\
&\quad+ \hbar\displaystyle\sum_{u=1}^{i}\limits t^{(1)}_{i,u}t^{(1)}_{u,i}-\hbar\displaystyle\sum_{u=1}^{i}\limits t^{(1)}_{i+1,u}t^{(1)}_{u,i+1}.
\end{align*}

In \cite{DKV}, De Sole-Kac-Valeri constructed the following homomorphism
\begin{equation*}
\widetilde{\Phi}\colon Y(\mathfrak{gl}(n))\to\mathcal{W}^{\text{fin}}(\mathfrak{gl}(m+n),f)
\end{equation*}
determined by
\begin{equation*}
\widetilde{\Phi}(t^{(1)}_{i,j})=p(w^{(1)}_{i,j}),\ \widetilde{\Phi}(t^{(2)}_{i,j})=p(w^{(2)}_{i,j}t),\ \widetilde{\Phi}(t^{(r)}_{m-n+i,m-n+j})=0\ (r\geq3).
\end{equation*}
\begin{Remark}
Brundan-Kleshchev \cite{BK} constructed a surjective homomorphism from a shifted Yangian to a finite $W$-algebra of type $A$. In this case, there exists a surjective homomorphism
\begin{equation*}
Y(\mathfrak{gl}(m),\sigma)\to \mathcal{W}^{\text{fin}}(\mathfrak{gl}(m+n),f)
\end{equation*}
where $\sigma$ is an $m\times m$ matrix satisfying
\begin{gather*}
s_{i,j}+s_{j,k}=s_{i,k}\text{ if }|i-j|+|j-k|=|i-k|,\\
s_{i,i}=0,\ 
s_{i,i+1}=\begin{cases}
0\text{ if }i\neq n\\
1\text{ if }i=n,
\end{cases},\ 
s_{i+1,i}=0.
\end{gather*}
The finite Yangian $Y(\mathfrak{gl}(n))$ can be embedded into $Y(\mathfrak{gl}(m),\sigma)$ and the homomorphism given in De Sole-Kac-Valeri \cite{DKV} is a restriction of the homomorphism given in \cite{BK} (See Remark 6.17 in \cite{DKV}).
\end{Remark}
By a direct computation, we obtain the following theorem.
\begin{Theorem}
The homomorphism $\Phi$ induce the homomorphism $\widetilde{\Phi}$;
\begin{equation*}
\widetilde{\Phi}\circ\iota_{-1}=p\circ\Phi\circ\omega.
\end{equation*}
\end{Theorem}
\appendix
\section{the proof of \eqref{OPE3-1}}
This subsection is devoted to the proof of \eqref{OPE3-1}. It is only due to a direct computation. We set the following notation;
\begin{gather*}
e^{(1)}_{i,j}=e_{i,j}\text{ for all }1\leq i,j\leq m,\\
e^{(2)}_{i,j}=e_{n+i,n+j}\text{ for all }m-n+1\leq i\leq m.
\end{gather*}
We denote by $\nu$ the Miura transformation (see \cite{KW1}) By the definition of the Miura transformation, we have
\begin{align*}
\nu(W^{(1)}_{i,j})&=\begin{cases}
e^{(1)}_{i,j}[-1]&\text{ if }1\leq i\leq m-n,\\
e^{(1)}_{i,j}[-1]+e^{(2)}_{i,j}[-1]&\text{ if }m-n<i\leq m,
\end{cases}\\
\nu(W^{(2)}_{i,j})&=\sum_{w>m-n}\limits e_{w,j}^{(1)}[-1]e_{i,w}^{(2)}[-1]-\sum_{w\leq m-n}\limits e_{w,j}^{(1)}[-1]e_{i,w}^{(1)}[-1]-\alpha_2e^{(1)}_{i,j}[-2].
\end{align*}
By \cite{A2}, in order to compute $(W^{(2)}_{p,q})_{(0)}W^{(2)}_{i,j}$ it is enough to  compute $(\nu(W^{(2)}_{p,q}))_{(0)}\nu(W^{(2)}_{i,j})$. In order to simplify the notation, we sometimes denote $e^{(r)}_{i,j}[-1]$ by $e^{(r)}_{i,j}$.
In the appendix, we denote the $i$-th non-zero term of $(\text{equation number})$ by $(\text{equation number})_i$.
\begin{proof}[proof of \eqref{OPE3-1}]
By the definition of $W^{(2)}_{i,j}$, we have
\begin{align*}
&\quad(W^{(2)}_{p,q})_{(0)}W^{(2)}_{i,j}\\
&=(\sum_{x>m-n}\limits e_{x,q}^{(1)}e_{p,x}^{(2)})_{(0)}(\sum_{w>m-n}\limits e_{w,j}^{(1)}e_{i,w}^{(2)}-\sum_{w\leq m-n}\limits e_{w,j}^{(1)}e_{i,w}^{(1)}-\alpha_2e^{(1)}_{i,j}[-2])\\
&\quad-(\sum_{x\leq m-n}\limits e_{x,q}^{(1)}e_{p,x}^{(1)})_{(0)}(\sum_{w>m-n}\limits e_{w,j}^{(1)}e_{i,w}^{(2)}-\sum_{w\leq m-n}\limits e_{w,j}^{(1)}e_{i,w}^{(1)}-\alpha_2e^{(1)}_{i,j}[-2]).
\end{align*}
We compute these terms respectively. By a direct computation, we obtain
\begin{align}
&\quad(\sum_{x>m-n}\limits e_{x,q}^{(1)}e_{p,x}^{(2)})_{(0)}\sum_{w>m-n}\limits e_{w,j}^{(1)}e_{i,w}^{(2)}\nonumber\\
&=\sum_{x,w>m-n}\limits e_{p,x}^{(2)}[e_{x,q}^{(1)},e_{w,j}^{(1)}]e_{i,w}^{(2)}+\sum_{x,w>m-n}\limits e_{x,q}^{(1)}e_{w,j}^{(1)}[e_{p,x}^{(2)},e_{i,w}^{(2)}]\nonumber\\
&\quad+\sum_{x,w>m-n}\limits e_{p,x}^{(2)}[-2](e_{x,q}^{(1)},e_{w,j}^{(1)})e_{i,w}^{(2)}+\sum_{x,w>m-n}\limits e_{x,q}^{(1)}[-2]e_{w,j}^{(1)}(e_{p,x}^{(2)},e_{i,w}^{(2)})\nonumber\\
&=\sum_{x,w>m-n}\limits e_{p,x}^{(2)}(\delta_{q,w}e^{(1)}_{x,j}-\delta_{j,x}e^{(1)}_{w,q})e_{i,w}^{(2)}+\sum_{x,w>m-n}\limits e_{x,q}^{(1)}e_{w,j}^{(1)}(\delta_{x,i}e^{(2)}_{p,w}-\delta_{w,p}e^{(2)}_{i,x})\nonumber\\
&\quad+\sum_{x,w>m-n}\limits e_{p,x}^{(2)}[-2](\alpha_1\delta_{x,j}\delta_{q,w}+\delta_{x,q}\delta_{w,j})e_{i,w}^{(2)}+\sum_{x,w>m-n}\limits e_{x,q}^{(1)}[-2]e_{w,j}^{(1)}(\alpha_2\delta_{x,i}\delta_{w,p}+\delta_{p,x}\delta_{i,w})\nonumber\\
&=\delta(q>m-n)\sum_{x>m-n}\limits e_{p,x}^{(2)}e^{(1)}_{x,j}e_{i,q}^{(2)}-\delta(j>m-n)\sum_{w>m-n}\limits e_{p,j}^{(2)}e^{(1)}_{w,q}e_{i,w}^{(2)}\nonumber\\
&\quad+\sum_{w>m-n}\limits e_{i,q}^{(1)}e_{w,j}^{(1)}e^{(2)}_{p,w}-\sum_{x>m-n}\limits e_{x,q}^{(1)}e_{p,j}^{(1)}e^{(2)}_{i,x}\nonumber\\
&\quad+\delta(q,j>m-n)\alpha_1e_{p,j}^{(2)}[-2]e_{i,q}^{(2)}+\delta(q,j>m-n)e_{p,q}^{(2)}[-2]e_{i,j}^{(2)}\nonumber\\
&\quad+\alpha_2e_{i,q}^{(1)}[-2]e_{p,j}^{(1)}+e_{p,q}^{(1)}[-2]e_{i,j}^{(1)}.\label{1-1}
\end{align}
By a direct computation, we obtain
\begin{align}
&\quad-(\sum_{x>m-n}\limits e_{x,q}^{(1)}e_{p,x}^{(2)})_{(0)}\sum_{w\leq m-n}\limits e_{w,j}^{(1)}e_{i,w}^{(1)}\nonumber\\
&=-\sum_{\substack{x>m-n\\w\leq m-n}}\limits e_{p,x}^{(2)}[e_{x,q}^{(1)},e_{w,j}^{(1)}]e_{i,w}^{(1)}-\sum_{\substack{x>m-n\\w\leq m-n}}\limits e_{p,x}^{(2)}e_{w,j}^{(1)}[e_{x,q}^{(1)},e_{i,w}^{(1)}]\nonumber\\
&\quad-\sum_{\substack{x>m-n\\w\leq m-n}}\limits e_{p,x}^{(2)}[-2](e_{x,q}^{(1)},e_{w,j}^{(1)})e_{i,w}^{(1)}-\sum_{\substack{x>m-n\\w\leq m-n}}\limits e_{p,x}^{(2)}[-2]e_{w,j}^{(1)}(e_{x,q}^{(1)},e_{i,w}^{(1)})\nonumber\\
&\quad-\sum_{\substack{x>m-n\\w\leq m-n}}\limits e_{p,x}^{(2)}[-2][[e_{x,q}^{(1)},e_{w,j}^{(1)}],e_{i,w}^{(1)}]-\sum_{\substack{x>m-n\\w\leq m-n}}\limits e_{p,x}^{(2)}[-3]([e_{x,q}^{(1)},e_{w,j}^{(1)}],e_{i,w}^{(1)})\nonumber\\
&=-\sum_{\substack{x>m-n\\w\leq m-n}}\limits e_{p,x}^{(2)}(\delta_{q,w}e^{(1)}_{x,j}-\delta_{j,x}e^{(1)}_{w,q})e_{i,w}^{(1)}-\sum_{\substack{x>m-n\\w\leq m-n}}\limits e_{p,x}^{(2)}e_{w,j}^{(1)}(\delta_{q,i}e_{w,x}^{(1)}-\delta_{x,w}e^{(1)}_{i,q})\nonumber\\
&\quad-\sum_{\substack{x>m-n\\w\leq m-n}}\limits e_{p,x}^{(2)}[-2](\alpha_1\delta_{x,j}\delta_{q,w}+\delta_{x,q}\delta_{w,j})e_{i,w}^{(1)}-\sum_{\substack{x>m-n\\w\leq m-n}}\limits e_{p,x}^{(2)}[-2]e_{w,j}^{(1)}(\alpha_1\delta_{x,w}\delta_{q,i}+\delta_{x,q}\delta_{i,w})\nonumber\\
&\quad-\sum_{\substack{x>m-n\\w\leq m-n}}\limits e_{p,x}^{(2)}[-2][(\delta_{q,w}e^{(1)}_{x,j}-\delta_{j,x}e^{(1)}_{w,q}),e_{i,w}^{(1)}]\nonumber\\
&\quad-\sum_{\substack{x>m-n\\w\leq m-n}}\limits \delta_{q,w}e_{p,x}^{(2)}[-3](\delta_{q,w}e^{(1)}_{x,j}-\delta_{j,x}e^{(1)}_{w,q},e^{(1)}_{i,w})\nonumber\\
&=-\sum_{\substack{x>m-n\\w\leq m-n}}\limits e_{p,x}^{(2)}(\delta_{q,w}e^{(1)}_{x,j}-\delta_{j,x}e^{(1)}_{w,q})e_{i,w}^{(1)}-\sum_{\substack{x>m-n\\w\leq m-n}}\limits e_{p,x}^{(2)}e_{w,j}^{(1)}(\delta_{q,i}e_{w,x}^{(1)}-\delta_{x,w}e^{(1)}_{i,q})\nonumber\\
&\quad-\sum_{\substack{x>m-n\\w\leq m-n}}\limits e_{p,x}^{(2)}[-2](\alpha_1\delta_{x,j}\delta_{q,w}+\delta_{x,q}\delta_{w,j})e_{i,w}^{(1)}-\sum_{\substack{x>m-n\\w\leq m-n}}\limits e_{p,x}^{(2)}[-2]e_{w,j}^{(1)}(\alpha_1\delta_{x,w}\delta_{q,i}+\delta_{x,q}\delta_{i,w})\nonumber\\
&\quad-\sum_{\substack{x>m-n\\w\leq m-n}}\limits \delta_{q,w}e_{p,x}^{(2)}[-2](\delta_{j,i}e^{(1)}_{x,w}-\delta_{x,w}e^{(1)}_{i,j})+\sum_{\substack{x>m-n\\w\leq m-n}}\limits \delta_{j,x}e_{p,x}^{(2)}[-2](\delta_{q,i}e^{(1)}_{w,w}-e^{(1)}_{i,q})\nonumber\\
&\quad-\sum_{\substack{x>m-n\\w\leq m-n}}\limits \delta_{q,w}e_{p,x}^{(2)}[-3](\delta_{q,w}e^{(1)}_{x,j}-\delta_{j,x}e^{(1)}_{w,q},e^{(1)}_{i,w})\nonumber\\
&\quad+\sum_{\substack{x>m-n\\w\leq m-n}}\limits \delta_{x,j}e_{p,x}^{(2)}[-3](\delta_{i,q}\alpha_1+\delta_{w,q}\delta_{i,w})\nonumber\\
&=-\delta(q\leq m-n)\sum_{\substack{x>m-n}}\limits e_{p,x}^{(2)}e^{(1)}_{x,j}e_{i,q}^{(1)}-\delta(j>m-n)\sum_{\substack{w\leq m-n}}\limits e_{p,j}^{(2)}e^{(1)}_{w,q}e_{i,w}^{(1)}\nonumber\\
&\quad-\delta_{q,i}\sum_{\substack{x>m-n\\w\leq m-n}}\limits e_{p,x}^{(2)}e_{w,j}^{(1)}e_{w,x}^{(1)}+0\nonumber\\
&\quad-\delta(j>m-n)\delta(q\leq m-n)\alpha_1e_{p,j}^{(2)}[-2]e_{i,q}^{(1)}-\delta(q>m-n)\delta(j\leq m-n)e_{p,q}^{(2)}[-2]e_{i,j}^{(1)}\nonumber\\
&\quad-0-0-\delta_{j,i}\delta(q\leq m-n)\sum_{\substack{x>m-n}}\limits e_{p,x}^{(2)}[-2]e^{(1)}_{x,q}-0\nonumber\\
&\quad+\delta_{q,i}\delta(j>m-n)\sum_{\substack{w\leq m-n}}\limits e_{p,j}^{(2)}[-2]e^{(1)}_{w,w}-\delta(j>m-n)(m-n)e_{p,j}^{(2)}[-2]e^{(1)}_{i,q}\nonumber\\
&\quad+0+0+\delta_{q,i}\delta(j>m-n)\alpha_1(m-n)e_{p,j}^{(2)}[-3]+0.\label{1-2}
\end{align}
By a direct computation, we obtain
\begin{align}
&\quad-\alpha_2(\sum_{x>m-n}\limits e_{x,q}^{(1)}e_{p,x}^{(2)})_{(0)}e^{(1)}_{i,j}[-2]\nonumber\\
&=-\alpha_2\sum_{x>m-n}\limits e_{p,x}^{(2)}[e_{x,q}^{(1)},e^{(1)}_{i,j}][-2]-\alpha_2\sum_{x>m-n}\limits e_{p,x}^{(2)}[-2][e_{x,q}^{(1)},e^{(1)}_{i,j}]\nonumber\\
&\quad-2\alpha_2\sum_{x>m-n}\limits e_{p,x}^{(2)}[-3](e_{x,q}^{(1)},e^{(1)}_{i,j})\nonumber\\
&=-\alpha_2\sum_{x>m-n}\limits e_{p,x}^{(2)}(\delta_{q,i}e^{(1)}_{x,j}-\delta_{x,j}e^{(1)}_{i,q})[-2]-\alpha_2\sum_{x>m-n}\limits e_{p,x}^{(2)}[-2](\delta_{q,i}e^{(1)}_{x,j}-\delta_{x,j}e^{(1)}_{i,q})\nonumber\\
&\quad-2\alpha_2\sum_{x>m-n}\limits e_{p,x}^{(2)}[-3](\delta_{x,j}\delta_{q,i}\alpha_1+\delta_{x,q}\delta_{i,j})\nonumber\\
&=-\delta_{q,i}\alpha_2\sum_{x>m-n}\limits e_{p,x}^{(2)}e^{(1)}_{x,j}[-2]+\delta(j>m-n)\alpha_2e_{p,j}^{(2)}e^{(1)}_{i,q}[-2]\nonumber\\
&\quad-\delta_{q,i}\alpha_2\sum_{x>m-n}\limits e_{p,x}^{(2)}[-2]e^{(1)}_{x,j}+\delta(j>m-n)\alpha_2e_{p,j}^{(2)}[-2]e^{(1)}_{i,q}\nonumber\\
&\quad-2\delta_{q,i}\delta(j>m-n)\alpha_1\alpha_2e_{p,j}^{(2)}[-3]-2\delta_{i,j}\delta(q>m-n)\alpha_2 e_{p,q}^{(2)}[-3].\label{1-3}
\end{align}
By a direct computation, we obtain
\begin{align}
&\quad-(\sum_{x\leq m-n}\limits e_{x,q}^{(1)}e_{p,x}^{(1)})_{(0)}\sum_{w>m-n}\limits e_{w,j}^{(1)}e_{i,w}^{(2)}\nonumber\\
&=-\sum_{\substack{x\leq m-n\\w>m-n}}\limits e_{x,q}^{(1)}[e_{p,x}^{(1)},e_{w,j}^{(1)}]e_{i,w}^{(2)}-\sum_{\substack{x\leq m-n\\w>m-n}}\limits e_{p,x}^{(1)}[e_{x,q}^{(1)},e_{w,j}^{(1)}]e_{i,w}^{(2)}\nonumber\\
&\quad-\sum_{\substack{x\leq m-n\\w>m-n}}\limits e_{x,q}^{(1)}[-2](e_{p,x}^{(1)},e_{w,j}^{(1)})e_{i,w}^{(2)}-\sum_{\substack{x\leq m-n\\w>m-n}}\limits e_{p,x}^{(1)}[-2](e_{x,q}^{(1)},e_{w,j}^{(1)})e_{i,w}^{(2)}\nonumber\\
&=-\sum_{\substack{x\leq m-n\\w>m-n}}\limits e_{x,q}^{(1)}(\delta_{x,w}e^{(1)}_{p,j}-\delta_{p,j}e^{(1)}_{w,x})e_{i,w}^{(2)}-\sum_{\substack{x\leq m-n\\w>m-n}}\limits e_{p,x}^{(1)}(\delta_{q,w}e^{(1)}_{x,j}-\delta_{j,x}e^{(1)}_{w,q})e_{i,w}^{(2)}\nonumber\\
&\quad-\sum_{\substack{x\leq m-n\\w>m-n}}\limits e_{x,q}^{(1)}[-2](\alpha_1\delta_{p,j}\delta_{x.w}+\delta_{p,x}\delta_{w,j})e_{i,w}^{(2)}-\sum_{\substack{x\leq m-n\\w>m-n}}\limits e_{p,x}^{(1)}[-2](\delta_{q,w}\delta_{x,j}\alpha_1+\delta_{x,q}\delta_{w,j})e_{i,w}^{(2)}\nonumber\\
&=0+\delta_{p,j}\sum_{\substack{x\leq m-n\\w>m-n}}\limits e_{x,q}^{(1)}e^{(1)}_{w,x}e_{i,w}^{(2)}\nonumber\\
&\quad-\delta(q>m-n)\sum_{\substack{x\leq m-n}}\limits e_{p,x}^{(1)}e^{(1)}_{x,j}e_{i,q}^{(2)}+\delta(j\leq m-n)\sum_{\substack{w>m-n}}\limits e_{p,j}^{(1)}e^{(1)}_{w,q}e_{i,w}^{(2)}\nonumber\\
&\quad-0-0-\delta(q>m-n,j\leq m-n)\alpha_1e_{p,j}^{(1)}[-2]e_{i,q}^{(2)}-\delta(q\leq m-n,j>m-n)e_{p,q}^{(1)}[-2]e_{i,j}^{(2)}.\label{1-4}
\end{align}
By a direct computation, we obtain
\begin{align}
&\quad(\sum_{x\leq m-n}\limits e_{x,q}^{(1)}e_{p,x}^{(1)})_{(0)}\sum_{w\leq m-n}\limits e_{w,j}^{(1)}e_{i,w}^{(1)}\nonumber\\
&=\sum_{x,w\leq m-n}\limits e_{x,q}^{(1)}[-2][e_{p,x}^{(1)},e_{w,j}^{(1)}]e_{i,w}^{(1)}+\sum_{x,w\leq m-n}\limits e_{x,q}^{(1)}[-2]e_{w,j}^{(1)}[e_{p,x}^{(1)},e_{i,w}^{(1)}]\nonumber\\
&\quad+\sum_{x,w\leq m-n}\limits e_{x,q}^{(1)}[-2](e_{p,x}^{(1)},e_{w,j}^{(1)})e_{i,w}^{(1)}+\sum_{x,w\leq m-n}\limits e_{x,q}^{(1)}[-2]e_{w,j}^{(1)}(e_{p,x}^{(1)},e_{i,w}^{(1)})\nonumber\\
&\quad+\sum_{x,w\leq m-n}\limits e_{x,q}^{(1)}[-2][[e_{p,x}^{(1)},e_{w,j}^{(1)}],e_{i,w}^{(1)}]\nonumber\\
&\quad+\sum_{x,w\leq m-n}\limits e_{p,x}^{(1)}[e_{x,q}^{(1)},e_{w,j}^{(1)}]e_{i,w}^{(1)}+\sum_{x,w\leq m-n}\limits e_{p,x}^{(1)}e_{w,j}^{(1)}[e_{x,q}^{(1)},e_{i,w}^{(1)}]\nonumber\\
&\quad+\sum_{x,w\leq m-n}\limits e_{p,x}^{(1)}[-2](e_{x,q}^{(1)},e_{w,j}^{(1)})e_{i,w}^{(1)}+\sum_{x,w\leq m-n}\limits e_{p,x}^{(1)}[-2]e_{w,j}^{(1)}(e_{x,q}^{(1)},e_{i,w}^{(1)})\nonumber\\
&\quad+\sum_{x,w\leq m-n}\limits e_{p,x}^{(1)}[-2][[e_{x,q}^{(1)},e_{w,j}^{(1)}],e_{i,w}^{(1)}]\nonumber\\
&\quad+\sum_{\substack{x,w\leq m-n}}\limits e_{p,x}^{(1)}[-3]([e_{x,q}^{(1)},e_{w,j}^{(1)}],e_{i,w}^{(1)})+\sum_{\substack{x,w\leq m-n}}\limits e_{x,q}^{(1)}[-3]([e_{p,x}^{(1)},e_{w,j}^{(1)}],e_{i,w}^{(1)})\nonumber\\
&=\sum_{x,w\leq m-n}\limits e_{x,q}^{(1)}(\delta_{x,w}e^{(1)}_{p,j}-\delta_{p,j}e^{(1)}_{w,x})e_{i,w}^{(1)}+\sum_{x,w\leq m-n}\limits e_{x,q}^{(1)}e_{w,j}^{(1)}(\delta_{x,i}e^{(1)}_{p,w}-\delta_{p,w}e^{(1)}_{i,x})\nonumber\\
&\quad+\sum_{x,w\leq m-n}\limits e_{x,q}^{(1)}[-2](\alpha_1\delta_{p,j}\delta_{x,w}+\delta_{p,x}\delta_{w,j})e_{i,w}^{(1)}+\sum_{x,w\leq m-n}\limits e_{x,q}^{(1)}[-2]e_{w,j}^{(1)}(\delta_{x,i}\delta_{w,p}\alpha_1+\delta_{p,x}\delta_{i,w})\nonumber\\
&\quad+\sum_{x,w\leq m-n}\limits e_{x,q}^{(1)}[-2][\delta_{x,w}e^{(1)}_{p,j}-\delta_{p,j}e^{(1)}_{w,x},e_{i,w}^{(1)}]\nonumber\\
&\quad+\sum_{x,w\leq m-n}\limits e_{p,x}^{(1)}(\delta_{q,w}e^{(1)}_{x,j}-\delta_{x,j}e^{(1)}_{w,q})e_{i,w}^{(1)}+\sum_{x,w\leq m-n}\limits e_{p,x}^{(1)}e_{w,j}^{(1)}(\delta_{q,i}e^{(1)}_{x,w}-\delta_{x,w}e^{(1)}_{i,q})\nonumber\\
&\quad+\sum_{x,w\leq m-n}\limits e_{p,x}^{(1)}[-2](\alpha_1\delta_{q,w}\delta_{x,j}+\delta_{x,q}\delta_{w,j})e_{i,w}^{(1)}+\sum_{x,w\leq m-n}\limits e_{p,x}^{(1)}[-2]e_{w,j}^{(1)}(\delta_{q,i}\delta_{w,x}\alpha_1+\delta_{w,i}\delta_{x,q})\nonumber\\
&\quad+\sum_{x,w\leq m-n}\limits e_{p,x}^{(1)}[-2][(\delta_{q,w}e^{(1)}_{x,j}-\delta_{x,j}e^{(1)}_{w,q}),e_{i,w}^{(1)}]\nonumber\\
&\quad+\sum_{\substack{x,w\leq m-n}}\limits e_{p,x}^{(1)}[-3](\delta_{q,w}e^{(1)}_{x,j}-\delta_{x,j}e^{(1)}_{w,q},e_{i,w}^{(1)})\nonumber\\
&\quad+\sum_{\substack{x,w\leq m-n}}\limits e_{x,q}^{(1)}[-3](\delta_{x,w}e^{(1)}_{p,j}-\delta_{p,j}e^{(1)}_{w,x},e_{i,w}^{(1)})\nonumber\\
&=\sum_{x,w\leq m-n}\limits e_{x,q}^{(1)}(\delta_{x,w}e^{(1)}_{p,j}-\delta_{p,j}e^{(1)}_{w,x})e_{i,w}^{(1)}+\sum_{x,w\leq m-n}\limits e_{x,q}^{(1)}e_{w,j}^{(1)}(\delta_{x,i}e^{(1)}_{p,w}-\delta_{p,w}e^{(1)}_{i,x})\nonumber\\
&\quad+\sum_{x,w\leq m-n}\limits e_{x,q}^{(1)}[-2](\alpha_1\delta_{p,j}\delta_{x,w}+\delta_{p,x}\delta_{w,j})e_{i,w}^{(1)}+\sum_{x,w\leq m-n}\limits e_{x,q}^{(1)}[-2]e_{w,j}^{(1)}(\delta_{x,i}\delta_{w,p}\alpha_1+\delta_{p,x}\delta_{i,w})\nonumber\\
&\quad+\sum_{x,w\leq m-n}\limits \delta_{x,w}e_{x,q}^{(1)}[-2](\delta_{i,j}e^{(1)}_{p,w}-\delta_{p,w}e^{(1)}_{i,j})-\sum_{x,w\leq m-n}\limits\delta_{p,j} e_{x,q}^{(1)}[-2](\delta_{i,x}e^{(1)}_{w,w}-e^{(1)}_{i,x})\nonumber\\
&\quad+\sum_{x,w\leq m-n}\limits e_{p,x}^{(1)}(\delta_{q,w}e^{(1)}_{x,j}-\delta_{x,j}e^{(1)}_{w,q})e_{i,w}^{(1)}+\sum_{x,w\leq m-n}\limits e_{p,x}^{(1)}e_{w,j}^{(1)}(\delta_{q,i}e^{(1)}_{x,w}-\delta_{x,w}e^{(1)}_{i,q})\nonumber\\
&\quad+\sum_{x,w\leq m-n}\limits e_{p,x}^{(1)}[-2](\alpha_1\delta_{q,w}\delta_{x,j}+\delta_{x,q}\delta_{w,j})e_{i,w}^{(1)}+\sum_{x,w\leq m-n}\limits e_{p,x}^{(1)}[-2]e_{w,j}^{(1)}(\delta_{q,i}\delta_{w,x}\alpha_1+\delta_{w,i}\delta_{x,q})\nonumber\\
&\quad+\sum_{x,w\leq m-n}\limits \delta_{q,w}e_{p,x}^{(1)}[-2](\delta_{i,j}e^{(1)}_{x,w}-\delta_{x,w}e^{(1)}_{i,j})-\sum_{x,w\leq m-n}\limits \delta_{x,j}e_{p,x}^{(1)}[-2](\delta_{i,q}e^{(1)}_{w,w}-e^{(1)}_{i,q})\nonumber\\
&\quad+\sum_{\substack{x,w\leq m-n}}\limits \delta_{q,w}e_{p,x}^{(1)}[-3](\alpha_1\delta_{i,j}\delta_{x,w}+\delta_{x,j}\delta_{i,w})\nonumber\\
&\quad-\sum_{\substack{x,w\leq m-n}}\limits \delta_{x,j}e_{p,x}^{(1)}[-3](\delta_{i,q}\alpha_1+\delta_{w,q}\delta_{i,w})\nonumber\\
&\quad+\sum_{\substack{x,w\leq m-n}}\limits \delta_{x,w}e_{x,q}^{(1)}[-3](\delta_{j,i}\delta_{p,w}\alpha_1+\delta_{p,j}\delta_{i,w})\nonumber\\
&\quad-\sum_{\substack{x,w\leq m-n}}\limits \delta_{p,j}e_{x,q}^{(1)}[-3](\delta_{i,x}\alpha_1+\delta_{x,w}\delta_{i,w})\nonumber\\
&=\sum_{x\leq m-n}\limits e_{x,q}^{(1)}e^{(1)}_{p,j}e_{i,x}^{(1)}-\delta_{p,j}\sum_{x,w\leq m-n}\limits e_{x,q}^{(1)}e^{(1)}_{w,x}e_{i,w}^{(1)}\nonumber\\
&\quad+0-0+\delta_{p,j}\alpha_1\sum_{x\leq m-n}\limits e_{x,q}^{(1)}[-2]e_{i,x}^{(1)}+0+0+0\nonumber\\
&\quad+\delta_{i,j}\sum_{x\leq m-n}\limits e_{x,q}^{(1)}[-2]e^{(1)}_{p,x}-0\nonumber\\
&\quad+0+\delta_{p,j}(m-n)\sum_{x\leq m-n}\limits e_{x,q}^{(1)}[-2]e_{i,x}^{(1)}\nonumber\\
&\quad+\delta(q\leq m-n)\sum_{x\leq m-n}\limits e_{p,x}^{(1)}e^{(1)}_{x,j}e_{i,q}^{(1)}-\delta(j\leq m-n)\sum_{w\leq m-n}\limits e_{p,j}^{(1)}e^{(1)}_{w,q}e_{i,w}^{(1)}\nonumber\\
&\quad+\delta_{q,i}\sum_{x,w\leq m-n}\limits e_{p,x}^{(1)}e_{w,j}^{(1)}e^{(1)}_{x,w}-\sum_{x\leq m-n}\limits e_{p,x}^{(1)}e_{x,j}^{(1)}e^{(1)}_{i,q}\nonumber\\
&\quad+\delta(q,j\leq m-n)\alpha_1\sum_{x,w\leq m-n}\limits e_{p,j}^{(1)}[-2]e_{i,q}^{(1)}+\delta(q,j\leq m-n)e_{p,q}^{(1)}[-2]e_{i,j}^{(1)}\nonumber\\
&\quad+\delta_{q,i}\alpha_1\sum_{x\leq m-n}\limits e_{p,x}^{(1)}[-2]e_{x,j}^{(1)}+0\nonumber\\
&\quad+\delta_{i,j}\delta(q\leq m-n)\sum_{x\leq m-n}\limits e_{p,x}^{(1)}[-2]e^{(1)}_{x,q}-\delta(q\leq m-n)e_{p,q}^{(1)}[-2]e^{(1)}_{i,j}\nonumber\\
&\quad-\delta_{i,q}\delta(j\leq m-n)\sum_{w\leq m-n}\limits e_{p,j}^{(1)}[-2]e^{(1)}_{w,w}+\delta(j\leq m-n)(m-n)e_{p,j}^{(1)}[-2]e^{(1)}_{i,q}\nonumber\\
&\quad+\delta_{i,j}\delta(q\leq m-n)\alpha_1e_{p,q}^{(1)}[-3]+0\nonumber\\
&\quad-\delta_{i,q}\delta(j\leq m-n)\alpha_1(m-n)e_{p,j}^{(1)}[-3]+0\nonumber\\
&\quad+0+0-0-0.\label{1-5}
\end{align}
By a direct computation, we obtain
\begin{align}
&\quad(\sum_{x\leq m-n}\limits e_{x,q}^{(1)}e_{p,x}^{(1)})_{(0)}\alpha_2e^{(1)}_{i,j}[-2]\nonumber\\
&=\alpha_2\sum_{x\leq m-n}\limits e_{x,q}^{(1)}[e_{p,x}^{(1)},e^{(1)}_{i,j}][-2]+\alpha_2\sum_{x\leq m-n}\limits e_{x,q}^{(1)}[-2][e_{p,x}^{(1)},e^{(1)}_{i,j}]\nonumber\\
&\quad+2\alpha_2\sum_{x\leq m-n}\limits e_{x,q}^{(1)}[-3](e_{p,x}^{(1)},e^{(1)}_{i,j})\nonumber\\
&\quad+\alpha_2\sum_{x\leq m-n}\limits e_{p,x}^{(1)}[e_{x,q}^{(1)},e^{(1)}_{i,j}][-2]+\alpha_2\sum_{x\leq m-n}\limits e_{p,x}^{(1)}[-2][e_{x,q}^{(1)},e^{(1)}_{i,j}]\nonumber\\
&\quad+2\alpha_2\sum_{x\leq m-n}\limits e_{p,x}^{(1)}[-3](e_{x,q}^{(1)},e^{(1)}_{i,j})\nonumber\\
&=\alpha_2\sum_{x\leq m-n}\limits e_{x,q}^{(1)}(\delta_{x,i}e^{(1)}_{p,j}-\delta_{p,j}e^{(1)}_{i,x})[-2]+\alpha_2\sum_{x\leq m-n}\limits e_{x,q}^{(1)}[-2](\delta_{x,i}e^{(1)}_{p,j}-\delta_{p,j}e^{(1)}_{i,x})\nonumber\\
&\quad+2\alpha_2\sum_{x\leq m-n}\limits e_{x,q}^{(1)}[-3](\alpha_1\delta_{p,j}\delta_{x,i}\alpha_1+\delta_{p,x}\delta_{i,j})\nonumber\\
&\quad+\alpha_2\sum_{x\leq m-n}\limits e_{p,x}^{(1)}(\delta_{q,i}e^{(1)}_{x,j}-\delta_{x,j}e^{(1)}_{i,q})[-2]+\alpha_2\sum_{x\leq m-n}\limits e_{p,x}^{(1)}[-2](\delta_{q,i}e^{(1)}_{x,j}-\delta_{x,j}e^{(1)}_{i,q})\nonumber\\
&\quad+2\alpha_2\sum_{x\leq m-n}\limits e_{p,x}^{(1)}[-3](\alpha_1\delta_{q,i}\delta_{x,j}+\delta_{x,q}\delta_{i,j})\nonumber\\
&=0-\delta_{p,j}\alpha_2\sum_{x\leq m-n}\limits e_{x,q}^{(1)}e^{(1)}_{i,x}[-2]+0-\delta_{p,j}\alpha_2\sum_{x\leq m-n}\limits e_{x,q}^{(1)}[-2]e^{(1)}_{i,x}\nonumber\\
&\quad+0+0\nonumber\\
&\quad+\delta_{q,i}\alpha_2\sum_{x\leq m-n}\limits e_{p,x}^{(1)}e^{(1)}_{x,j}[-2]-\delta(j\leq m-n)\alpha_2 e_{p,j}^{(1)}e^{(1)}_{i,q}[-2]\nonumber\\
&\quad+\delta_{q,i}\alpha_2\sum_{x\leq m-n}\limits e_{p,x}^{(1)}[-2]e^{(1)}_{x,j}-\delta(j\leq m-n)\alpha_2e_{p,j}^{(1)}[-2]e^{(1)}_{i,q}\nonumber\\
&\quad+2\delta_{q,i}\delta(j\leq m-n)\alpha_1\alpha_2e_{p,x}^{(1)}[-3]+2\delta_{i,j}\delta(q\leq m-n)\alpha_2 e_{p,q}^{(1)}[-3].\label{1-6}
\end{align}
We rearrange the sum of \eqref{1-1}-\eqref{1-6}. 
First, let us compute
\begin{align}
&\quad\eqref{1-1}_1+\eqref{1-1}_3+\eqref{1-2}_1\nonumber\\
&=\delta(q>m-n)\sum_{x>m-n}\limits e_{p,x}^{(2)}e^{(1)}_{x,j}e_{i,q}^{(2)}+\sum_{w>m-n}\limits e_{i,q}^{(1)}e^{(1)}_{w,j}e_{p,w}^{(2)}\nonumber\\
&\quad-\delta(q\leq m-n)\sum_{x>m-n}\limits e_{p,x}^{(2)}e^{(1)}_{x,j}e_{i,q}^{(1)}\nonumber\\
&=\delta(q>m-n)\sum_{x>m-n}\limits e_{p,x}^{(2)}e^{(1)}_{x,j}e_{i,q}^{(2)}+\sum_{w>m-n}\limits e^{(1)}_{w,j}e_{p,w}^{(2)}e_{i,q}^{(1)}\nonumber\\
&\quad+\sum_{w>m-n}\limits [e_{i,q}^{(1)},e^{(1)}_{w,j}][-2]e_{p,w}^{(2)}\nonumber\\
&\quad-\delta(q\leq m-n)\sum_{x>m-n}\limits e_{p,x}^{(2)}e^{(1)}_{x,j}e_{i,q}^{(1)}\nonumber\\
&=\delta(q>m-n)\sum_{x>m-n}\limits e_{p,x}^{(2)}e^{(1)}_{x,j}e_{i,q}^{(2)}+\sum_{w>m-n}\limits e^{(1)}_{w,j}e_{p,w}^{(2)}e_{i,q}^{(1)}\nonumber\\
&\quad+\sum_{w>m-n}\limits (e^{(1)}_{i.j}\delta_{q,w}-e^{(1)}_{w,q}\delta_{i,j})[-2]e_{p,w}^{(2)}-\delta(q\leq m-n)\sum_{x>m-n}\limits e_{p,x}^{(2)}e^{(1)}_{x,j}e_{i,q}^{(1)}\nonumber\\
&=\delta(q>m-n)\sum_{x>m-n}\limits e_{p,x}^{(2)}e^{(1)}_{x,j}(e_{i,q}^{(2)}+e^{(1)}_{i,q})\nonumber\\
&\quad+\delta(q>m-n)e^{(1)}_{i.j}[-2]e_{p,q}^{(2)}-\sum_{w>m-n}\limits e^{(1)}_{w,q}[-2]\delta_{i,j}e_{p,w}^{(2)}\nonumber\\
&=\delta(q>m-n)\sum_{x>m-n}\limits (e_{p,x}^{(2)}e^{(1)}_{x,j})_{(-1)}(e_{i,q}^{(2)}+e^{(1)}_{i,q})\nonumber\\
&\quad-\delta(q>m-n)\sum_{x>m-n}\limits e_{p,x}^{(2)}[-2][e^{(1)}_{x,j},(e_{i,q}^{(2)}+e^{(1)}_{i,q})]\nonumber\\
&\quad-\delta(q>m-n)\sum_{x>m-n}\limits e_{p,x}^{(2)}[-3](e^{(1)}_{x,j},(e_{i,q}^{(2)}+e^{(1)}_{i,q}))\nonumber\\
&\quad-\delta(q>m-n)\sum_{x>m-n}\limits e^{(1)}_{x,j}[-2][e_{p,x}^{(2)},(e_{i,q}^{(2)}+e^{(1)}_{i,q})]\nonumber\\
&\quad-\delta(q>m-n)\sum_{x>m-n}\limits e^{(1)}_{x,j}[-3](e_{p,x}^{(2)},(e_{i,q}^{(2)}+e^{(1)}_{i,q}))\nonumber\\
&\quad+\delta(q>m-n)e^{(1)}_{i.j}[-2]e_{p,q}^{(2)}-\delta_{i,j}\sum_{w>m-n}\limits e^{(1)}_{w,q}[-2]e_{p,w}^{(2)}\nonumber\\
&=\delta(q>m-n)\sum_{x>m-n}\limits (e_{p,x}^{(2)}e^{(1)}_{x,j})_{(-1)}(e_{i,q}^{(2)}+e^{(1)}_{i,q})\nonumber\\
&\quad-\delta(q>m-n)\sum_{x>m-n}\limits e_{p,x}^{(2)}[-2](\delta_{i,j}e^{(1)}_{x,q}-\delta_{q,x}e^{(1)}_{i,j})\nonumber\\
&\quad-\delta(q>m-n)\sum_{x>m-n}\limits e_{p,x}^{(2)}[-3](\delta_{x,q}\delta_{i,j}\alpha_1+\delta_{x,j}\delta_{i,q})\nonumber\\
&\quad-\delta(q>m-n)\sum_{x>m-n}\limits e^{(1)}_{x,j}[-2](\delta_{x,i}e^{(2)}_{p,q}-\delta_{p,q}e^{(2)}_{i,x})\nonumber\\
&\quad-\delta(q>m-n)\sum_{x>m-n}\limits e^{(1)}_{x,j}[-3](\delta_{x,i}\delta_{p,q}\alpha_2+\delta_{p,x}\delta_{i,q})\nonumber\\
&\quad+\delta(q>m-n)e^{(1)}_{i.j}[-2]e_{p,q}^{(2)}-\delta_{i,j}\sum_{w>m-n}\limits e^{(1)}_{w,q}[-2]e_{p,w}^{(2)}\nonumber\\
&=\delta(q>m-n)\sum_{x>m-n}\limits (e_{p,x}^{(2)}e^{(1)}_{x,j})_{(-1)}(e_{i,q}^{(2)}+e^{(1)}_{i,q})\nonumber\\
&\quad-\delta_{i,j}\delta(q>m-n)\sum_{x>m-n}\limits e_{p,x}^{(2)}[-2]e^{(1)}_{x,q}+\delta(q>m-n)e_{p,q}^{(2)}[-2]e^{(1)}_{i,j}\nonumber\\
&\quad-\delta_{i,j}\delta(q>m-n)\alpha_1e_{p,q}^{(2)}[-3]-\delta_{i,q}\delta(q,j>m-n)e_{p,j}^{(2)}[-3]\nonumber\\
&\quad-\delta(q>m-n)e^{(1)}_{i,j}[-2]e^{(2)}_{p,q}+\delta(q>m-n)\sum_{x>m-n}\limits e^{(1)}_{x,j}[-2]\delta_{p,q}e^{(2)}_{i,x}\nonumber\\
&\quad-\delta_{p,q}\delta(q>m-n)\alpha_2e^{(1)}_{i,j}[-3]-\delta_{i,q}\delta(q>m-n)e^{(1)}_{p,j}[-3]\nonumber\\
&\quad+\delta(q>m-n)e^{(1)}_{i.j}[-2]e_{p,q}^{(2)}-\delta_{i,j}\sum_{w>m-n}\limits e^{(1)}_{w,q}[-2]e_{p,w}^{(2)}.\label{1-1-1}
\end{align}
Next, let us compute
\begin{align}
&\quad\eqref{1-4}_2+\eqref{1-5}_6+\eqref{1-5}_9\nonumber\\
&=-\delta(q>m-n)\sum_{x\leq m-n}\limits e_{p,x}^{(1)}e^{(1)}_{x,j}e_{i,q}^{(2)}\nonumber\\
&\quad+\delta(q\leq m-n)\sum_{x\leq m-n}\limits e_{p,x}^{(1)}e^{(1)}_{x,j}e_{i,q}^{(1)}-\sum_{x\leq m-n}\limits e_{p,x}^{(1)}e^{(1)}_{x,j}e_{i,q}^{(1)}\nonumber\\
&=-\delta(q>m-n)\sum_{x\leq m-n}\limits e_{p,x}^{(1)}e^{(1)}_{x,j}(e_{i,q}^{(2)}+e^{(1)}_{i,q})\nonumber\\
&=-\delta(q>m-n)\sum_{x\leq m-n}\limits (e_{p,x}^{(1)}e^{(1)}_{x,j})_{(-1)}(e_{i,q}^{(2)}+e^{(1)}_{i,q})\nonumber\\
&\quad+\delta(q>m-n)\sum_{x\leq m-n}\limits e_{p,x}^{(1)}[-2][e^{(1)}_{x,j},(e_{i,q}^{(2)}+e^{(1)}_{i,q})]\nonumber\\
&\quad+\delta(q>m-n)\sum_{x\leq m-n}\limits e_{p,x}^{(1)}[-3](e^{(1)}_{x,j},(e_{i,q}^{(2)}+e^{(1)}_{i,q}))\nonumber\\
&\quad+\delta(q>m-n)\sum_{x\leq m-n}\limits e^{(1)}_{x,j}[-2][e_{p,x}^{(1)},(e_{i,q}^{(2)}+e^{(1)}_{i,q}))]\nonumber\\
&\quad+\delta(q>m-n)\sum_{x\leq m-n}\limits e^{(1)}_{x,j}[-3](e_{p,x}^{(1)},(e_{i,q}^{(2)}+e^{(1)}_{i,q}))\nonumber\\
&=-\delta(q>m-n)\sum_{x\leq m-n}\limits (e_{p,x}^{(1)}e^{(1)}_{x,j})_{(-1)}(e_{i,q}^{(2)}+e^{(1)}_{i,q})\nonumber\\
&\quad+\delta(q>m-n)\sum_{x\leq m-n}\limits e_{p,x}^{(1)}[-2](\delta_{i,j}e^{(1)}_{x,q}-\delta_{x,q}e^{(1)}_{i,j})\nonumber\\
&\quad+\delta(q>m-n)\sum_{x\leq m-n}\limits e_{p,x}^{(1)}[-3](\delta_{x,q}\delta_{i,j}\alpha_1+\delta_{x,j}\delta_{i,q})\nonumber\\
&\quad+\delta(q>m-n)\sum_{x\leq m-n}\limits e^{(1)}_{x,j}[-2](\delta_{x,i}e^{(1)}_{p,q}-\delta_{p,q}e^{(1)}_{x,i})\nonumber\\
&\quad+\delta(q>m-n)\sum_{x\leq m-n}\limits e^{(1)}_{x,j}[-3](\delta_{p,q}\delta_{x,i}\alpha_1+\delta_{p,x}\delta_{i,q})\nonumber\\
&=-\delta(q>m-n)\sum_{x\leq m-n}\limits (e^{(1)}_{x,j}e_{p,x}^{(1)})_{(-1)}(e_{i,q}^{(2)}+e^{(1)}_{i,q})-\sum_{x\leq m-n}\limits [e_{p,x}^{(1)},e^{(1)}_{x,j}][-2](e_{i,q}^{(2)}+e^{(1)}_{i,q})\nonumber\\
&\quad+\delta_{i,j}\delta(q>m-n)\sum_{x\leq m-n}\limits e_{p,x}^{(1)}[-2]e^{(1)}_{x,q}-0+0+\delta_{i,q}\delta(q>m-n)\delta(j\leq m-n)e_{p,j}^{(1)}[-3]\nonumber\\
&\quad+0-\delta_{p,q}\delta(q>m-n)\sum_{x\leq m-n}\limits e^{(1)}_{x,j}[-2]e^{(1)}_{i,x}+0+0\nonumber\\
&=-\delta(q>m-n)\sum_{x\leq m-n}\limits (e^{(1)}_{x,j}e_{p,x}^{(1)})_{(-1)}(e_{i,q}^{(2)}+e^{(1)}_{i,q})\nonumber\\
&\quad-\delta(q>m-n)\sum_{x\leq m-n}\limits (e^{(1)}_{p,j}-\delta_{p,j}e^{(1)}_{x,x})[-2](e_{i,q}^{(2)}+e^{(1)}_{i,q})\nonumber\\
&\quad+\delta_{i,j}\delta(q>m-n)\sum_{x\leq m-n}\limits e_{p,x}^{(1)}[-2]e^{(1)}_{x,q}-0+0+\delta_{i,q}\delta(q>m-n)\delta(j\leq m-n)e_{p,j}^{(1)}[-3]\nonumber\\
&\quad-\delta_{p,q}\delta(q>m-n)\sum_{x\leq m-n}\limits e^{(1)}_{x,j}[-2]e^{(1)}_{i,x}\nonumber\\
&=-\delta(q>m-n)\sum_{x\leq m-n}\limits (e^{(1)}_{x,j}e_{p,x}^{(1)})_{(-1)}(e_{i,q}^{(2)}+e^{(1)}_{i,q})\nonumber\\
&\quad-\delta(q>m-n)(m-n)e^{(1)}_{p,j}[-2](e_{i,q}^{(2)}+e^{(1)}_{i,q})+\delta(q>m-n)\delta_{p,j}\sum_{x\leq m-n}\limits e^{(1)}_{x,x}[-2](e_{i,q}^{(2)}+e^{(1)}_{i,q})\nonumber\\
&\quad+\delta_{i,j}\delta(q>m-n)\sum_{x\leq m-n}\limits e_{p,x}^{(1)}[-2]e^{(1)}_{x,q}+\delta_{i,q}\delta(q>m-n)\delta(j\leq m-n)e_{p,j}^{(1)}[-3]\nonumber\\
&\quad-\delta_{p,q}\delta(q>m-n)\sum_{x\leq m-n}\limits e^{(1)}_{x,j}[-2]e^{(1)}_{i,x}.\label{1-1-2}
\end{align}
By the definition of $W^{(2)}_{i,j}$, we obtain
\begin{align}
&\quad\eqref{1-1-1}_1+\eqref{1-1-2}_1\nonumber\\
&=\delta(q>m-n)W^{(2)}_{p,j}W^{(1)}_{i,q}+\delta(q>m-n)\alpha_2e^{(1)}_{p,j}[-2]e_{i,q}^{(2)}+\delta(q>m-n)\alpha_2e^{(1)}_{p,j}[-2]e_{i,q}^{(1)}.\label{1-1-3}
\end{align}
Next, since
\begin{align*}
&\quad\eqref{1-1}_4+\eqref{1-4}_3+\eqref{1-5}_1+\eqref{1-5}_7\nonumber\\
&=-\sum_{x>m-n}\limits e_{x,q}^{(1)}e_{p,j}^{(1)}e^{(2)}_{i,x}+\delta(j\leq m-n)\sum_{w>m-n}\limits e_{p,j}^{(1)}e^{(1)}_{w,q}e_{i,w}^{(2)}\\
&\quad+\sum_{\substack{w\leq m-n}}\limits e^{(1)}_{x,q}e_{p,j}^{(1)}e_{i,x}^{(1)}-\delta(j\leq m-n)\sum_{w\leq m-n}\limits e_{p,j}^{(1)}e^{(1)}_{w,q}e_{i,w}^{(1)}\\
&=-\sum_{x>m-n}\limits e_{p,j}^{(1)}e_{x,q}^{(1)}e^{(2)}_{i,x}-\sum_{x>m-n}\limits [e_{x,q}^{(1)},e_{p,j}^{(1)}][-2]e^{(2)}_{i,x}\\
&\quad+\delta(j\leq m-n)\sum_{w>m-n}\limits e_{p,j}^{(1)}e^{(1)}_{w,q}e_{i,w}^{(2)}\\
&\quad+\sum_{\substack{w\leq m-n}}\limits e_{p,j}^{(1)}e^{(1)}_{x,q}e_{i,x}^{(1)}+\sum_{\substack{w\leq m-n}}\limits [e^{(1)}_{x,q},e_{p,j}^{(1)}][-2]e_{i,x}^{(1)}\\
&\quad-\delta(j\leq m-n)\sum_{w\leq m-n}\limits e_{p,j}^{(1)}e^{(1)}_{w,q}e_{i,w}^{(1)}\\
&=-\delta(j>m-n)e^{(1)}_{p,j}(\sum_{w>m-n}\limits e^{(1)}_{w,q}e_{i,w}^{(2)}-\sum_{w\leq m-n}\limits e^{(1)}_{w,q}e_{i,w}^{(1)})\\
&\quad+\sum_{\substack{w\leq m-n}}\limits [e^{(1)}_{x,q},e_{p,j}^{(1)}][-2]e_{i,x}^{(1)}-\sum_{x>m-n}\limits [e_{x,q}^{(1)},e_{p,j}^{(1)}][-2]e^{(2)}_{i,x}
\end{align*}
holds, we obtain
\begin{align}
&\quad\eqref{1-1}_2+\eqref{1-2}_2+\eqref{1-1}_4+\eqref{1-4}_3+\eqref{1-5}_1+\eqref{1-5}_7\nonumber\\
&=-\delta(j>m-n)\sum_{w>m-n}\limits e_{p,j}^{(2)}e^{(1)}_{w,q}e_{i,w}^{(2)}\delta(j>m-n)\sum_{\substack{w\leq m-n}}\limits e_{p,j}^{(2)}e^{(1)}_{w,q}e_{i,w}^{(1)}\nonumber\\
&\quad-\delta(j>m-n)e^{(1)}_{p,j}(\sum_{w>m-n}\limits e^{(1)}_{w,q}e_{i,w}^{(2)}-\sum_{w\leq m-n}\limits e^{(1)}_{w,q}e_{i,w}^{(1)})\nonumber\\
&\quad+\sum_{\substack{w\leq m-n}}\limits [e^{(1)}_{x,q},e_{p,j}^{(1)}][-2]e_{i,x}^{(1)}-\sum_{x>m-n}\limits [e_{x,q}^{(1)},e_{p,j}^{(1)}][-2]e^{(2)}_{i,x}\nonumber\\
&=-\delta(j>m-n)e_{p,j}^{(2)}(\sum_{w>m-n}\limits e^{(1)}_{w,q}e_{i,w}^{(2)}-\sum_{\substack{w\leq m-n}}\limits e^{(1)}_{w,q}e_{i,w}^{(1)})\nonumber\\
&\quad-\delta(j>m-n)e^{(1)}_{p,j}(\sum_{w>m-n}\limits e^{(1)}_{w,q}e_{i,w}^{(2)}-\sum_{w\leq m-n}\limits e^{(1)}_{w,q}e_{i,w}^{(1)})\nonumber\\
&\quad+\sum_{\substack{w\leq m-n}}\limits [e^{(1)}_{x,q},e_{p,j}^{(1)}][-2]e_{i,x}^{(1)}-\sum_{x>m-n}\limits [e_{x,q}^{(1)},e_{p,j}^{(1)}][-2]e^{(2)}_{i,x}\nonumber\\
&=-\delta(j>m-n)W^{(1)}_{p,j}(W^{(2)}_{i,q}+\alpha_2 e^{(1)}_{i,q}[-2])\nonumber\\
&\quad+\sum_{\substack{w\leq m-n}}\limits (\delta_{p,q}e^{(1)}_{x,j}-\delta_{x,j}e^{(1)}_{p,q})[-2]e_{i,x}^{(1)}-\sum_{x>m-n}\limits (\delta_{p,q}e^{(1)}_{x,j}-\delta_{x,j}e^{(1)}_{p,q})[-2]e^{(2)}_{i,x}\nonumber\\
&=-\delta(j>m-n)W^{(1)}_{p,j}W^{(2)}_{i,q}-\delta(j>m-n)\alpha_2e^{(1)}_{p,j} e^{(1)}_{i,q}[-2]-\delta(j>m-n)\alpha_2e^{(2)}_{p,j} e^{(1)}_{i,q}[-2]\nonumber\\
&\quad+\sum_{\substack{x\leq m-n}}\limits (\delta_{p,q}e^{(1)}_{x,j}-\delta_{x,j}e^{(1)}_{p,q})[-2]e_{i,x}^{(1)}-\sum_{x>m-n}\limits (\delta_{p,q}e^{(1)}_{x,j}-\delta_{x,j}e^{(1)}_{p,q})[-2]e^{(2)}_{i,x}\nonumber\\
&=-\delta(j>m-n)W^{(1)}_{p,j}W^{(2)}_{i,q}-\delta(j>m-n)\alpha_2e^{(1)}_{p,j} e^{(1)}_{i,q}[-2]-\delta(j>m-n)\alpha_2e^{(2)}_{p,j} e^{(1)}_{i,q}[-2]\nonumber\\
&\quad+\delta_{p,q}\sum_{\substack{x\leq m-n}}\limits e^{(1)}_{x,j}[-2]e_{i,x}^{(1)}-\delta(j\leq m-n)e^{(1)}_{p,q}[-2]e_{i,j}^{(1)}\nonumber\\
&\quad-\delta_{p,q}\sum_{x>m-n}\limits e^{(1)}_{x,j}[-2]e^{(2)}_{i,x}+\delta(j>m-n)e^{(1)}_{p,q}[-2]e^{(2)}_{i,j}.\label{1-1-4}
\end{align}

Next, let us compute the term including $e_{p,j}^{(a)}[-2]e^{(b)}_{i,q}$ or $e^{(b)}_{i,q}e_{p,j}^{(a)}[-2]$. We divide
\begin{align*}
\eqref{1-1-2}_2&=-\delta(q>m-n)(m-n)e^{(1)}_{p,j}[-2]e^{(2)}_{i,q}-\delta(q>m-n)(m-n)e^{(1)}_{p,j}[-2]e^{(1)}_{i,q}
\end{align*}
and denote it by $\eqref{1-1-2}_{2,1}$ and $\eqref{1-1-2}_{2,2}$ respectively.
Since
\begin{align*}
&\quad\eqref{1-2}_4+\eqref{1-2}_8+\eqref{1-3}_4\\
&=-\delta(j>m-n)\delta(q\leq m-n)\alpha_1e^{(2)}_{p,j}[-2]e^{(1)}_{i,q}-\delta(j>m-n)(m-n)e^{(2)}_{p,j}[-2]e^{(1)}_{i,q}\\
&\quad+\delta(j>m-n)\alpha_2e^{(2)}_{p,j}[-2]e^{(1)}_{i,q}\\
&=\delta(q,j>m-n)\alpha_1e^{(2)}_{p,j}[-2]e^{(1)}_{i,q},\\
&\quad\eqref{1-4}_4+\eqref{1-1-2}_{2,1}+\eqref{1-1-3}_2\\
&=-\delta(q>m-n)\delta(j\leq m-n)\alpha_1e^{(1)}_{p,j}[-2]e^{(2)}_{i,q}-\delta(q>m-n)(m-n)e^{(1)}_{p,j}[-2]e^{(2)}_{i,q}\\
&\quad+\delta(q>m-n)\alpha_2e^{(1)}_{p,j}[-2]e_{i,q}^{(2)}\\
&=\delta(q,j>m-n)\alpha_1e^{(1)}_{p,j}[-2]e_{i,q}^{(2)},\\
&\quad\eqref{1-5}_{10}+\eqref{1-5}_{16}+\eqref{1-6}_6+\eqref{1-1-2}_{2,2}+\eqref{1-1-3}_3\\
&=\delta(q,j\leq m-n)\alpha_1e^{(1)}_{p,j}[-2]e^{(1)}_{i,q}+\delta(j\leq m-n)(m-n)e^{(1)}_{p,j}[-2]e^{(1)}_{i,q}\\
&\quad-\delta(j\leq m-n)\alpha_2e^{(1)}_{p,j}[-2]e^{(1)}_{i,q}-\delta(q>m-n)(m-n)e^{(1)}_{p,j}[-2]e^{(1)}_{i,q}\\
&\quad+\delta(q>m-n)\alpha_2e^{(1)}_{p,j}[-2]e_{i,q}^{(1)}\\
&=\delta(q,j>m-n)\alpha_1e^{(1)}_{p,j}[-2]e_{i,q}^{(1)}
\end{align*}
hold, we obtain
\begin{align*}
&\quad\eqref{1-1}_5+\eqref{1-2}_4+\eqref{1-2}_8+\eqref{1-3}_4+\eqref{1-4}_4\\
&\qquad\qquad+\eqref{1-1-2}_{2,1}+\eqref{1-1-3}_2+\eqref{1-5}_{10}+\eqref{1-5}_{16}+\eqref{1-6}_6+\eqref{1-1-2}_{2,2}+\eqref{1-1-3}_3\\
&=\delta(q,j>m-n)\alpha_1e^{(2)}_{p,j}[-2]e^{(2)}_{i,q}+\delta(q,j>m-n)\alpha_1e^{(2)}_{p,j}[-2]e^{(1)}_{i,q}\\
&\quad+\delta(q,j>m-n)\alpha_1e^{(1)}_{p,j}[-2]e_{i,q}^{(2)}+\delta(q,j>m-n)\alpha_1e^{(1)}_{p,j}[-2]e_{i,q}^{(1)}\\
&=\delta(q,j>m-n)\alpha_1(\partial W^{(1)}_{p,j})_{(-1)}W^{(1)}_{i,q}.
\end{align*}
Next, let us compute the term including $e_{i,q}^{(a)}[-2]e_{p,j}^{(b)}$ or $e_{p,j}^{(b)}e_{i,q}^{(a)}[-2]$. Since
\begin{align*}
&\quad\eqref{1-1}_7+\eqref{1-6}_4+\eqref{1-1-4}_2\\
&=\alpha_2e^{(1)}_{i,q}[-2]e^{(1)}_{p,j}-\delta(j\leq m-n)\alpha_2e^{(1)}_{p,j} e^{(1)}_{i,q}[-2]-\delta(j>m-n)\alpha_2e^{(1)}_{p,j} e^{(1)}_{i,q}[-2]\\
&=\alpha_2[e^{(1)}_{i,q},e^{(1)}_{p,j}][-3]\\
&=\alpha_2(\delta_{p,q}e^{(1)}_{i,j}-\delta_{i,j}e^{(1)}_{p,q})[-3],\\
&\quad\eqref{1-3}_2+\eqref{1-1-4}_3\\
&=\delta(j>m-n)\alpha_2e^{(2)}_{p,j}e^{(1)}_{i,q}[-2]-\delta(j>m-n)\alpha_2e^{(2)}_{p,j} e^{(1)}_{i,q}[-2]\\
&=0
\end{align*}
hold, we have
\begin{align}
&\quad\eqref{1-1}_7+\eqref{1-3}_2+\eqref{1-6}_4+\eqref{1-1-4}_2+\eqref{1-1-4}_3\nonumber\\
&=\delta_{p,q}\alpha_2e^{(1)}_{i,j}[-3]-\delta_{i,j}\alpha_2e^{(1)}_{p,q}[-3].\label{1-1-5}
\end{align}
Next, we compute the sum of the terms including $e^{(a)}_{p,q}[-2]e^{(b)}_{i,j}$ or $e^{(b)}_{i,j}[-2]e^{(a)}_{p,q}$.
Since
\begin{align*}
&\quad\eqref{1-2}_5+\eqref{1-1-1}_3\\
&=-\delta(q>m-n)\delta(j\leq m-n)e_{p,q}^{(2)}[-2]e_{i,j}^{(1)}+\delta(q>m-n)e_{p,q}^{(2)}[-2]e^{(1)}_{i,j}\\
&=\delta(q,j>m-n)e_{p,q}^{(2)}[-2]e^{(1)}_{i,j},\\
&\quad\eqref{1-4}_5+\eqref{1-1-1}_6+\eqref{1-1-1}_{10}+\eqref{1-1-4}_7\\
&=-\delta(j> m-n)\delta(q\leq m-n)e^{(1)}_{p,q}[-2]e^{(2)}_{i,j}-\delta(q>m-n)e^{(1)}_{i,j}[-2]e^{(2)}_{p,q}\\
&\quad+\delta(q>m-n)e^{(1)}_{i.j}[-2]e_{p,q}^{(2)}+\delta(j>m-n)e^{(1)}_{p,q}[-2]e^{(2)}_{i,j}\\
&=\delta(q,j>m-n)e_{p,q}^{(1)}[-2]e^{(2)}_{i,j},\\
&\quad\eqref{1-1}_8+\eqref{1-5}_{11}+\eqref{1-5}_{14}+\eqref{1-1-4}_5\\
&=e_{p,q}^{(1)}[-2]e_{i,j}^{(1)}+\delta(q,j\leq m-n)e^{(1)}_{p,q}[-2]e^{(1)}_{i,j}\\
&\quad-\delta(q\leq m-n)e^{(1)}_{p,q}[-2]e^{(1)}_{i,j}-\delta(j\leq m-n)e^{(1)}_{p,q}[-2]e_{i,j}^{(1)}\\
&=\delta(q,j>m-n)e_{p,q}^{(1)}[-2]e^{(2)}_{i,j}
\end{align*}
hold, we have
\begin{align*}
&\quad\eqref{1-1}_6+\eqref{1-2}_5+\eqref{1-1-1}_3\\
&\qquad\qquad+\eqref{1-4}_5+\eqref{1-1-1}_6+\eqref{1-1-1}_{10}+\eqref{1-1-4}_7+\eqref{1-1}_8+\eqref{1-5}_{11}+\eqref{1-5}_{14}+\eqref{1-1-4}_5\\
&=\delta(q,j>m-n)e_{p,q}^{(2)}[-2]e_{i,j}^{(2)}+\delta(q,j>m-n)e_{p,q}^{(2)}[-2]e^{(1)}_{i,j}\\
&\quad+\delta(q,j>m-n)e_{p,q}^{(1)}[-2]e^{(2)}_{i,j}+\delta(q,j>m-n)e_{p,q}^{(1)}[-2]e^{(2)}_{i,j}\\
&=\delta(q,j>m-n)(\partial W^{(1)}_{p,q})_{(-1)}W^{(1)}_{i,j}.
\end{align*}
Next, let us compute the sum of terms including $\delta_{i,q}$.
By a direct computation, we obtain
\begin{align}
&\quad\eqref{1-2}_3+\eqref{1-5}_8\nonumber\\
&=-\delta_{q,i}\sum_{\substack{x>m-n\\w\leq m-n}}\limits e_{p,x}^{(2)}e_{w,j}^{(1)}e_{w,x}^{(1)}+\delta_{q,i}\sum_{x,w\leq m-n}\limits e^{(1)}_{p,x}e^{(1)}_{w,j}e^{(1)}_{x,w}\nonumber\\
&=-\delta_{q,i}\sum_{\substack{x>m-n\\w\leq m-n}}\limits e_{w,j}^{(1)}e_{p,x}^{(2)}e_{w,x}^{(1)}\nonumber\\
&\quad+\delta_{q,i}\sum_{x,w\leq m-n}\limits e^{(1)}_{w,j}e^{(1)}_{p,x}e^{(1)}_{x,w}+\delta_{q,i}\sum_{x,w\leq m-n}\limits [e^{(1)}_{p,x},e^{(1)}_{w,j}][-2]e^{(1)}_{x,w}\nonumber\\
&=-\delta_{q,i}\sum_{w\leq m-n}\limits e_{w,j}^{(1)}(\sum_{\substack{x>m-n}}\limits e_{p,x}^{(2)}e_{x,w}^{(1)}-\sum_{x\leq m-n}\limits e^{(1)}_{p,x}e^{(1)}_{x,w})\nonumber\\
&\quad+\delta_{q,i}\sum_{w\leq m-n}\limits e^{(1)}_{p,j}[-2]e^{(1)}_{w,w}-\delta_{q,i}\delta_{p,j}\sum_{x,w\leq m-n}\limits e^{(1)}_{w,x}[-2]e^{(1)}_{x,w}\nonumber\\
&=-\delta_{q,i}\sum_{w\leq m-n}\limits W^{(1)}_{w,j}(W^{(2)}_{p,w}+\alpha_2e^{(1)}_{p,w}[-2])+\delta_{q,i}\sum_{w\leq m-n}\limits\sum_{x\leq m-n}\limits e_{w,j}^{(1)}[e^{(1)}_{p,x},e^{(1)}_{x,w}][-2]\nonumber\\
&\quad-\delta_{q,i}\delta_{p,j}\sum_{x,w\leq m-n}\limits \partial(W^{(1)}_{w,x})_{(-1)}W^{(1)}_{x,w}\nonumber\\
&\quad+\delta_{q,i}(m-n)\sum_{w\leq m-n}\limits e^{(1)}_{w,j}e^{(1)}_{p,w}[-2]+\delta_{q,i}\sum_{w\leq m-n}\limits e^{(1)}_{p,j}[-2]e^{(1)}_{w,w}\nonumber\\
&=-\delta_{q,i}\sum_{w\leq m-n}\limits W^{(1)}_{w,j}W^{(2)}_{p,w}-\delta_{q,i}\alpha_2\sum_{w\leq m-n}\limits e^{(1)}_{w,j}e^{(1)}_{p,w}[-2]\nonumber\\
&\quad-\delta_{q,i}\delta_{p,j}\sum_{x,w\leq m-n}\limits\partial(W^{(1)}_{w,x})_{(-1)}W^{(1)}_{x,w}\nonumber\\
&\quad+\delta_{q,i}(m-n)\sum_{w\leq m-n}\limits e^{(1)}_{w,j}e^{(1)}_{p,w}[-2]+\delta_{q,i}\sum_{w\leq m-n}\limits e^{(1)}_{p,j}[-2]e^{(1)}_{w,w}.\label{1-3-1}
\end{align}
Since
\begin{align*}
&\quad\eqref{1-2}_7+\eqref{1-5}_{15}+\eqref{1-3-1}_5\\
&=\delta_{q,i}\delta(j>m-n)\sum_{\substack{w\leq m-n}}\limits e_{p,j}^{(2)}[-2]e^{(1)}_{w,w}\\
&\quad-\delta_{i,q}\delta(j\leq m-n)\sum_{w\leq m-n}\limits e_{p,j}^{(1)}[-2]e^{(1)}_{w,w}+\delta_{q,i}\sum_{w\leq m-n}\limits e^{(1)}_{p,j}[-2]e^{(1)}_{w,w}\\
&=\delta_{q,i}\delta(j>m-n)\sum_{\substack{w\leq m-n}}\limits (\partial W^{(1)}_{p,j})W^{(1)}_{w,w},\\
&\quad\eqref{1-3}_1+\eqref{1-3}_3\\
&=-\delta_{q,i}\alpha_2\sum_{x>m-n}\limits e_{p,x}^{(2)}e^{(1)}_{x,j}[-2]-\delta_{q,i}\alpha_2\sum_{x>m-n}\limits e_{p,x}^{(2)}[-2]e^{(1)}_{x,j}[-1]\\
&=-\delta_{q,i}\alpha_2\partial (e_{p,x}^{(2)}e^{(1)}_{x,j}),\\
&\quad\eqref{1-5}_{12}+\eqref{1-6}_3+\eqref{1-6}_5+\eqref{1-3-1}_2+\eqref{1-3-1}_4\\
&=\delta_{q,i}\alpha_1\sum_{x\leq m-n}\limits e_{p,x}^{(1)}[-2]e_{x,j}^{(1)}\\
&\quad+\delta_{q,i}\alpha_2\sum_{x\leq m-n}\limits e_{p,x}^{(1)}e^{(1)}_{x,j}[-2]+\delta_{q,i}\alpha_2\sum_{x\leq m-n}\limits e_{p,x}^{(1)}[-2]e^{(1)}_{x,j}\\
&\quad-\delta_{q,i}\alpha_2\sum_{w\leq m-n}\limits e^{(1)}_{w,j}e^{(1)}_{p,w}[-2]+\delta_{q,i}(m-n)\sum_{w\leq m-n}\limits e_{w,j}^{(1)}e^{(1)}_{p,w}[-2]\\
&=\delta_{q,i}\alpha_2\partial(\sum_{w\leq m-n}\limits e_{w,j}^{(1)}e^{(1)}_{p,w})+\delta_{q,i}(\alpha_1+2\alpha_2)\sum_{x\leq m-n}\limits [e_{p,x}^{(1)},e_{x,j}^{(1)}][-3]
\end{align*}
hold, we have
\begin{align}
&\quad\eqref{1-2}_7+\eqref{1-3}_1+\eqref{1-3}_3+\eqref{1-5}_9+\eqref{1-6}_3+\eqref{1-6}_5+\eqref{1-3-1}_2+\eqref{1-3-1}_4+\eqref{1-3-1}_5\nonumber\\
&=\delta_{q,i}\delta(j>m-n)\sum_{w\leq m-n}\limits(\partial W^{(1)}_{p,j})W^{(1)}_{w,w}-\delta_{q,i}\alpha_2(\partial W^{(2)}_{p,j}+2\alpha_2e^{(1)}_{p,j}[-3])\nonumber\\
&\quad+\delta_{q,i}(\alpha_1+2\alpha_2)(m-n)e^{(1)}_{p,j}[-3]-\dfrac{1}{2}\delta_{q,i}\delta_{p,j}(\alpha_1+2\alpha_2)\sum_{x\leq m-n}\limits\partial^2 W^{(1)}_{x,x}\nonumber\\
&=\delta_{q,i}\delta(j>m-n)\sum_{w\leq m-n}\limits(\partial W^{(1)}_{p,j})W^{(1)}_{w,w}\nonumber\\
&\quad-\delta_{q,i}\alpha_2\partial W^{(2)}_{p,j}-2\delta_{q,i}\alpha_2^2e^{(1)}_{p,j}[-3]\nonumber\\
&\quad+(\alpha_1+2\alpha_2)\delta_{q,i}(m-n)e^{(1)}_{p,j}[-3]-\dfrac{1}{2}\delta_{q,i}\delta_{p,j}(\alpha_1+2\alpha_2)\sum_{x\leq m-n}\limits\partial^2 W^{(1)}_{x,x}.\label{1-1-9}
\end{align}
Since
\begin{align*}
&\quad\eqref{1-2}_9+\eqref{1-3}_5\\
&=-\delta_{q,i}\delta(j>m-n)(\alpha_1^2+\alpha_1\alpha_2)e^{(2)}_{p,j}[-3],\\
&\quad\eqref{1-1-9}_3+\eqref{1-1-9}_4+\eqref{1-5}_{19}+\eqref{1-6}_7\\
&=-\delta_{q,i}\delta(j>m-n)(\alpha_1^2+\alpha_1\alpha_2)e^{(1)}_{p,j}[-3],\\
&\quad\eqref{1-1-1}_5+\eqref{1-1-1}_9+\eqref{1-1-2}_5\\
&=-\dfrac{1}{2}\delta_{i,q}\delta(j>m-n)\partial^2W^{(1)}_{p,j}
\end{align*}
hold, we have
\begin{align*}
&\quad\eqref{1-2}_9+\eqref{1-3}_5+\eqref{1-5}_{19}+\eqref{1-6}_7+\eqref{1-1-1}_5+\eqref{1-1-1}_9+\eqref{1-1-2}_5+\eqref{1-1-9}_3+\eqref{1-1-9}_4\\
&=-\dfrac{1}{2}\delta_{q,i}\delta(j>m-n)(\alpha_1^2+\alpha_1\alpha_2)\partial^2W^{(1)}_{p,j}-\dfrac{1}{2}\delta_{i,q}\delta(j>m-n)\partial^2W^{(1)}_{p,j}.
\end{align*}
Next, we compute the sum of terms containing $\delta_{p,j}$
By a direct computation, we obtain
\begin{align}
&\quad\eqref{1-4}_1+\eqref{1-5}_2\nonumber\\
&=\delta_{p,j}\sum_{x\leq m-n}\limits e^{(1)}_{x,q}(\sum_{w>m-n}\limits e^{(1)}_{w,x}e^{(2)}_{i,w}-\sum_{w\leq m-n}\limits e^{(1)}_{w,x}e^{(1)}_{i,w})\nonumber\\
&=\delta_{p,j}\sum_{x\leq m-n}\limits (W^{(1)}_{x,q})_{(-1)}W^{(2)}_{i,x}+\delta_{p,j}\alpha_2\sum_{x\leq m-n}\limits e^{(1)}_{x,q}e^{(1)}_{i,x}.\label{1-4-1}
\end{align}
Since
\begin{align*}
\eqref{1-1-2}_3&=\delta_{p,j}\delta(q>m-n)\sum_{x\leq m-n}\limits (\partial W^{(1)}_{x,x})_{(-1)}W^{(1)}_{i,q},\\
\eqref{1-5}_3+\eqref{1-5}_5+\eqref{1-6}_2&=0,\\
\eqref{1-6}_1+\eqref{1-4-1}_2&=0
\end{align*}
hold, we have
\begin{align*}
&\quad\eqref{1-5}_3+\eqref{1-5}_5+\eqref{1-6}_1+\eqref{1-6}_2+\eqref{1-1-2}_3+\eqref{1-4-1}_2\\
&=\delta_{p,j}\delta(q>m-n)\sum_{x\leq m-n}\limits (\partial W^{(1)}_{x,x})_{(-1)}W^{(1)}_{i,q}.
\end{align*}
Next, let us compute the sum of terms containing $\delta_{i,j}$.
Since
\begin{align*}
&\quad\eqref{1-2}_6+\eqref{1-1-1}_2+\eqref{1-1-1}_{11}\\
&=-\delta_{j,i}\delta(q\leq m-n)\sum_{\substack{x>m-n}}\limits e_{p,x}^{(2)}[-2]e^{(1)}_{x,q}\\
&\quad-\delta_{i,j}\delta(q>m-n)\sum_{x>m-n}\limits e_{p,x}^{(2)}[-2]e^{(1)}_{x,q}-\delta_{i,j}\sum_{w>m-n}\limits e^{(1)}_{w,q}[-2]e_{p,w}^{(2)}\\
&=-\delta_{i,j}\partial(\sum_{w>m-n}\limits e^{(1)}_{w,q}e_{p,w}^{(2)}),\\
&\quad\eqref{1-5}_4+\eqref{1-5}_{13}+\eqref{1-1-2}_4\\
&=\delta_{i,j}\sum_{x\leq m-n}\limits e_{x,q}^{(1)}[-2]e^{(1)}_{p,x}+\delta_{i,j}\delta(q\leq m-n)\sum_{x\leq m-n}\limits e_{p,x}^{(1)}[-2]e^{(1)}_{x,q}\\
&\quad+\delta_{i,j}\delta(q>m-n)\sum_{x\leq m-n}\limits e_{p,x}^{(1)}[-2]e^{(1)}_{x,q}\\
&=\delta_{i,j}\partial(\sum_{x\leq m-n}\limits e_{x,q}^{(1)}e^{(1)}_{p,x})+\delta_{i,j}\sum_{x\leq m-n}\limits [e_{p,x}^{(1)},e^{(1)}_{x,q}][-3]
\end{align*}
hold, we have
\begin{align}
&\quad\eqref{1-2}_6+\eqref{1-5}_4+\eqref{1-5}_{13}+\eqref{1-1-1}_2+\eqref{1-1-1}_{11}+\eqref{1-1-2}_4\nonumber\\
&=-\delta_{i,j}(\partial W^{(2)}_{p,q}+2\alpha_2e^{(1)}_{p,q}[-3])+\delta_{i,j}(m-n)e^{(1)}_{p,q}[-3]-\delta_{i,j}\delta_{p,q}\sum_{x\leq m-n}\limits \partial^2W^{(1)}_{x,x}\nonumber\\
&=-\delta_{i,j}\partial W^{(2)}_{p,q}-2\delta_{i,j}\alpha_2e^{(1)}_{p,q}[-3]+\delta_{i,j}(m-n)e^{(1)}_{p,q}[-3]-\delta_{i,j}\delta_{p,q}\sum_{x\leq m-n}\limits \partial^2W^{(1)}_{x,x}.\label{1-1-6}
\end{align}
Since
\begin{align*}
&\quad\eqref{1-5}_{17}+\eqref{1-6}_8\\
&=\delta_{i,j}\alpha_1\delta(q\leq m-n) e_{p,q}^{(1)}[-3]+2\delta_{i,j}\alpha_2\delta(q\leq m-n)e^{(1)}_{p,q}[-3],\\
&=\dfrac{1}{2}\delta_{i,j}(\alpha_1+2\alpha_2)\delta(q\leq m-n)\partial^2e^{(1)}_{p,q},\\
&\quad\eqref{1-3}_6+\eqref{1-1-1}_4+\eqref{1-1-5}_2+\eqref{1-1-6}_2+\eqref{1-1-6}_3\\
&=-2\delta_{i,j}\alpha_2\delta(q> m-n)e_{p,q}^{(2)}[-3]-\delta_{i,j}\delta(q>m-n)\alpha_1e^{(2)}_{p,q}[-3]\\
&\quad-\delta_{i,j}\alpha_2e^{(1)}_{p,q}[-3]-2\delta_{i,j}\alpha_2e^{(1)}_{p,q}[-3]+\delta_{i,j}(m-n)e^{(1)}_{p,q}[-3]\\
&=-\dfrac{1}{2}\delta_{i,j}(\alpha_1+2\alpha_2)\partial^2e^{(1)}_{p,q}
\end{align*}
hold, we have
\begin{align*}
&\quad\eqref{1-3}_6+\eqref{1-5}_{17}+\eqref{1-6}_8+\eqref{1-1-1}_4+\eqref{1-1-5}_2+\eqref{1-1-6}_2+\eqref{1-1-6}_3\\
&=-\dfrac{1}{2}\delta_{i,j}\delta(q>m-n)\alpha_2\partial^2W^{(1)}_{p,q}.
\end{align*}
Finally, let us compute the sum of terms containing $\delta_{p,q}$. By a direct computation, we obtain
\begin{align*}
&\quad\eqref{1-1-1}_7+\eqref{1-1-2}_6+\eqref{1-1-4}_4+\eqref{1-1-4}_6\\
&=\delta_{p,q}\delta(q>m-n)\sum_{x>m-n}\limits e^{(1)}_{x,j}[-2]e^{(2)}_{i,x}-\delta_{p,q}\delta(q>m-n)\sum_{\substack{x\leq m-n}}\limits e^{(1)}_{x,j}[-2]e_{i,x}^{(1)}\\
&\quad+\delta_{p,q}\sum_{\substack{x\leq m-n}}\limits e^{(1)}_{x,j}[-2]e_{i,x}^{(1)}-\delta_{p,q}\sum_{x>m-n}\limits e^{(1)}_{x,j}[-2]e^{(2)}_{i,x}\\
&=0.
\end{align*}
By a direct computation, we obtain
\begin{align*}
&\quad\eqref{1-1-1}_8+\eqref{1-1-5}_1\\
&=-\delta_{p,q}\alpha_2e^{(1)}_{i,j}[-3]+\delta_{p,q}\alpha_2e^{(1)}_{i,j}[-3]=0
\end{align*}
We complete the proof of \eqref{OPE3-1}.
\end{proof}
\section{the proof of \eqref{OPE3-2}}
This section is devoted to the proof of \eqref{OPE3-2}. We use the same notation as appendix A.
By the definition of $W^{(2)}_{i,j}$, we have
\begin{align*}
&\quad(W^{(2)}_{p,q})_{(1)}W^{(2)}_{i,j}\\
&=-(\alpha_2e^{(1)}_{p,q}[-2])_{(1)}(\sum_{w>m-n}\limits e_{w,j}^{(1)}e_{i,w}^{(2)}-\sum_{w\leq m-n}\limits e_{w,j}^{(1)}e_{i,w}^{(1)}-\alpha_2e^{(1)}_{i,j}[-2])\\
&\quad+(\sum_{x>m-n}\limits e_{x,q}^{(1)}e_{p,x}^{(2)})_{(1)}(\sum_{w>m-n}\limits e_{w,j}^{(1)}e_{i,w}^{(2)}-\sum_{w\leq m-n}\limits e_{w,j}^{(1)}e_{i,w}^{(1)}-\alpha_2e^{(1)}_{i,j}[-2])\\
&\quad-(\sum_{x\leq m-n}\limits e_{x,q}^{(1)}e_{p,x}^{(1)})_{(1)}(\sum_{w>m-n}\limits e_{w,j}^{(1)}e_{i,w}^{(2)}-\sum_{w\leq m-n}\limits e_{w,j}^{(1)}e_{i,w}^{(1)}-\alpha_2e^{(1)}_{i,j}[-2]).
\end{align*}
By a direct computation, we obtain
\begin{align}
&\quad-(\alpha_2e^{(1)}_{p,q}[-2])_{(1)}(\sum_{w>m-n}\limits e_{w,j}^{(1)}e_{i,w}^{(2)}-\sum_{w\leq m-n}\limits e_{w,j}^{(1)}e_{i,w}^{(1)}-\alpha_2e^{(1)}_{i,j}[-2])\nonumber\\
&=\alpha_2\sum_{w>m-n}\limits [e^{(1)}_{p,q},e_{w,j}^{(1)}]e_{i,w}^{(2)}\nonumber\\
&\quad-\alpha_2\sum_{w\leq m-n}\limits [e^{(1)}_{p,q},e_{w,j}^{(1)}]e_{i,w}^{(1)}-\alpha_2\sum_{w\leq m-n}\limits e_{w,j}^{(1)}[e^{(1)}_{p,q},e_{i,w}^{(1)}]\nonumber\\
&\quad-\alpha_2^2[e^{(1)}_{p,q},e^{(1)}_{i,j}][-2]\nonumber\\
&=\alpha_2\sum_{w\leq m-n}\limits (\delta_{q,w}e^{(1)}_{p,j}-\delta_{p,j}e^{(1)}_{w,q})e_{i,w}^{(2)}\nonumber\\
&\quad-\alpha_2\sum_{w\leq m-n}\limits (\delta_{q,w}e^{(1)}_{p,j}-\delta_{p,j}e^{(1)}_{w,q})e_{i,w}^{(1)}-\alpha_2\sum_{w\leq m-n}\limits e_{w,j}^{(1)}(\delta_{q,i}e^{(1)}_{p,w}-\delta_{p,w}e^{(1)}_{i,q})\nonumber\\
&\quad-\alpha_2^2(\delta_{q,i}e^{(1)}_{p,j}-\delta_{p,j}e^{(1)}_{i,q})[-2]\nonumber\\
&=\delta(q>m-n)\alpha_2e^{(1)}_{p,j}e_{i,q}^{(2)}-\delta_{p,j}\alpha_2\sum_{w\leq m-n}\limits e^{(1)}_{w,q}e_{i,w}^{(2)}\nonumber\\
&\quad-\delta(q\leq m-n)\alpha_2e^{(1)}_{p,j}e_{i,q}^{(1)}+\delta_{p,j}\alpha_2\sum_{w\leq m-n}\limits e^{(1)}_{w,q}e_{i,w}^{(1)}\nonumber\\
&\quad-\delta_{q,i}\alpha_2\sum_{w\leq m-n}\limits e_{w,j}^{(1)}e^{(1)}_{p,w}+0\nonumber\\
&\quad-\delta_{q,i}\alpha_2^2e^{(1)}_{p,j}[-2]+\delta_{p,j}\alpha_2^2e^{(1)}_{i,q}[-2].\label{2-1}
\end{align}
By a direct computation, we obtain
\begin{align}
&\quad(\sum_{x>m-n}\limits e_{x,q}^{(1)}e_{p,x}^{(2)})_{(1)}\sum_{w>m-n}\limits e_{w,j}^{(1)}e_{i,w}^{(2)}\nonumber\\
&=\sum_{w,x>m-n}\limits [e_{x,q}^{(1)},e_{w,j}^{(1)}][e_{p,x}^{(2)},e_{i,w}^{(2)}]\nonumber\\
&\quad+\sum_{x,w>m-n}\limits e_{x,q}^{(1)}e_{w,j}^{(1)}(e_{p,x}^{(2)},e_{i,w}^{(2)})+\sum_{x,w>m-n}\limits e_{p,x}^{(2)}(e_{x,q}^{(1)},e_{w,j}^{(1)})e_{i,w}^{(2)}\nonumber\\
&=\sum_{w,x>m-n}\limits (\delta_{q,w}e^{(1)}_{x,j}-\delta_{x,j}e^{(1)}_{w,q})(\delta_{x,i}e^{(2)}_{p,w}-\delta_{p,w}e^{(2)}_{i,x})\nonumber\\
&\quad+\sum_{x,w>m-n}\limits e_{x,q}^{(1)}e_{w,j}^{(1)}(\alpha_2\delta_{x,i}\delta_{w,p}+\delta_{p,x}\delta_{i,w})+\sum_{x,w>m-n}\limits e_{p,x}^{(2)}(\delta_{q,w}\delta_{x,j}\alpha_1+\delta_{x,q}\delta_{w,j})e_{i,w}^{(2)}\nonumber\\
&=\delta(q>m-n)e^{(1)}_{i,j}e^{(2)}_{p,q}-\delta_{q,p}\sum_{x>m-n}\limits e^{(1)}_{x,j}e^{(2)}_{i,x}\nonumber\\
&\quad-\delta_{i,j}\sum_{w>m-n}\limits e^{(1)}_{w,q}e^{(2)}_{p,w}+\delta(j>m-n)e^{(1)}_{p,q}e^{(2)}_{i,j}\nonumber\\
&\quad+\alpha_2e_{i,q}^{(1)}e_{p,j}^{(1)}+ e_{p,q}^{(1)}e_{i,j}^{(1)}\nonumber\\
&\quad+\delta(q,j>m-n)\alpha_1 e_{p,j}^{(2)}e_{i,q}^{(2)}+\delta(q,j>m-n) e_{p,q}^{(2)}e_{i,j}^{(2)}.\label{2-2}
\end{align}
By a direct computation, we obtain
\begin{align}
&\quad-(\sum_{x>m-n}\limits e_{x,q}^{(1)}e_{p,x}^{(2)})_{(1)}\sum_{w\leq m-n}\limits e_{w,j}^{(1)}e_{i,w}^{(1)}\nonumber\\
&=-\sum_{\substack{x>m-n\\w\leq m-n}}\limits (e_{x,q}^{(1)},e_{w,j}^{(1)})e_{i,w}^{(1)}e_{p,x}^{(2)}-\sum_{\substack{x>m-n\\w\leq m-n}}\limits e_{w,j}^{(1)}(e_{x,q}^{(1)},e_{i,w}^{(1)})e_{p,x}^{(2)}\nonumber\\
&\quad-\sum_{\substack{x>m-n\\w\leq m-n}}\limits [[e_{x,q}^{(1)},e_{w,j}^{(1)}],e_{i,w}^{(1)}]e_{p,x}^{(2)}\nonumber\\
&\quad-\sum_{\substack{x>m-n\\w\leq m-n}}\limits e_{p,x}^{(2)}[-2]([e_{x,q}^{(1)},e_{w,j}^{(1)}],e_{i,w}^{(1)})\nonumber\\
&=-\sum_{\substack{x>m-n\\w\leq m-n}}\limits(\alpha_1\delta_{x,j}\delta_{w,q}+\delta_{x,q}\delta_{w,j})e_{i,w}^{(1)}e_{p,x}^{(2)}-\sum_{\substack{x>m-n\\w\leq m-n}}\limits e_{w,j}^{(1)}(\delta_{x,w}\delta_{q,i}\alpha_1+\delta_{x,q}\delta_{i,w})e_{p,x}^{(2)}\nonumber\\
&\quad-\sum_{\substack{x>m-n\\w\leq m-n}}\limits [\delta_{q,w}e^{(1)}_{x,j}-\delta_{j,x}e^{(1)}_{w,q},e^{(1)}_{i,w}]e_{p,x}^{(2)}\nonumber\\
&\quad-\sum_{\substack{x>m-n\\w\leq m-n}}\limits e_{p,x}^{(2)}[-2](\delta_{q,w}e^{(1)}_{x,j}-\delta_{x,j}e^{(1)}_{w,q},e_{i,w}^{(1)})\nonumber\\
&=-\sum_{\substack{x>m-n\\w\leq m-n}}\limits(\alpha_1\delta_{x,j}\delta_{w,q}+\delta_{x,q}\delta_{w,j})e_{i,w}^{(1)}e_{p,x}^{(2)}-\sum_{\substack{x>m-n\\w\leq m-n}}\limits e_{w,j}^{(1)}(\delta_{x,w}\delta_{q,i}\alpha_1+\delta_{x,q}\delta_{i,w})e_{p,x}^{(2)}\nonumber\\
&\quad-\sum_{\substack{x>m-n\\w\leq m-n}}\limits\delta_{q,w}(\delta_{i,j}e^{(1)}_{x,w}-\delta_{x,w}e^{(1)}_{i,j})e_{p,x}^{(2)}+\sum_{\substack{x>m-n\\w\leq m-n}}\limits\delta_{j,x}(\delta_{i,q}e^{(1)}_{w,w}-e^{(1)}_{i,q})e_{p,x}^{(2)}\nonumber\\
&\quad-\sum_{\substack{x>m-n\\w\leq m-n}}\limits e_{p,x}^{(2)}[-2]\delta_{q,w}(\alpha_1\delta_{x,w}\delta_{i,j}+\delta_{x,j}\delta_{i,w})+\sum_{\substack{x>m-n\\w\leq m-n}}\limits \delta_{x,j}e_{p,x}^{(2)}[-2](\alpha_1\delta_{i,q}+\delta_{i,w}\delta_{w,q})\nonumber\\
&=-\delta(q\leq m-n)\delta(j>m-n)\alpha_1e_{i,q}^{(1)}e_{p,j}^{(2)}-\delta(j\leq m-n)\delta(q>m-n)e_{i,j}^{(1)}e_{p,q}^{(2)}\nonumber\\
&\quad-0-0\nonumber\\
&\quad-\delta(q\leq m-n)\sum_{\substack{x>m-n}}\limits \delta_{j,i}e^{(1)}_{x,q}e_{p,x}^{(2)}-0\nonumber\\
&\quad+\delta(j>m-n)\sum_{\substack{w\leq m-n}}\limits\delta_{q,i}e^{(1)}_{w,w}e_{p,j}^{(2)}-\delta(j>m-n)(m-n)e^{(1)}_{i,q}e_{p,j}^{(2)}\nonumber\\
&\quad+0+0+\delta_{i,q}\delta(j>m-n)\alpha_1(m-n)e_{p,j}^{(2)}[-2]+0.\label{2-3}
\end{align}
By a direct computation, we obtain
\begin{align}
&\quad-(\sum_{x>m-n}\limits e_{x,q}^{(1)}e_{p,x}^{(2)})_{(1)}\alpha_2e^{(1)}_{i,j}[-2])\nonumber\\
&=-\alpha_2\sum_{x>m-n}\limits e_{p,x}^{(2)}[-1][e_{x,q}^{(1)},e^{(1)}_{i,j}]-2\alpha_2\sum_{x>m-n}\limits e_{p,x}^{(2)}[-2](e_{x,q}^{(1)},e^{(1)}_{i,j})\nonumber\\
&=-\alpha_2\sum_{x>m-n}\limits e_{p,x}^{(2)}[-1](\delta_{q,i}e^{(1)}_{x,j}-\delta_{x,j}e^{(1)}_{i,q})-2\alpha_2\sum_{x>m-n}\limits e_{p,x}^{(2)}[-2](\delta_{x,j}\delta_{q,i}\alpha_1+\delta_{x,q}\delta_{i,j})\nonumber\\
&=-\delta_{q,i}\alpha_2\sum_{x>m-n}\limits e_{p,x}^{(2)}e^{(1)}_{x,j}+\delta(j>m-n)\alpha_2e_{p,j}^{(2)}[-1]e^{(1)}_{i,q}\nonumber\\
&\quad-2\delta_{q,i}\delta(j>m-n)\alpha_1\alpha_2 e_{p,j}^{(2)}[-2]-2\delta_{i,j}\delta(q>m-n)\alpha_2e_{p,q}^{(2)}[-2].\label{2-4}
\end{align}
By a direct computation, we obtain
\begin{align}
&\quad-(\sum_{x\leq m-n}\limits e_{x,q}^{(1)}e_{p,x}^{(1)})_{(1)}\sum_{w>m-n}\limits e_{w,j}^{(1)}e_{i,w}^{(2)}\nonumber\\
&=-\sum_{\substack{x\leq m-n\\w>m-n}}\limits [e_{p,x}^{(1)},[e_{x,q}^{(1)},e_{w,j}^{(1)}]]e_{i,w}^{(2)}\nonumber\\
&\quad-\sum_{\substack{x\leq m-n\\w>m-n}}\limits e_{x,q}^{(1)}(e_{p,x}^{(1)},e_{w,j}^{(1)})e_{i,w}^{(2)}-\sum_{\substack{x\leq m-n\\w>m-n}}\limits e_{p,x}^{(1)}(e_{x,q}^{(1)},e_{w,j}^{(1)})e_{i,w}^{(2)}\nonumber\\
&=-\sum_{\substack{x\leq m-n\\w>m-n}}\limits [e_{p,x}^{(1)},\delta_{q,w}e^{(1)}_{x,j}-\delta_{x,j}e^{(1)}_{w,q}]e_{i,w}^{(2)}\nonumber\\
&\quad-\sum_{\substack{x\leq m-n\\w>m-n}}\limits e_{x,q}^{(1)}(e_{p,x}^{(1)},e_{w,j}^{(1)})e_{i,w}^{(2)}-\sum_{\substack{x\leq m-n\\w>m-n}}\limits e_{p,x}^{(1)}(e_{x,q}^{(1)},e_{w,j}^{(1)})e_{i,w}^{(2)}\nonumber\\
&=-\sum_{\substack{x\leq m-n\\w>m-n}}\limits\delta_{q,w}(e^{(1)}_{p,j}-\delta_{p,j}e^{(1)}_{x,x})e_{i,w}^{(2)}+\sum_{\substack{x\leq m-n\\w>m-n}}\limits\delta_{x,j}(\delta_{x,w}e^{(1)}_{p,q}-\delta_{p,q}e^{(1)}_{w,x})e_{i,w}^{(2)}\nonumber\\
&\quad-\sum_{\substack{x\leq m-n\\w>m-n}}\limits e_{x,q}^{(1)}(\delta_{p,j}\delta_{x,w}\alpha_1+\delta_{p,x}\delta_{w,j})e_{i,w}^{(2)}-\sum_{\substack{x\leq m-n\\w>m-n}}\limits e_{p,x}^{(1)}(\delta_{q,w}\delta_{x,j}\alpha_1+\delta_{q,x}\delta_{w,j})e_{i,w}^{(2)}\nonumber\\
&=-\delta(q>m-n)(m-n)e^{(1)}_{p,j}e_{i,q}^{(2)}+\delta_{p,j}\delta(q>m-n)\sum_{\substack{x\leq m-n}}\limits e^{(1)}_{x,x}e_{i,q}^{(2)}\nonumber\\
&\quad+0-\delta_{p,q}\delta(j\leq m-n)\sum_{\substack{w>m-n}}\limits e^{(1)}_{w,j}e_{i,w}^{(2)}\nonumber\\
&\quad-0-0-\delta(q>m-n,j\leq m-n)\alpha_1e_{p,j}^{(1)}e_{i,q}^{(2)}-\delta(j>m-n,q\leq m-n)e_{p,q}^{(1)}e_{i,j}^{(2)}.\label{2-5}
\end{align}
By a direct computation, we obtain
\begin{align}
&\quad(\sum_{x\leq m-n}\limits e_{x,q}^{(1)}e_{p,x}^{(1)})_{(1)}\sum_{w\leq m-n}\limits e_{w,j}^{(1)}e_{i,w}^{(1)}\nonumber\\
&=\sum_{x,w\leq m-n}\limits [e_{p,x}^{(1)},[e_{x,q}^{(1)},e_{w,j}^{(1)}]]e_{i,w}^{(1)}+\sum_{x,w\leq m-n}\limits e_{w,j}^{(1)}[e_{p,x}^{(1)},[e_{x,q}^{(1)},e_{i,w}^{(1)}]]\nonumber\\
&\quad+\sum_{x,w\leq m-n}\limits [e_{p,x}^{(1)},e_{w,j}^{(1)}][e_{x,q}^{(1)},e_{i,w}^{(1)}]+\sum_{x,w\leq m-n}\limits [e_{x,q}^{(1)},e_{w,j}^{(1)}][e_{p,x}^{(1)},e_{i,w}^{(1)}]\nonumber\\
&\quad+\sum_{x,w\leq m-n}\limits e_{x,q}^{(1)}(e_{p,x}^{(1)},e_{w,j}^{(1)})e_{i,w}^{(1)}+\sum_{x,w\leq m-n}\limits e_{x,q}^{(1)}e_{w,j}^{(1)}(e_{p,x}^{(1)},e_{i,w}^{(1)})\nonumber\\
&\quad+\sum_{x,w\leq m-n}\limits e_{x,q}^{(1)}[[e_{p,x}^{(1)},e_{w,j}^{(1)}],e_{i,w}^{(1)}]\nonumber\\
&\quad+\sum_{x,w\leq m-n}\limits e_{p,x}^{(1)}(e_{x,q}^{(1)},e_{w,j}^{(1)})e_{i,w}^{(1)}+\sum_{x,w\leq m-n}\limits e_{p,x}^{(1)}e_{w,j}^{(1)}(e_{x,q}^{(1)},e_{i,w}^{(1)})\nonumber\\
&\quad+\sum_{x,w\leq m-n}\limits e_{p,x}^{(1)}[[e_{x,q}^{(1)},e_{w,j}^{(1)}],e_{i,w}^{(1)}]\nonumber\\
&\quad+\sum_{\substack{x\leq m-n\\w\leq m-n}}\limits e_{p,x}^{(1)}[-2]([e_{x,q}^{(1)},e_{w,j}^{(1)}],e_{i,w}^{(1)})\nonumber\\
&\quad+\sum_{\substack{x\leq m-n\\w\leq m-n}}\limits e_{x,q}^{(1)}[-2]([e_{p,x}^{(1)},e_{w,j}^{(1)}],e_{i,w}^{(1)})\nonumber\\
&=\sum_{x,w\leq m-n}\limits [e_{p,x}^{(1)},\delta_{q,w}e^{(1)}_{x,j}-\delta_{x,j}e^{(1)}_{w,q}]e_{i,w}^{(1)}+\sum_{x,w\leq m-n}\limits e_{w,j}^{(1)}[e_{p,x}^{(1)},\delta_{q,i}e^{(1)}_{x,w}-\delta_{x,w}e^{(1)}_{i,q}]\nonumber\\
&\quad+\sum_{x,w\leq m-n}\limits (\delta_{x,w}e^{(1)}_{p,j}-\delta_{p,j}e^{(1)}_{w,x})(\delta_{q,i}e^{(1)}_{x,w}-\delta_{w,x}e^{(1)}_{i,q})\nonumber\\
&\quad++\sum_{x,w\leq m-n}\limits (\delta_{w,q}e^{(1)}_{x,j}-\delta_{x,j}e^{(1)}_{w,q})(\delta_{x,i}e^{(1)}_{p,w}-\delta_{p,w}e^{(1)}_{i,x})\nonumber\\
&\quad+\sum_{x,w\leq m-n}\limits e_{x,q}^{(1)}(e_{p,x}^{(1)},e_{w,j}^{(1)})e_{i,w}^{(1)}+\sum_{x,w\leq m-n}\limits e_{x,q}^{(1)}e_{w,j}^{(1)}(e_{p,x}^{(1)},e_{i,w}^{(1)})\nonumber\\
&\quad+\sum_{x,w\leq m-n}\limits e_{x,q}^{(1)}[(\delta_{x,w}e^{(1)}_{p,j}-\delta_{p,j}e^{(1)}_{w,x}),e_{i,w}^{(1)}]\nonumber\\
&\quad+\sum_{x,w\leq m-n}\limits e_{p,x}^{(1)}(e_{x,q}^{(1)},e_{w,j}^{(1)})e_{i,w}^{(1)}+\sum_{x,w\leq m-n}\limits e_{p,x}^{(1)}e_{w,j}^{(1)}(e_{x,q}^{(1)},e_{i,w}^{(1)})\nonumber\\
&\quad+\sum_{x,w\leq m-n}\limits e_{p,x}^{(1)}[\delta_{q,w}e^{(1)}_{x,j}-\delta_{x,j}e^{(1)}_{w,q},e_{i,w}^{(1)}]\nonumber\\
&\quad+\sum_{\substack{x\leq m-n\\w\leq m-n}}\limits e_{p,x}^{(1)}[-2](\delta_{q,w}e^{(1)}_{x,j}-\delta_{x,j}e^{(1)}_{w,q},e_{i,w}^{(1)})\nonumber\\
&\quad+\sum_{\substack{x\leq m-n\\w\leq m-n}}\limits e_{x,q}^{(1)}[-2](\delta_{x,w}e^{(1)}_{p,j}-\delta_{p,j}e^{(1)}_{w,x},e_{i,w}^{(1)})\nonumber\\
&=\sum_{x,w\leq m-n}\limits\delta_{q,w}(e^{(1)}_{p,j}-\delta_{p,j}e^{(1)}_{x,x})e_{i,w}^{(1)}-\sum_{x,w\leq m-n}\limits\delta_{x,j}(\delta_{x,w}e^{(1)}_{p,q}-\delta_{p,q}e^{(1)}_{w,x})e_{i,w}^{(1)}\nonumber\\
&\quad+\sum_{x,w\leq m-n}\limits\delta_{q,i} e_{w,j}^{(1)}(e^{(1)}_{p,w}-\delta_{p,w}e^{(1)}_{x,x})-\sum_{x,w\leq m-n}\limits \delta_{x,w}e_{w,j}^{(1)}(\delta_{x,i}e^{(1)}_{p,q}-\delta_{p,q}e^{(1)}_{i,x})\nonumber\\
&\quad+\sum_{x,w\leq m-n}\limits \delta_{x,w}e^{(1)}_{p,j}(\delta_{q,i}e^{(1)}_{x,w}-\delta_{x,w}e_{i,q})-\sum_{x,w\leq m-n}\limits\delta_{p,j}e^{(1)}_{w,x}(\delta_{q,i}e^{(1)}_{x,w}-\delta_{x,w}e_{i,q})\nonumber\\
&\quad+\sum_{x,w\leq m-n}\limits \delta_{w,q}e^{(1)}_{x,j}(\delta_{x,i}e^{(1)}_{p,w}-\delta_{p,w}e^{(1)}_{i,x})-\sum_{x,w\leq m-n}\limits\delta_{x,j}e^{(1)}_{w,q}(\delta_{x,i}e^{(1)}_{p,w}-\delta_{p,w}e^{(1)}_{i,x})\nonumber\\
&\quad+\sum_{x,w\leq m-n}\limits e_{x,q}^{(1)}[-1](\delta_{x,w}\delta_{p,j}\alpha_1+\delta_{p,x}\delta_{w,j})e_{i,w}^{(1)}+\sum_{x,w\leq m-n}\limits e_{x,q}^{(1)}[-1]e_{w,j}^{(1)}(\delta_{x,i}\delta_{p,w}\alpha_1+\delta_{p,x}\delta_{i,w})\nonumber\\
&\quad+\sum_{x,w\leq m-n}\limits \delta_{x,w}e_{x,q}^{(1)}(\delta_{i,j}e^{(1)}_{p,w}-\delta_{p,w}e^{(1)}_{i,j})-\sum_{x,w\leq m-n}\limits \delta_{p,j}e_{x,q}^{(1)}(\delta_{x,i}e^{(1)}_{w,w}-e^{(1)}_{i,x})\nonumber\\
&\quad+\sum_{x,w\leq m-n}\limits e_{p,x}^{(1)}(\delta_{q,w}\delta_{x,j}\alpha_1+\delta_{x,q}\delta_{w,j})e_{i,w}^{(1)}+\sum_{x,w\leq m-n}\limits e_{p,x}^{(1)}e_{w,j}^{(1)}(\delta_{x,w}\delta_{i,q}\alpha_1+\delta_{x,q}\delta_{i,w})\nonumber\\
&\quad+\sum_{x,w\leq m-n}\limits \delta_{q,w}e_{p,x}^{(1)}(\delta_{j,i}e^{(1)}_{x,w}-\delta_{x,w}e^{(1)}_{i,j})-\sum_{x,w\leq m-n}\limits \delta_{x,j}e_{p,x}^{(1)}(\delta_{q,i}e^{(1)}_{w,w}-e^{(1)}_{i,q})\nonumber\\
&\quad+\sum_{\substack{x\leq m-n\\w\leq m-n}}\limits \delta_{q,w}e_{p,x}^{(1)}[-2](\alpha_1\delta_{i,j}\delta_{x,w}+\delta_{x,j}\delta_{i,w})\nonumber\\
&\quad-\sum_{\substack{x\leq m-n\\w\leq m-n}}\limits \delta_{x,j}e_{p,x}^{(1)}[-2](\delta_{i,q}\alpha_1+\delta_{w,q}\delta_{i,w})\nonumber\\
&\quad+\sum_{\substack{x\leq m-n\\w\leq m-n}}\limits \delta_{x,w}e_{x,q}^{(1)}[-2](\delta_{j,i}\delta_{p,w}\alpha_1+\delta_{p,j}\delta_{i,w})\nonumber\\
&\quad-\sum_{\substack{x\leq m-n\\w\leq m-n}}\limits \delta_{p,j}e_{x,q}^{(1)}[-2](\delta_{i,x}\alpha_1+\delta_{w,i}\delta_{w,x})\nonumber\\
&=\delta(q\leq m-n)(m-n)e^{(1)}_{p,j}e_{i,q}^{(1)}-\delta_{p,j}\delta(q\leq m-n)\sum_{x\leq m-n}\limits e^{(1)}_{x,x}e_{i,q}^{(1)}\nonumber\\
&\quad-\delta(j\leq m-n)e^{(1)}_{p,q}e_{i,j}^{(1)}+\delta_{p,q}\delta(j\leq m-n)\sum_{w\leq m-n}\limits e^{(1)}_{w,j}e_{i,w}^{(1)}\nonumber\\
&\quad+\delta_{q,i}(m-n)\sum_{w\leq m-n}\limits e_{w,j}^{(1)}e^{(1)}_{p,w}-0\nonumber\\
&\quad-0+\delta_{q,p}\sum_{x\leq m-n}\limits e^{(1)}_{x,j}e_{i,x}^{(1)}\nonumber\\
&\quad+\delta_{q,i}\sum_{x\leq m-n}\limits e^{(1)}_{p,j}e^{(1)}_{x,x}-(m-n)e^{(1)}_{p,j}e_{i,q}^{(1)}\nonumber\\
&\quad-\delta_{p,j}\delta_{q,i}\sum_{x,w\leq m-n}\limits e^{(1)}_{w,x}e^{(1)}_{x,w}+\delta_{p,j}\sum_{x\leq m-n}\limits e^{(1)}_{x,x}e_{i,q}^{(1)}+0+0\nonumber\\
&\quad-0-0+\delta_{p,j}\alpha_1\sum_{x\leq m-n}\limits e_{x,q}^{(1)}e_{i,x}^{(1)}+0\nonumber\\
&\quad+0+0+\delta_{i,j}\sum_{x\leq m-n}\limits e_{x,q}^{(1)}e^{(1)}_{p,x}-0\nonumber\\
&\quad-0+\delta_{p,j}(m-n)\sum_{x,w\leq m-n}\limits e_{x,q}^{(1)}e^{(1)}_{i,x}\nonumber\\
&\quad+\delta(q,j\leq m-n)\alpha_1e_{p,j}^{(1)}e_{i,q}^{(1)}+\delta(q,j\leq m-n)e_{p,q}^{(1)}e_{i,j}^{(1)}\nonumber\\
&\quad+\delta_{i,q}\alpha_1\sum_{x\leq m-n}\limits e_{p,x}^{(1)}e_{x,j}^{(1)}+0\nonumber\\
&\quad+\delta_{j,i}\delta(q\leq m-n)\sum_{x\leq m-n}\limits e_{p,x}^{(1)}e^{(1)}_{x,q}-\delta(q\leq m-n)e_{p,q}^{(1)}e^{(1)}_{i,j}\nonumber\\
&\quad-\delta_{q,i}\delta(j\leq m-n)\sum_{w\leq m-n}\limits e_{p,x}^{(1)}e^{(1)}_{w,w}+(m-n)\delta(j\leq m-n)e_{p,j}^{(1)}e^{(1)}_{i,q}\nonumber\\
&\quad+\delta_{i,j}\delta(q\leq m-n)\alpha_1e_{p,q}^{(1)}[-2]+0\nonumber\\
&\quad-\delta_{i,q}\delta(j\leq m-n)\alpha_1(m-n)e_{p,j}^{(1)}[-2]+0\nonumber\\
&\quad+0+0-0-0.
\label{2-6}
\end{align}
By a direct computation, we obtain
\begin{align}
&\quad(\sum_{x\leq m-n}\limits e_{x,q}^{(1)}e_{p,x}^{(1)})_{(1)}\alpha_2e^{(1)}_{i,j}[-2])\nonumber\\
&=\alpha_2\sum_{x\leq m-n}\limits [e_{p,x}^{(1)},[e_{x,q}^{(1)},e^{(1)}_{i,j}]][-2]\nonumber\\
&\quad+\alpha_2\sum_{x\leq m-n}\limits e_{x,q}^{(1)}[e_{p,x}^{(1)},e^{(1)}_{i,j}]+2\alpha_2\sum_{x\leq m-n}\limits e_{x,q}^{(1)}[-2](e_{p,x}^{(1)},e^{(1)}_{i,j})\nonumber\\
&\quad+\alpha_2\sum_{x\leq m-n}\limits e_{p,x}^{(1)}[e_{x,q}^{(1)},e^{(1)}_{i,j}]+2\alpha_2\sum_{x\leq m-n}\limits e_{p,x}^{(1)}[-2](e_{x,q}^{(1)},e^{(1)}_{i,j})\nonumber\\
&=\alpha_2\sum_{x\leq m-n}\limits [e_{p,x}^{(1)},\delta_{i,q}e^{(1)}_{x,j}-\delta_{j,x}e^{(1)}_{i,q}][-2]\nonumber\\
&\quad+\alpha_2\sum_{x\leq m-n}\limits e_{x,q}^{(1)}(\delta_{x,i}e^{(1)}_{p,j}-\delta_{p,j}e^{(1)}_{i,x})+2\alpha_2\sum_{x\leq m-n}\limits e_{x,q}^{(1)}[-2](\delta_{p,j}\delta_{i,x}\alpha_1+\delta_{p,x}\delta_{i,j})\nonumber\\
&\quad+\alpha_2\sum_{x\leq m-n}\limits e_{p,x}^{(1)}(\delta_{q,i}e^{(1)}_{x,j}-\delta_{x,j}e^{(1)}_{x,j})+2\alpha_2\sum_{x\leq m-n}\limits e_{p,x}^{(1)}[-2](\delta_{x,j}\delta_{q,i}\alpha_1+\delta_{q,x}\delta_{i,j})\nonumber\\
&=\alpha_2\sum_{x\leq m-n}\limits\delta_{i,q}(e^{(1)}_{p,j}-\delta_{p,j}e^{(1)}_{x,x})[-2]-\alpha_2\sum_{x\leq m-n}\limits\delta_{j,x}(\delta_{i,x}e^{(1)}_{p,q}-\delta_{p,q}e^{(1)}_{i,x})[-2]\nonumber\\
&\quad+\alpha_2\sum_{x\leq m-n}\limits e_{x,q}^{(1)}(\delta_{x,i}e^{(1)}_{p,j}-2\delta_{p,j}e^{(1)}_{i,x})+2\alpha_2\sum_{x\leq m-n}\limits e_{x,q}^{(1)}[-2](\delta_{p,j}\delta_{i,x}\alpha_1+\delta_{p,x}\delta_{i,j})\nonumber\\
&\quad+\alpha_2\sum_{x\leq m-n}\limits e_{p,x}^{(1)}(\delta_{q,i}e^{(1)}_{x,j}-\delta_{x,j}e^{(1)}_{i,q})+2\alpha_2\sum_{x\leq m-n}\limits e_{p,x}^{(1)}[-2](\delta_{x,j}\delta_{q,i}\alpha_1+\delta_{q,x}\delta_{i,j})\nonumber\\
&=\delta_{i,q}\alpha_2(m-n)e^{(1)}_{p,j}[-2]-\delta_{i,q}\delta_{p,j}\alpha_2\sum_{x\leq m-n}\limits e^{(1)}_{x,x}[-2]\nonumber\\
&\quad+0+\delta_{p,q}\delta(j\leq m-n)\alpha_2e^{(1)}_{i,j}[-2]\nonumber\\
&\quad-0-\delta_{p,j}\alpha_2\sum_{x\leq m-n}\limits e_{x,q}^{(1)}e^{(1)}_{i,x}\nonumber\\
&\quad-0-0+\delta_{q,i}\alpha_2\sum_{x\leq m-n}\limits e_{p,x}^{(1)}e^{(1)}_{x,j}-\delta(j\leq m-n)\alpha_2e_{p,j}^{(1)}e^{(1)}_{i,q}\nonumber\\
&\quad+2\delta(j\leq m-n)\alpha_2e_{p,j}^{(1)}[-2]\delta_{q,i}\alpha_1+2\delta_{i,j}\delta(q\leq m-n)\alpha_2 e_{p,q}^{(1)}[-2].\label{2-7}
\end{align}
First, we compute the sum of the terms containing $e_{p,j}^{(a)}e^{(b)}_{i,q}$ or $e^{(b)}_{i,q}e_{p,j}^{(a)}$.
Since
\begin{align*}
&\quad\eqref{2-3}_1+\eqref{2-3}_5+\eqref{2-4}_2\\
&=-\delta(q\leq m-n)\delta(j>m-n)\alpha_1e_{i,q}^{(1)}e_{p,j}^{(2)}[-1]-\delta(j>m-n)(m-n)e^{(1)}_{i,q}e_{p,j}^{(2)}\\
&\quad+\alpha_2\delta(j>m-n)e_{p,j}^{(2)}[-1]e^{(1)}_{i,q}\\
&=\delta(q,j>m-n)\alpha_1e^{(2)}_{p,j}e^{(1)}_{i,q},\\
&\quad\eqref{2-1}_1+\eqref{2-5}_1+\eqref{2-5}_4\\
&=\delta(q>m-n)\alpha_2e_{p,j}^{(1)}e^{(2)}_{i,q}-\delta(q>m-n)(m-n)e^{(1)}_{p,j}e_{i,q}^{(2)}\\
&\quad-\delta(q>m-n,j\leq m-n)e^{(1)}_{p,j}e^{(2)}_{i,q}\\
&=\delta(q,j>m-n)\alpha_1e^{(1)}_{p,j}e^{(2)}_{i,q},\\
&\quad\eqref{2-1}_3+\eqref{2-2}_5+\eqref{2-6}_1+\eqref{2-6}_8+\eqref{2-6}_{14}+\eqref{2-6}_{20}+\eqref{2-7}_6\\
&=-\delta(q\leq m-n)\alpha_2e^{(1)}_{p,j}e_{i,q}^{(1)}+\alpha_2e_{i,q}^{(1)}e_{p,j}^{(1)}\\
&\quad+\delta(q>m-n)(m-n)e^{(1)}_{p,j}e_{i,q}^{(1)}-(m-n)e^{(1)}_{p,j}e^{(1)}_{i,q}\\
&\quad+\delta(q,j\leq m-n)\alpha_1e_{p,j}^{(1)}[-1]e_{i,q}^{(1)}+\delta(j\leq m-n)(m-n)e^{(1)}_{p,j}e^{(1)}_{i,q}-\delta(j\leq m-n)\alpha_2e^{(1)}_{p,j}e^{(1)}_{i,q}\\
&=\delta(q,j>m-n)\alpha_1e^{(1)}_{p,j}e^{(1)}_{i,q}+\alpha_2[e_{i,q}^{(1)},e_{p,j}^{(1)}][-2]
\end{align*}
hold, we have
\begin{align}
&\quad\eqref{2-1}_1+\eqref{2-3}_1+\eqref{2-3}_5+\eqref{2-4}_2\nonumber\\
&\qquad+\eqref{2-1}_1+\eqref{2-5}_1+\eqref{2-5}_4+\eqref{2-1}_3+\eqref{2-2}_5+\eqref{2-6}_1+\eqref{2-6}_8+\eqref{2-6}_{14}+\eqref{2-6}_{20}+\eqref{2-7}_6\nonumber\\
&=\delta(q,j>m-n)\alpha_1e^{(2)}_{p,j}e^{(2)}_{i,q}+\delta(q,j>m-n)\alpha_1e^{(2)}_{p,j}e^{(1)}_{i,q}+\alpha_1\delta(q,j>m-n)e^{(1)}_{p,j}e^{(2)}_{i,q}\nonumber\\
&\quad+\delta(q,j>m-n)\alpha_1e^{(1)}_{p,j}e^{(1)}_{i,q}+\alpha_2(\delta_{p,q}e^{(1)}_{i,j}-\delta_{i,j}e^{(1)}_{i,j})[-2]\nonumber\\
&=\delta(q,j>m-n)\alpha_1W^{(1)}_{p,j}W^{(1)}_{i,q}+\delta_{p,q}\alpha_2e^{(1)}_{i,j}[-2]-\delta_{i,j}\alpha_2e^{(1)}_{i,j}[-2].\label{2-1-1}
\end{align}

Next, we compute the sum of the terms containing $e_{p,q}^{(a)}e^{(b)}_{i,j}$ or $e^{(b)}_{i,j}e_{p,q}^{(a)}$. Since
\begin{align*}
&\quad\eqref{2-2}_1+\eqref{2-3}_2\\
&=\delta(q>m-n)e^{(1)}_{i,j}e^{(2)}_{p,q}-\delta(j\leq m-n)\delta(q>m-n)e_{i,j}^{(1)}e_{p,q}^{(2)}[-1]\\
&=\delta(q,j>m-n)e^{(1)}_{i,j}e^{(2)}_{p,q},\\
&\quad\eqref{2-2}_4+\eqref{2-5}_5\\
&=\delta(j>m-n)e^{(1)}_{p,q}e^{(2)}_{i,j}-\delta(j>m-n,q\leq m-n)e^{(1)}_{p,q}e^{(2)}_{i,j}\\
&=\delta(q,j>m-n)e^{(1)}_{p,q}e^{(2)}_{i,j},\\
&\quad\eqref{2-2}_5+\eqref{2-6}_3+\eqref{2-6}_{15}+\eqref{2-6}_{18}\\
&=e_{p,q}^{(1)}e_{i,j}^{(1)}-\delta(j\leq m-n)e^{(1)}_{p,q}e_{i,j}^{(1)}\\
&\quad+\delta(q,j\leq m-n)e_{p,q}^{(1)}e_{i,j}^{(1)}-\delta(q\leq m-n)e_{p,q}^{(1)}e^{(1)}_{i,j}\\
&=\delta(q,j>m-n)e_{p,q}^{(1)}e^{(1)}_{i,j}.
\end{align*}
hold, we have
\begin{align*}
&\quad\eqref{2-2}_8+\eqref{2-2}_1+\eqref{2-3}_2+\eqref{2-2}_4+\eqref{2-5}_5+\eqref{2-2}_5+\eqref{2-6}_3+\eqref{2-6}_{15}+\eqref{2-6}_{18}\\
&=\delta(q,j>m-n)W_{p,q}^{(1)}W^{(1)}_{i,j}.
\end{align*}

Next, we compute the sum of terms containing $\delta_{i,q}$.
Since
\begin{align*}
&\quad\eqref{2-3}_4+\eqref{2-6}_7+\eqref{2-6}_{19}\\
&=\delta_{q,i}\delta(j>m-n)\sum_{\substack{w\leq m-n}}\limits e^{(1)}_{w,w}e_{p,j}^{(2)}+\delta_{q,i}\sum_{\substack{x\leq m-n}}\limits e_{p,j}^{(1)}e^{(1)}_{x,x}\\
&\quad-\delta_{q,i}\delta(j\leq m-n)\sum_{\substack{w\leq m-n}}\limits e_{p,j}^{(1)}e^{(1)}_{w,w}\\
&=\delta_{q,i}\delta(j>m-n)\sum_{\substack{w\leq m-n}}\limits e^{(1)}_{w,w}e_{p,j}^{(2)}\\
&\quad+\delta_{q,i}\sum_{\substack{x\leq m-n}}\limits e^{(1)}_{x,x}e_{p,j}^{(1)}+\delta_{q,i}\sum_{\substack{x\leq m-n}}\limits [e_{p,j}^{(1)},e^{(1)}_{x,x}]\\
&\quad-\delta_{q,i}\delta(j\leq m-n)\sum_{\substack{w\leq m-n}}\limits e^{(1)}_{w,w}e_{p,j}^{(1)}-\delta_{q,i}\delta(j\leq m-n)\sum_{\substack{w\leq m-n}}\limits [e_{p,j}^{(1)},e^{(1)}_{w,w}]\\
&=\delta_{q,i}\delta(j>m-n)\sum_{\substack{w\leq m-n}}\limits e^{(1)}_{w,w}e_{p,j}^{(2)}\\
&\quad+\delta_{q,i}\sum_{\substack{x\leq m-n}}\limits e^{(1)}_{x,x}e_{p,j}^{(1)}+\delta_{q,i}\delta(j\leq m-n) e_{p,j}^{(1)}[-2]\\
&\quad-\delta_{q,i}\delta(j\leq m-n)\sum_{\substack{w\leq m-n}}\limits e^{(1)}_{w,w}e_{p,j}^{(1)}-\delta_{q,i}\delta(j\leq m-n) e_{p,j}^{(1)}[-2]\\
&=\delta_{q,i}\delta(j>m-n)\sum_{\substack{w\leq m-n}}\limits e_{w,w}^{(1)}(e^{(1)}_{p,j}+e^{(2)}_{p,j}),\\
&\quad\eqref{2-1}_5+\eqref{2-4}_1+\eqref{2-6}_5+\eqref{2-6}_{16}+\eqref{2-7}_5\\
&=-\delta_{q,i}\alpha_2\sum_{w\leq m-n}\limits e_{w,j}^{(1)}e^{(1)}_{p,w}\\
&\quad-\delta_{q,i}\alpha_2\sum_{x>m-n}\limits e_{p,x}^{(2)}[-1]e^{(1)}_{x,j}+\delta_{i,q}(m-n)\sum_{w\leq m-n}\limits e^{(1)}_{w,j}e^{(1)}_{p,w}\\
&\quad+\delta_{i,q}\alpha_1\sum_{x\leq m-n}\limits e^{(1)}_{p,x}e^{(1)}_{x,j}+\delta_{i,q}\alpha_2\sum_{x\leq m-n}\limits e^{(1)}_{p,x}e^{(1)}_{x,j}\\
&=-\delta_{q,i}\alpha_2\sum_{w\leq m-n}\limits e_{w,j}^{(1)}e^{(1)}_{p,w}\\
&\quad-\delta_{q,i}\alpha_2\sum_{x>m-n}\limits e_{p,x}^{(2)}e^{(1)}_{x,j}+\delta_{i,q}(m-n)\sum_{w\leq m-n}\limits e^{(1)}_{w,j}e^{(1)}_{p,w}\\
&\quad+\delta_{i,q}\alpha_1\sum_{x\leq m-n}\limits e^{(1)}_{x,j}e^{(1)}_{p,x}+\delta_{i,q}\alpha_1\sum_{x\leq m-n}\limits [e^{(1)}_{p,x},e^{(1)}_{x,j}][-2]\\
&\quad+\delta_{i,q}\alpha_2\sum_{x\leq m-n}\limits e^{(1)}_{x,j}e^{(1)}_{p,x}+\delta_{i,q}\alpha_2\sum_{x\leq m-n}\limits [e^{(1)}_{p,x},e^{(1)}_{x,j}][-2]\\
&=-\delta_{i,q}\alpha_2(W^{(2)}_{p,j}+e^{(1)}_{p,j}[-2])+\delta_{i,q}\alpha_1\sum_{x\leq m-n}\limits [e^{(1)}_{p,x},e^{(1)}_{x,j}][-2]\\
&\quad+\delta_{i,q}\alpha_2\sum_{x\leq m-n}\limits [e^{(1)}_{p,x},e^{(1)}_{x,j}][-2]
\end{align*}
hold, we have
\begin{align}
&\quad\eqref{2-1}_5+\eqref{2-3}_4+\eqref{2-4}_1+\eqref{2-6}_5+\eqref{2-6}_7+\eqref{2-6}_{16}+\eqref{2-6}_{19}+\eqref{2-7}_5\nonumber\\
&=\delta_{q,i}\delta(j>m-n)\sum_{\substack{w\leq m-n}}\limits W_{w,w}^{(1)}W^{(1)}_{p,j}-\delta_{i,q}\alpha_2W^{(2)}_{p,j}\nonumber\\
&\quad-\delta_{i,q}\alpha_2^2e^{(1)}_{p,j}[-2]+\delta_{i,q}\alpha_1\sum_{x\leq m-n}\limits(e^{(1)}_{p,j}-\delta_{p,j}e^{(1)}_{x,x})[-2]\nonumber\\
&\quad+\delta_{i,q}\alpha_2\sum_{x\leq m-n}\limits(e^{(1)}_{p,j}-\delta_{p,j}e^{(1)}_{x,x})[-2]\nonumber\\
&=\delta_{q,i}\delta(j>m-n)\sum_{\substack{w\leq m-n}}\limits W_{w,w}^{(1)}W^{(1)}_{p,j}-\delta_{i,q}\alpha_2W^{(2)}_{p,j}\nonumber\\
&\quad-\delta_{i,q}\alpha_2^2e^{(1)}_{p,j}[-2]-\delta_{i,q}(\alpha_1+\alpha_2)(m-n)e^{(1)}_{p,j}[-2]\nonumber\\
&\quad-\delta_{i,q}\delta_{p,j}(\alpha_1+\alpha_2)\sum_{x\leq m-n}\limits e^{(1)}_{x,x}[-2].\label{2-1-2}
\end{align}
Since
\begin{align*}
&\quad\eqref{2-3}_6+\eqref{2-4}_3\\
&=\delta_{i,q}\delta(j>m-n)\alpha_1(m-n)e^{(2)}_{p,j}[-2]-2\delta_{q,i}\delta(j>m-n)\alpha_1\alpha_2e^{(2)}_{p,j}[-2]\\
&=-\delta_{q,i}\delta(j>m-n)\alpha_1(\alpha_1+\alpha_2)e^{(2)}_{p,j}[-2],\\
&\quad\eqref{2-1}_6+\eqref{2-6}_{22}+\eqref{2-7}_1+\eqref{2-7}_7+\eqref{2-1-2}_3+\eqref{2-1-2}_4\\
&=-\delta_{q,i}\alpha_2^2e^{(1)}_{p,j}[-2]-\delta(j\leq m-n)\alpha_1(m-n)e^{(1)}_{p,q}[-2]+\delta_{i,q}\alpha_2(m-n)e^{(1)}_{p,j}[-2]\\
&\quad+2\delta(j\leq m-n)\alpha_1\alpha_2e^{(1)}_{p,j}[-2]-\delta_{i,q}\alpha_2^2e^{(1)}_{p,j}[-2]+\delta_{i,q}(\alpha_1+\alpha_2)(m-n)e^{(1)}_{p,j}[-2]\nonumber\\
&=-\delta_{q,i}\delta(j>m-n)\alpha_2(\alpha_1+\alpha_2)e^{(1)}_{p,j}[-2]
\end{align*}
and
\begin{align*}
\eqref{2-1-2}_5&=-\delta_{i,q}(\alpha_1+\alpha_2)\sum_{x\leq m-n}\limits \delta_{p,j}\partial W^{(1)}_{x,x}\\
&=-\delta_{i,q}\delta_{p,j}(\alpha_1+\alpha_2)\sum_{x\leq m-n}\limits \partial W^{(1)}_{x,x}
\end{align*}
hold, we obtain
\begin{align*}
&\quad\eqref{2-1}_6+\eqref{2-3}_6+\eqref{2-4}_3+\eqref{2-6}_{22}+\eqref{2-7}_1+\eqref{2-7}_7+\eqref{2-1-2}_3+\eqref{2-1-2}_4+\eqref{2-1-2}_5\\
&=-\delta_{q,i}\delta(j>m-n)\alpha_1(\alpha_1+\alpha_2)\partial W^{(1)}_{p,j}-\delta_{i,q}\delta_{p,j}(\alpha_1+\alpha_2)\sum_{x\leq m-n}\limits \partial W^{(1)}_{x,x}.
\end{align*}

Next, we compute the sum of terms containing $\delta_{p,j}$.
Since
\begin{align*}
&\quad\eqref{2-5}_2+\eqref{2-6}_2+\eqref{2-6}_{10}\\
&=\delta_{p,j}\delta(q>m-n)\sum_{x\leq m-n}\limits e^{(1)}_{x,x}e^{(2)}_{i,q}-\delta_{p,j}\delta(q\leq m-n)\sum_{x\leq m-n}\limits e^{(1)}_{x,x}e^{(1)}_{i,q}\\
&\quad+\delta_{p,j}\sum_{x\leq m-n}\limits e^{(1)}_{x,x}e^{(1)}_{i,q}\\
&=\delta_{p,j}\delta(q>m-n)\sum_{x\leq m-n}\limits e^{(1)}_{x,x}(e^{(1)}_{i,q}+e^{(2)}_{i,q})\\
&=\delta_{p,j}\delta(q>m-n)\sum_{x\leq m-n}\limits W^{(1)}_{x,x}W^{(1)}_{i,q},\\
&\quad\eqref{2-1}_2+\eqref{2-1}_4+\eqref{2-6}_{11}+\eqref{2-6}_{13}+\eqref{2-7}_4\\
&=-\delta_{p,j}\alpha_2\sum_{w>m-n}\limits e^{(1)}_{w,q}e^{(2)}_{i,w}+\delta_{p,j}\alpha_2\sum_{w\leq m-n}\limits e^{(1)}_{w,q}e_{i,w}^{(1)}\\
&\quad+\delta_{p,j}\alpha_1\sum_{x\leq m-n}\limits e^{(1)}_{x,q}e^{(1)}_{i,x}\\
&\quad+\delta_{p,j}(m-n)\sum_{x\leq m-n}\limits e^{(1)}_{x,q}e^{(1)}_{i,x}-\delta_{p,j}\alpha_2\sum_{x\leq m-n}\limits e^{(1)}_{x,q}e^{(1)}_{i,x}\\
&=-\delta_{p,j}\alpha_2(W^{(2)}_{i,q}+\alpha_2e^{(1)}_{i,q}[-2])
\end{align*}
hold, we have
\begin{align}
&\quad\eqref{2-1}_2+\eqref{2-1}_4+\eqref{2-5}_2+\eqref{2-6}_2+\eqref{2-6}_{10}+\eqref{2-6}_{11}+\eqref{2-6}_{13}+\eqref{2-7}_4\nonumber\\
&=\delta_{p,j}\delta(q>m-n)\sum_{x\leq m-n}\limits W^{(1)}_{x,x}W^{(1)}_{i,q}-\delta_{p,j}\alpha_2W^{(2)}_{i,q}-\delta_{p,j}\alpha_2^2e^{(1)}_{i,q}[-2].\label{2-1-3}
\end{align}
By a direct computation, we obtain
\begin{align*}
&\quad\eqref{2-1}_7+\eqref{2-1-3}_2\\
&=\alpha_2^2\delta_{p,j}e^{(1)}_{i,q}[-2]-\alpha_2^2\delta_{p,j}e^{(1)}_{i,q}[-2]=0.
\end{align*}
Next, we compute the sum of terms containing $\delta_{i,j}$. By a direct computation, we obtain
\begin{align}
&\quad\eqref{2-2}_3+\eqref{2-6}_{12}\nonumber\\
&=-\sum_{w>m-n}\limits \delta_{i,j}e^{(1)}_{w,q}e^{(2)}_{p,w}+\sum_{x\leq m-n}\limits \delta_{i,j}e_{x,q}^{(1)}e^{(1)}_{p,x}\nonumber\\
&=-\delta_{i,j}(W^{(2)}_{p,q}+\alpha_2e^{(1)}_{p,q}[-2])\nonumber\\
&=-\delta_{i,j}W^{(2)}_{p,q}-\delta_{i,j}\alpha_2e^{(1)}_{p,q}[-2]\label{2-1-4}
\end{align}
and
\begin{align}
&\quad\eqref{2-3}_3+\eqref{2-6}_{17}\nonumber\\
&=-\delta_{i,j}\delta(q\leq m-n)\sum_{x>m-n}\limits e^{(1)}_{x,q}e^{(2)}_{p,x}+\delta_{i,j}\delta(q\leq m-n)\sum_{x>m-n}\limits e^{(1)}_{p,x}e^{(1)}_{x,q}\nonumber\\
&=-\delta_{i,j}\delta(q\leq m-n)\sum_{x>m-n}\limits e^{(1)}_{x,q}e^{(2)}_{p,x}+\delta_{i,j}\delta(q\leq m-n)\sum_{x>m-n}\limits e^{(1)}_{x,q}e^{(1)}_{p,x}\nonumber\\
&\quad+\delta_{i,j}\delta(q\leq m-n)\sum_{x>m-n}\limits [e^{(1)}_{p,x},e^{(1)}_{x,q}][-2]\nonumber\\
&=-\delta_{i,j}\delta(q\leq m-n)W^{(2)}_{p,q}-\delta_{i,j}\delta(q\leq m-n)\alpha_2 e^{(1)}_{p,q}[-2]+\delta_{i,j}\delta(q\leq m-n)(m-n)e^{(1)}_{p,q}[-2].\label{2-1-4.5}
\end{align}
By a direct computation, we obtain
\begin{align*}
&\quad\eqref{2-4}_4+\eqref{2-6}_{21}+\eqref{2-7}_8+\eqref{2-1-1}_3+\eqref{2-1-4}_2+\eqref{2-1-4.5}_2+\eqref{2-1-4.5}_3\\
&=-2\delta_{i,j}\delta(q>m-n)\alpha_2e_{p,q}^{(2)}[-2]+\delta_{i,j}\delta(q\leq m-n)\alpha_1e_{p,q}^{(1)}[-2]+2\delta(q\leq m-n)\alpha_2\delta_{i,j}e^{(1)}_{p,q}[-2]\\
&\quad-\delta_{i,j}\alpha_2e^{(1)}_{p,q}[-2]-\delta_{i,j}e^{(1)}_{p,q}[-2]-\delta_{i,j}\delta(q\leq m-n)\alpha_2 e^{(1)}_{p,q}[-2]+\delta_{i,j}\delta(q\leq m-n)(m-n)e^{(1)}_{p,q}[-2]\\
&=-2\alpha_2\delta(q>m-n)\partial W_{p,q}^{(1)}\delta_{i,j}.
\end{align*}

Next, we compute the sum of terms containing $\delta_{p,q}$.
\begin{align}
&\quad\eqref{2-2}_2+\eqref{2-5}_3+\eqref{2-6}_4+\eqref{2-6}_6\nonumber\\
&=-\sum_{x>m-n}\limits \delta_{q,p}e^{(1)}_{x,j}e^{(2)}_{i,x}-\delta_{p,q}\delta(j\leq m-n)\sum_{w>m-n}\limits e^{(1)}_{w,j}e^{(2)}_{i,w}\nonumber\\
&\quad-\delta_{p,q}\delta(j\leq m-n)\sum_{w\leq m-n}\limits e^{(1)}_{w,j}e^{(1)}_{i,w}+\sum_{x\leq m-n}\limits \delta_{q,p}e^{(1)}_{x,j}e_{i,x}^{(1)}\nonumber\\
&=-\delta_{p,q}(1+\delta(j\leq m-n))(W^{(2)}_{i,j}+\alpha_2e^{(1)}_{i,j}[-2])\nonumber\\
&=-\delta_{p,q}(1+\delta(j\leq m-n))W^{(2)}_{i,j}-\delta_{p,q}\alpha_2(1+\delta(j\leq m-n))e^{(1)}_{i,j}[-2].\label{2-1-7}
\end{align}
By a direct computation, we obtain
\begin{align*}
&\quad\eqref{2-7}_3+\eqref{2-1-1}_2+\eqref{2-1-7}_1\\
&=\delta(j\leq m-n)\alpha_2\delta_{p,q}e^{(1)}_{i,j}+\alpha_2\delta_{p,q}e^{(1)}_{i,j}[-2]\\
&\quad-\alpha_2\delta_{p,q}(1+\delta(j\leq m-n))e^{(1)}_{i,j}[-2]\\
&=0.
\end{align*}
Finally, we compute the sum of the remaining terms. By a direct computation, we obtain
\begin{align*}
&\quad\eqref{2-6}_9+\eqref{2-7}_2\\
&=-\sum_{x,w\leq m-n}\limits \delta_{p,j}\delta_{q,i}e^{(1)}_{w,x}e^{(1)}_{x,w}-\alpha_2\sum_{x\leq m-n}\limits\delta_{p,j}\delta_{q,i}e^{(1)}_{x,x}[-2]\\
&=-\sum_{x,w\leq m-n}\limits \delta_{p,j}\delta_{q,i}W^{(1)}_{w,x}W^{(1)}_{x,w}-\alpha_2\sum_{x\leq m-n}\limits\delta_{p,j}\delta_{q,i}\partial W^{(1)}_{x,x}.
\end{align*}
We complete the proof of \eqref{OPE3-2}.
\section*{Acknowledgement}

The author wishes to express his gratitude to Daniele Valeri. This article is inspired by his lecture at "Quantum symmetries: Tensor categories, Topological quantum field theories, Vertex algebras" held at the University of Montreal. The author is also grateful to Thomas Creutzig for proposing this problem. The author expresses his sincere thanks to Shigenori Nakatsuka for helping me to compute the OPEs and giving me lots of advice and comments.

\section*{Data Availability}
The authors confirm that the data supporting the findings of this study are available within the article and its supplementary materials.

\section*{Competing Interests and Funding}
The authors declare that they have no conflict of interest.
\bibliographystyle{plain}
\bibliography{syuu}
\end{document}